\documentclass[reqno]{amsart}
\usepackage{mathrsfs}
\usepackage{color}
\usepackage{amsmath}
\usepackage{amsfonts}
\usepackage{amssymb}
\usepackage{graphicx}
\usepackage{hyperref}

\newtheorem{theorem}{Theorem}[section]

\newtheorem{lemma}[theorem]{Lemma}
\newtheorem{proposition}[theorem]{Proposition}
\newtheorem{problem}[theorem]{Problem}

\newtheorem{conjecture}[theorem]{Conjecture}

\newtheorem{remark}[theorem]{Remark}

\numberwithin{equation}{section}

\begin{document}
	
	\title[Concavity property of minimal $L^2$ integrals]
	{Concavity property of minimal $L^2$ integrals with Lebesgue measurable gain \uppercase\expandafter{\romannumeral5}--fibrations over open Riemann surfaces}

	\author{Shijie Bao}
        \address{Shijie Bao: Institute of Mathematics, Academy of Mathematics
        and Systems Science, Chinese Academy of Sciences, Beijing 100190, China.}
        \email{bsjie@amss.ac.cn}
	
	\author{Qi'an Guan}
	\address{Qi'an Guan: School of
		Mathematical Sciences, Peking University, Beijing 100871, China.}
	\email{guanqian@math.pku.edu.cn}
	
	\author{Zheng Yuan}
	\address{Zheng Yuan: School of
		Mathematical Sciences, Peking University, Beijing 100871, China.}
	\email{zyuan@pku.edu.cn}
	
	\thanks{}
	
	
	
	\date{\today}
	
	\dedicatory{}
	
	\commby{}
	
	
	\begin{abstract}
In this article, we present characterizations of the concavity property of minimal $L^2$ integrals degenerating to linearity in the case of fibrations over open Riemann surfaces.
As applications, we obtain characterizations of the holding of equality in optimal jets $L^2$ extension problem from fibers over analytic subsets to fibrations over open Riemann surfaces,
which implies characterizations of the fibration versions of the equality parts of Suita conjecture and extended Suita conjecture.
	\end{abstract}
	
	\maketitle

	\section{Introduction}
The strong openness property of multiplier ideal sheaves \cite{GZSOC} i.e. $\mathcal{I}(\varphi)=\mathcal{I}_+(\varphi):=\mathop{\cup} \limits_{\epsilon>0}\mathcal{I}((1+\epsilon)\varphi)$
	(conjectured by Demailly \cite{DemaillySoc}) has opened the door to new types of approximation techniques
	in several complex variables, complex algebraic geometry and complex differential geometry
	(see e.g. \cite{GZSOC,K16,cao17,cdM17,FoW18,DEL18,ZZ2018,GZ20,berndtsson20,ZZ2019,ZhouZhu20siu's,FoW20,KS20,DEL21}),
	where $\varphi$ is a plurisubharmonic function on a complex manifold $M$ (see \cite{Demaillybook}), and multiplier ideal sheaf $\mathcal{I}(\varphi)$ is the sheaf of germs of holomorphic functions $f$ such that $|f|^2e^{-\varphi}$ is locally integrable (see e.g. \cite{Tian,Nadel,Siu96,DEL,DK01,DemaillySoc,DP03,Lazarsfeld,Siu05,Siu09,DemaillyAG,Guenancia}).

	When $\mathcal{I}(\varphi)=\mathcal{O}$, the strong openness property degenerates to the openness property (conjectured by Demailly-Koll\'ar \cite{DK01}).
	Berndtsson \cite{Berndtsson2} proved the openness property by establishing an effectiveness result of the openness property.
	Stimulated by Berndtsson's effectiveness result, and continuing the solution of the strong openness property \cite{GZSOC},
	Guan-Zhou \cite{GZeff} established an effectiveness result of the strong openness property by considering the minimal $L^{2}$ integral on the pseudoconvex domain $D$.

	Considering the minimal $L^{2}$ integrals on the sublevel sets of the weight $\varphi$,
	Guan \cite{G16} obtained a sharp version of Guan-Zhou's effectiveness result,
	and established a concavity property of the minimal $L^2$ integrals on the sublevel sets of the weight $\varphi$ (with constant gain).
	The concavity property deduces a proof of Saitoh's conjecture for conjugate Hardy $H^2$ kernels \cite{Guan2019},
	and the sufficient and necessary condition of the existence of decreasing equisingular approximations with analytic singularities for the multiplier ideal sheaves with weights $\log(|z_{1}|^{a_{1}}+\cdots+|z_{n}|^{a_{n}})$ \cite{guan-20}.
	
	For smooth gain, Guan \cite{G2018} (see also \cite{GM}) obtained the concavity property on Stein manifolds (weakly pseudoconvex K\"{a}hler case was obtained by Guan-Mi\cite{GM_Sci}).
    The concavity property deduces an optimal support function related to the strong openness property (obtained by Guan-Yuan) \cite{GY-support}, and an effectiveness result of the strong openness property in $L^p$ (obtained by Guan-Yuan) \cite{GY-lp-effe}.
	For Lebesgue measurable gain, Guan-Yuan \cite{GY-concavity} obtained the concavity property on Stein manifolds (weakly pseudoconvex K\"{a}hler case was obtained by Guan-Mi-Yuan \cite{GMY}),
	which deduces a twisted $L^p$ version of strong openness property \cite{GY-twisted}.
	
	Note that the linearity is a degenerate concavity. A natural problem was posed in \cite{GY-concavity3}:
	
	\begin{problem}[\cite{GY-concavity3}]\label{Q:chara}
		How to characterize the concavity property degenerating to linearity?
	\end{problem}
Recall that for 1-dim case, Guan-Yuan \cite{GY-concavity} gave an answer to Problem \ref{Q:chara} for single point, i.e. for weights may not be subharmonic (the case of subharmonic weights was answered by Guan-Mi \cite{GM}),
and Guan-Yuan \cite{GY-concavity3} gave an answer to Problem \ref{Q:chara} for finite points.
For the case of products of open Riemann surfaces, Guan-Yuan \cite{GY-concavity4} gave an answer to Problem \ref{Q:chara} for product of finite points.
	
In this article, we give answers to Problem \ref{Q:chara} for the case of fibrations over open Riemann surfaces.
	
	Let $\Omega$ be an open Riemann surface, which admits a nontrivial Green function $G_{\Omega}$, and let $K_{\Omega}$ be the canonical (holomorphic) line bundle on $\Omega$. Let $Y$ be an $n-1$ dimensional weakly pseudoconvex K\"{a}hler manifold. Let $M=\Omega\times Y$ be a complex manifold, and $K_M$ be the canonical line bundle on $M$. Let $\pi_1$, $\pi_2$ be the natural projections from $M$ to $\Omega$ and $Y$.
	
	Let $Z^1_0$ be a (closed) analytic subset of $\Omega$, denote $Z_0:=\pi_1^{-1}(Z^1_0)$ be an analytic subset of $M$. Let $\psi_1$ be a negative subharmonic function on $\Omega$ such that $\psi_1(z)=-\infty$ for any $z\in Z^1_0$, and let $\varphi_1$ be a Lebesgue measurable function on $\Omega$ such that $\varphi_1+\psi_1$ is subharmonic on $\Omega$. Let $\varphi_2$ be a plurisubharmonic function on $Y$. Denote that $\psi:=\pi_1^*(\psi_1)$, $\varphi=\pi_1^*(\varphi_1)+\pi_2^*(\varphi_2)$. Let $c$ be a positive function on $(0,+\infty)$ such that $\int_0^{+\infty}c(t)e^{-t}dt<+\infty$, $c(t)e^{-t}$ is decreasing on $(0,+\infty)$ and $e^{-\varphi}c(-\psi)$ has a positive lower bound on any compact subset of $M\setminus Z_0$. Let $f$ be a holomorphic $(n,0)$ form on a neighborhood of $Z_0$. Denote
	\begin{flalign*}
		\begin{split}
			\inf\bigg\{\int_{\{\psi<-t\}}|\tilde{f}|^2e^{-\varphi}c(-\psi) :& (\tilde{f}-f,(z,y))\in (\mathcal{O}(K_M))_{(z,y)}\otimes\mathcal{I}(\varphi+\psi)_{(z,y)} \\
			&\text{\ for \ any } (z,y)\in Z_0  \\
			& \& \ \tilde{f}\in H^0(\{\psi<-t\},\mathcal{O}(K_M))\bigg\}
		\end{split}
	\end{flalign*}
	by $G(t;c)$ for any $t\in [0,+\infty)$. Here $|f|^2:=\sqrt{-1}^{n^2}f\wedge\bar{f}$ for any $(n,0)$ form $f$.  We may denote $G(t;c)$ by $G(t)$ if there are no misunderstandings.
	
	Recall that $G(h^{-1}(r))$ is concave with respect to $r$ (\cite{GMY}), where $h(t)$ $=\int_t^{+\infty}c(l)e^{-l}dl$ for any $t\in [0,+\infty)$.
	
	Recall some notation related to open Riemann surfaces (see \cite{OF81}, see also \cite{guan-zhou13ap,GY-concavity,GMY}). Let $P:\Delta\rightarrow \Omega$ be the universal covering from the unite disc $\Delta$ to $\Omega$. The holomorphic function $\hat{f}$ (resp. holomorphic $(1,0)$ form $\hat{F}$) on $\Delta$ is called as a multiplicative function (resp. multiplicative differential (Prym differential)), if there is a character $\chi$, which is the representation of the fundamental group of $\Omega$ such that $g^*\hat{f}=\chi(g)\hat{f}$ (resp. $g^*(\hat{F})=\chi(g)\hat{F}$), where $|\chi|=1$ and $g$ is an element of the fundamental group of $\Omega$. Denote the set of such $\hat{f}$ by $\mathcal{O}^{\chi}(\Omega)$ (resp. $\Gamma^{\chi}(\Omega)$).
	
	It is known that for any harmonic function $u$ of $\Omega$, there exists a $\chi_u$ and a multiplicative function $f_u\in\mathcal{O}^{\chi_u}(\Omega)$ such that $|f_u|=P^*e^u$. If $u_1-u_2=\log|\hat{f}|$, then $\chi_{u_1}=\chi_{u_2}$, where $u_1$ and $u_2$ are harmonic functions on $\Omega$ and $\hat{f}$ is a holomorphic function on $\Omega$. Recall that for the Green function $G_{\Omega}(z,z_0)$, there exists a $\chi_{z_0}$ and a multiplicative function $f_{z_0}\in\mathcal{O}^{\chi_{z_0}}(\Omega)$ such that $|f_{z_0}(z)|=P^*e^{G_{\Omega}(z,z_0)}$ (see \cite{suita72}).
	
	\subsection{Main result : linearity of the minimal $L^2$ integrals on fibrations over open Riemann surfaces}
 
 \
 
In this section, we present characterizations of the concavity property of minimal $L^2$ integrals degenerating to linearity in the case of fibrations over open Riemann surfaces.

	Let $Z^1_0=\{z_0\}\subset\Omega$. Let $w$ be a local coordinate on a neighborhood $V_{z_0}$ of $z_0$ satisfying $w(z_0)=0$. Let $f$ be a holomorphic $(n,0)$ form on $V_{z_0}\times Y$ which is a neighborhood of $Z_0=\pi_1^{-1}(z_0)$, and $G(t)$ be the minimal $L^2$ integral on $M$ with respect to $f$ for any $t\geq 0$.
	
	We present a characterization of the concavity of $G(h^{-1}(r))$ degenerating to linearity for the fibers over single point sets as follows.
	
	\begin{theorem}\label{one-p}
		Assume that $G(0)\in (0,+\infty)$ and $(\psi_1-2pG_{\Omega}(\cdot,z_0))(z_0)>-\infty$, where $p=\frac{1}{2}v(dd^c(\psi_1),z_0)>0$. Then $G(h^{-1}(r))$ is linear with respect to $r\in (0,\int_0^{+\infty}c(t)e^{-t}dt]$ if and only if the following statements hold:
		
		(1). $\psi_1=2p G_{\Omega}(\cdot,z_j)$;
		
		(2). On $V_{z_0}\times Y$, $f=\pi_1^*(w^kdw)\wedge \pi_2^*(f_{Y})+f_0$, where $k$ is a nonnegative integer, $f_{Y}$ is a holomorphic $(n-1,0)$ form on $Y$ such that $\int_{Y}|f_{Y}|^2e^{-\varphi_2}\in (0,+\infty)$, and $(f_0,(z_0,y))\in (\mathcal{O}(K_M))_{(z_0,y)}\otimes\mathcal{I}(\varphi+\psi)_{(z_0,y)}$ for any $(z_0,y)\in Z_0$;
		
		(3). $\varphi_1+\psi_1=2\log |g|+2G_{\Omega}(\cdot,z_0)+2u$, where $g$ is a holomorphic function on $\Omega$ such that $ord_{z_0}(g)=k$ and $u$ is a harmonic function on $\Omega$;
		
		(4). $\chi_{z_0}=\chi_{-u}$, where $\chi_{-u}$ and $\chi_{z_0}$ are the characters associated to the functions $-u$ and $G_{\Omega}(\cdot,z_0)$ respectively.
		
	\end{theorem}
	
	When $Y$ is a single point, Theorem \ref{one-p} can be referred to \cite{GY-concavity}.
	
	\begin{remark}\label{rem-p}
		When the four statements in Theorem \ref{one-p} hold,
		\begin{equation*}
			c_0\pi_1^*(gP_*(f_udf_{z_0}))\wedge\pi_2^*(f_{Y})
		\end{equation*}
		is the unique holomorphic $(n,0)$ form $F$ on $M$ such that $(F-f,(z_0,x))\in (\mathcal{O}(K_M))_{(z_0,x)}\otimes\mathcal{I}(\varphi+\psi)_{(z_0,x)}$ for any $(z_0,x)\in Z_0$ and
		\begin{equation*}
			G(t)=\int_{\{\psi<-t\}}|F|^2e^{-\varphi}c(-\psi)=\left(\int_t^{+\infty}c(s)e^{-s}ds\right)\frac{2\pi e^{-2u(z_0)}}{p|d|^2}\int_{Y}|f_{Y}|^2e^{-\varphi_2}
		\end{equation*}
		for any $t\geq 0$, where $f_u$ is a holomorphic function on $\Delta$ such that $|f_u|=P^*(e^u)$, $f_{z_0}$ is a holomorphic function on $\Delta$ such that $|f_{z_0}|=P^*(e^{G_{\Omega}(\cdot,z_0)})$,
		\begin{equation*}
			c_0:=\lim_{z\rightarrow z_0}\frac{w^kdw}{gP_*(f_udf_{z_0})}\in\mathbb{C}\setminus\{0\},
		\end{equation*}
		and
		\begin{equation*}
			d:=\lim_{z\rightarrow z_0}\frac{g}{w^k}(z).
		\end{equation*}
	\end{remark}
	
	Let $Z_0^1:=\{z_j:j\in \mathbb{N} \& 1\leq j\leq m\}$ be a finite subset of the open Riemann surface $\Omega$. Let $Y$ be an $n-1$ dimensional weakly pseudoconvex K\"{a}hler manifold. Let $M=\Omega\times Y$ be a complex manifold, and $K_M$ be the canonical line bundle on $M$. Let $\pi_1$, $\pi_2$ be the natural projections from $M$ to $\Omega$ and $Y$ and $Z_0:=\pi_1^{-1}(Z_0^1)$. Let $\psi_1$ be a subharmonic function on $\Omega$ such that $p_j=\frac{1}{2}v(dd^c\psi_1,z_j)>0$, and let $\varphi_1$ be a Lebesgue measurable function on $\Omega$ such that $\varphi_1+\psi_1$ is subharmonic on $\Omega$. Let $\varphi_2$ be a plurisubharmonic function on $Y$. Denote that $\psi:=\pi_1^*(\psi_1)$, $\varphi:=\pi_1^*(\varphi_1)+\pi_2^*(\varphi_2)$.
	
	Let $w_j$ be a local coordinate on a neighborhood $V_{z_j}\subset\subset\Omega$ of $z_j$ satisfying $w_j(z_j)=0$ for $z_j\in Z_0^1$, where $V_{z_j}\cap V_{z_k}=\emptyset$ for any $j,k$, $j\neq k$. Denote that $V_0:=\bigcup_{1\leq j\leq m}V_{z_j}$. Let $f$ be a holomorphic $(n,0)$ form on $V_0\times Y$ and $G(t)$ be the minimal $L^2$ integral on $M$ with respect to $f$ for any $t\geq 0$.
	
	We present a characterization of the concavity of $G(h^{-1}(r))$ degenerating to linearity for the fibers over sets of finite points as follows.

	\begin{theorem}\label{finite-p}
		Assume that $G(0)\in (0,+\infty)$ and $(\psi_1-2p_jG_{\Omega}(\cdot,z_j))(z_j)>-\infty$, where $p_j=\frac{1}{2}v(dd^c(\psi_1),z_j)>0$ for any $j\in\{1,2,\ldots,m\}$. Then $G(h^{-1}(r))$ is linear with respect to $r\in (0,\int_0^{+\infty}c(t)e^{-t}dt]$ if and only if the following statements hold:
		
		(1). $\psi_1=2\sum\limits_{j=1}^mp_j G_{\Omega}(\cdot,z_j)$;
		
		(2). for any $j\in\{1,2,\ldots,m\}$, $f=\pi_1^*(a_jw_j^{k_j}dw_j)\wedge \pi_2^*(f_{Y})+f_j$ on $V_{z_j}\times Y$, where $a_j\in\mathbb{C}\setminus \{0\}$ is a constant, $k_j$ is a nonnegative integer, $f_{Y}$ is a holomorphic $(n-1,0)$ form on $Y$ such that $\int_{Y}|f_{Y}|^2e^{-\varphi_2}\in (0,+\infty)$, and $(f_j,(z_j,y))\in (\mathcal{O}(K_M))_{(z_j,y)}\otimes\mathcal{I}(\varphi+\psi)_{(z_j,y)}$ for any $j\in\{1,2,\ldots,m\}$ and $y\in Y$;
		
		(3). $\varphi_1+\psi_1=2\log |g|+2\sum\limits_{j=1}^mG_{\Omega}(\cdot,z_j)+2u$, where $g$ is a holomorphic function on $\Omega$ such that $ord_{z_j}(g)=k_j$ and $u$ is a harmonic function on $\Omega$;
		
		(4). $\prod\limits_{j=1}^m\chi_{z_j}=\chi_{-u}$, where $\chi_{-u}$ and $\chi_{z_j}$ are the characters associated to the functions $-u$ and $G_{\Omega}(\cdot,z_j)$ respectively;
		
		(5). for any $j\in\{1,2,\ldots,m\}$,
		\begin{equation}
			\lim_{z\rightarrow z_j}\frac{a_jw_j^{k_j}dw_j}{gP_*\left(f_u\left(\prod\limits_{l=1}^mf_{z_l}\right)\left(\sum\limits_{l=1}^mp_l\dfrac{d{f_{z_{l}}}}{f_{z_{l}}}\right)\right)}=c_0,
		\end{equation}
		where $c_0\in\mathbb{C}\setminus\{0\}$ is a constant independent of $j$.
	\end{theorem}
	
	\begin{remark}\label{rem-finite}
		When the five statements in Theorem \ref{one-p} hold,
		\begin{equation*}
			c_0\pi_1^*\left(gP_*\left(f_u\left(\prod\limits_{l=1}^mf_{z_l}\right)\left(\sum\limits_{l=1}^mp_l\dfrac{d{f_{z_{l}}}}{f_{z_{l}}}\right)\right)\right)\wedge\pi_2^*(f_{Y})
		\end{equation*}
		is the unique holomorphic $(n,0)$ form $F$ on $M$ such that $(F-f,(z_j,y))\in (\mathcal{O}(K_M))_{(z_j,y)}\otimes\mathcal{I}(\varphi+\psi)_{(z_j,y)}$ for any $(z_j,y)\in Z_0$ and
		\begin{equation*}
			G(t)=\int_{\{\psi<-t\}}|F|^2e^{-\varphi}c(-\psi)=\left(\int_t^{+\infty}c(s)e^{-s}ds\right)\left(\sum_{j=1}^m\frac{2\pi |a_j|^2e^{-2u(z_j)}}{p_j|d_j|^2}\right)\int_{Y}|f_{Y}|^2e^{-\varphi_2}
		\end{equation*}
		for any $t\geq 0$, where $f_u$ is a holomorphic function on $\Delta$ such that $|f_u|=P^*(e^u)$, $f_{z_j}$ is a holomorphic function on $\Delta$ such that $|f_{z_j}|=P^*(e^{G_{\Omega}(\cdot,z_j)})$ for any $j\in\{1,2,\ldots,m\}$,
		and $d_j:=\lim\limits_{z\rightarrow z_j}\frac{g}{w^{k_j}}(z_j)$.
	\end{remark}
	
	Let $Z_0^1:=\{z_j:j\in \mathbb{N}_+\}$ be an infinite discrete subset of the open Riemann surface $\Omega$. Let $Y$ be an $n-1$ dimensional weakly pseudoconvex K\"{a}hler manifold. Let $M=\Omega\times Y$ be a complex manifold, and $K_M$ be the canonical line bundle on $M$. Let $\pi_1$, $\pi_2$ be the natural projections from $M$ to $\Omega$ and $Y$ and $Z_0:=\pi_1^{-1}(Z_0^1)$. Let $\psi_1$ be a subharmonic function on $\Omega$ such that $p_j=\frac{1}{2}v(dd^c\psi_1,z_j)>0$, and let $\varphi_1$ be a Lebesgue measurable function on $\Omega$ such that $\varphi_1+\psi_1$ is subharmonic on $\Omega$. Let $\varphi_2$ be a plurisubharmonic function on $Y$. Denote that $\psi:=\pi_1^*(\psi_1)$, $\varphi:=\pi_1^*(\varphi_1)+\pi_2^*(\varphi_2)$.
	
	Let $w_j$ be a local coordinate on a neighborhood $V_{z_j}\subset\subset\Omega$ of $z_j$ satisfying $w_j(z_j)=0$ for $z_j\in Z_0^1$, where $V_{z_j}\cap V_{z_k}=\emptyset$ for any $j,k$, $j\neq k$. Denote that $V_0:=\bigcup_{j=1}^{\infty}V_{z_j}$. Let $f$ be a holomorphic $(n,0)$ form on $V_0\times Y$ and $G(t)$ be the minimal $L^2$ integral on $M$ with respect to $f$ for any $t\geq 0$.
	
	We present a necessary condition such that $G(h^{-1}(r))$ is linear for the fibers over infinite analytic subsets.
	\begin{proposition}\label{infinite-p}
		Assume that $G(0)\in (0,+\infty)$ and $(\psi_1-2p_jG_{\Omega}(\cdot,z_j))(z_j)>-\infty$, where $p_j=\frac{1}{2}v(dd^c(\psi_1),z_j)>0$ for any $j\in\mathbb{N}_+$. Assume that $G(h^{-1}(r))$ is linear with respect to $r\in (0,\int_0^{+\infty}c(t)e^{-t}dt]$, then the following statements hold:
		
		(1). $\psi_1=2\sum\limits_{j=1}^{\infty}p_j G_{\Omega}(\cdot,z_j)$;
		
		(2). for any $j\in\mathbb{N}_+$, $f=\pi_1^*(a_jw_j^{k_j}dw_j)\wedge \pi_2^*(f_{Y})+f_j$ on $V_{z_j}\times Y$, where $a_j\in\mathbb{C}\setminus \{0\}$ is a constant, $k_j$ is a nonnegative integer, $f_{Y}$ is a holomorphic $(n-1,0)$ form on $Y$ such that $\int_{Y}|f_{Y}|^2e^{-\varphi_2}\in (0,+\infty)$, and $(f_j,(z_j,y))\in (\mathcal{O}(K_M))_{(z_j,y)}\otimes\mathcal{I}(\varphi+\psi)_{(z_j,y)}$ for any $j\in\mathbb{N}_+$ and $y\in Y$;
		
		(3). $\varphi_1+\psi_1=2\log |g|$, where $g$ is a holomorphic function on $\Omega$ such that $ord_{z_j}(g)=k_j+1$ for any $j\in\mathbb{N}_+$;
		
		(4). for any $j\in\mathbb{N}_+$,
		\begin{equation}
			\frac{p_j}{ord_{z_j}g}\lim_{z\rightarrow z_j}\frac{dg}{a_jw_j^{k_j}dw_j}=c_0,
		\end{equation}
		where $c_0\in\mathbb{C}\setminus\{0\}$ is a constant independent of $j$;
		
		(5). $\sum\limits_{j\in\mathbb{N}_+}p_j<+\infty$.
	\end{proposition}
	
	Let $Z_0^1:=\{z_j:j\in \mathbb{N} \& 1\leq j<\gamma\}$ be a discrete subset of the open Riemann surface $\Omega$, where $\gamma\in\mathbb{N}_+\cup\{+\infty\}$. Let $Y$ be an $n-1$ dimensional weakly pseudoconvex K\"{a}hler manifold. Let $M=\Omega\times Y$ be a complex manifold. Assume that $\tilde{M}\subset M$ is an $n-$dimensional weakly pseudoconvex submanifold satisfying that $Z_0\subset\tilde{M}$. Let $f$ be a holomorphic $(n,0)$ form on a neighborhood of $Z_0$ in $\tilde{M}$. Denote that
	\begin{flalign*}
		\begin{split}
		\tilde{G}(t):=\inf\bigg\{\int_{\{\psi<-t\}\cap \tilde{M}}|\tilde{f}|^2e^{-\varphi}c(-\psi) :& (\tilde{f}-f,(z,y))\in (\mathcal{O}(K_{\tilde{M}}))_{(z,y)}\otimes\mathcal{I}(\varphi+\psi)_{(z,y)} \\
			&\text{\ for \ any } (z,y)\in Z_0  \\
			& \& \ \tilde{f}\in H^0(\{\psi<-t\}\cap\tilde{M},\mathcal{O}(K_{\tilde{M}}))\bigg\}
		\end{split}
	\end{flalign*} 
	for any $t\geq 0$, where $K_{\tilde{M}}$ is the canonical line bundle on $\tilde{M}$, and $\psi, \varphi$ are as above.

We present a necessary condition such that $\tilde{G}(h^{-1}(r))$ is linear for the fibrations over analytic subsets.
	
	\begin{proposition}\label{tildeM}
		Assume that $\tilde{G}(0)\in (0,+\infty)$ and $(\psi_1-2p_jG_{\Omega}(\cdot,z_j))(z_j)>-\infty$, where $p_j=\frac{1}{2}v(dd^c(\psi_1),z_j)>0$ for any $j\in\mathbb{N}$. If $\tilde{G}(h^{-1}(r))$ is linear with respect to $r\in (0,\int_0^{+\infty}c(t)e^{-t}dt]$, then $\tilde{M}=M$.
	\end{proposition}
	
	\subsection{Applications: optimal $L^2$ extension problem from fibers over analytic subsets to fibrations over open Riemann surfaces}
	
	\
	
In this section, we give characterizations of the holding of equality in optimal jets $L^2$ extension problem from fibers over analytic subsets to fibrations over open Riemann surfaces. 	
	
	Let $Z_0^1:=\{z_j:j\in \mathbb{N} \& 1\leq j\leq m\}$ be a finite subset of the open Riemann surface $\Omega$. Let $Y$ be an $n-1$ dimensional weakly pseudoconvex K\"{a}hler manifold. Let $M=\Omega\times Y$ be a complex manifold, and $K_M$ be the canonical line bundle on $M$. Let $\pi_1$, $\pi_2$ be the natural projections from $M$ to $\Omega$ and $Y$, and $Z_0:=\pi_1^{-1}(Z_0^1)$.
	
	Let $w_j$ be a local coordinate on a neighborhood $V_{z_j}\subset\subset\Omega$ of $z_j$ satisfying $w_j(z_j)=0$ for $z_j\in Z_0^1$, where $V_{z_j}\cap V_{z_k}=\emptyset$ for any $j,k$, $j\neq k$. Let $c_{\beta}(z)$ be the logarithmic capacity (see \cite{S-O69}) on $\Omega$ which is defined by
	\begin{equation*}
		c_{\beta}(z_j):=\exp \lim_{z\rightarrow z_0}(G_{\Omega}(z,z_j)-\log|w_j(z)|).
	\end{equation*}
	Denote that $V_0:=\bigcup_{1\leq j\leq m}V_{z_j}$. Assume that $\tilde{M}\subset M$ is an $n-$dimensional weakly pseudoconvex submanifold satisfying that $Z_0\subset \tilde{M}$.

As an application of Theorem \ref{finite-p}, we give a characterization of the holding of equality in optimal $L^2$ extension problem from analytic subsets to fibrations over open Riemann surfaces, where the analytic subsets are fibers over finite points on open Riemann surfaces.
	
	\begin{theorem}\label{fib-L2ext}
		Let $k_j$ be a nonnegative integer for any $j\in\{1,2,...,m\}$. Let $\psi_1$ be a negative subharmonic function on $\Omega$ satisfying that   $\frac{1}{2}v(dd^{c}\psi_1,z_j)=p_j>0$ for any $j\in\{1,2,...,m\}$. Let $\varphi_1$ be a Lebesgue measurable function on $\Omega$ such that $\varphi_1+\psi_1$ is subharmonic on $\Omega$, $\frac{1}{2}v(dd^c(\varphi_1+\psi_1),z_j)=k_j+1$ and $\alpha_j:=(\varphi_1+\psi_1-2(k_j+1)G_{\Omega}(\cdot,z_j))(z_j)>-\infty$ for any $j$. Let $\varphi_2$ be a plurisubharmonic function on $Y$. Let $c(t)$ be a positive measurable function on $(0,+\infty)$ satisfying that $c(t)e^{-t}$ is decreasing on $(0,+\infty)$ and $\int_{0}^{+\infty}c(s)e^{-s}ds<+\infty$. Let $a_j$ be a constant for any $j$. Let $F_j$ be a holomorphic $(n-1,0)$ form on $Y$ such that $\int_Y|F_j|^2e^{-\varphi_2}<+\infty$ for any $j$.
		
		Let $f$ be a holomorphic $(n,0)$ form on $V_0\times Y$ satisfying that $f=\pi_1^*(a_jw_j^{k_j}dw_j)\wedge\pi_2^*(F_j)$ on $V_{z_j}\times Y$. Then there exists a holomorphic $(n,0)$ form $F$ on $\tilde{M}$ such that $(F-f,(z_j,y))\in(\mathcal{O}(K_{\tilde{M}})\otimes\mathcal{I}(\varphi+\psi))_{(z_j,y)}$ for any $(z_j,y)\in Z_0$ and
		\begin{equation}\label{L2result}
			\int_{\tilde{M}}|F|^2e^{-\varphi}c(-\psi)\leq\left(\int_0^{+\infty}c(s)e^{-s}ds\right)\sum_{j=1}^m\frac{2\pi|a_j|^2e^{-\alpha_j}}{p_jc_{\beta}(z_j)^{2(k_j+1)}}\int_Y|F_j|^2e^{-\varphi_2}.
		\end{equation}
		
		Moreover, equality $\left(\int_0^{+\infty}c(s)e^{-s}ds\right)\sum\limits_{j=1}^m\frac{2\pi|a_j|^2e^{-\alpha_j}}{p_jc_{\beta}(z_j)^{2(k_j+1)}}\int_Y|F_j|^2e^{-\varphi_2}=\inf\big\{$ 
		$\int_M|\tilde{F}|^2e^{-\varphi}c(-\psi):\tilde{F}$ is a holomorphic $(n,0)$ form on $\tilde{M}$ such that $(\tilde{F}-f,(z_j,y))\in(\mathcal{O}(K_{\tilde{M}})\otimes\mathcal{I}(\varphi+\psi))_{(z_j,y)}$ for any $(z_j,y)\in Z_0\big\}$ holds if and only if the following statements hold:
		
		(1). $\psi_1=2\sum\limits_{j=1}^mp_j G_{\Omega}(\cdot,z_j)$;
		
		(2). $\varphi_1+\psi_1=2\log |g|+2\sum\limits_{j=1}^m(k_j+1)G_{\Omega}(\cdot,z_j)+2u$, where $g$ is a holomorphic function on $\Omega$ such that $g(z_j)\neq 0$ for any $j\in\{1,2,\ldots,m\}$ and $u$ is a harmonic function on $\Omega$;
		
		(3). $\prod\limits_{j=1}^m\chi_{z_j}^{k_j+1}=\chi_{-u}$, where $\chi_{-u}$ and $\chi_{z_j}$ are the characters associated to the functions $-u$ and $G_{\Omega}(\cdot,z_j)$ respectively;
		
		(4). for any $j\in\{1,2,\ldots,m\}$,
		\begin{equation}
			\lim_{z\rightarrow z_j}\frac{a_jw_j^{k_j}dw_j}{gP_*\left(f_u\left(\prod\limits_{l=1}^mf_{z_l}^{k_l+1}\right)\left(\sum\limits_{l=1}^mp_l\dfrac{d{f_{z_{l}}}}{f_{z_{l}}}\right)\right)}=c_j\in\mathbb{C}\setminus\{0\},
		\end{equation}
		and there exist $c_0\in\mathbb{C}\setminus\{0\}$ and a holomorphic $(n-1,0)$ form $F_Y$ on $Y$ which are independent of $j$ such that $c_0F_Y=c_jF_j$ for any $j\in\{1,2,\ldots,m\}$;
		
		(5). $\tilde{M}=M$.
		
	\end{theorem}
	
	\begin{remark}\label{rem:fib-L2ext}
		When the five statements in Theorem \ref{fib-L2ext} hold,
		\[c_0\pi_1^*\left(gP_*\left(f_u\left(\prod\limits_{l=1}^mf_{z_l}^{k_l+1}\right)\left(\sum\limits_{l=1}^mp_l\dfrac{d{f_{z_{l}}}}{f_{z_{l}}}\right)\right)\right)\wedge\pi_2^*(F_{Y})\]
		is the unique holomorphic $(n,0)$ form $F$ on $M$ such that $(F-f,(z_j,y))\in(\mathcal{O}(K_M)\otimes\mathcal{I}(\varphi+\psi))_{(z_j,y)}$ for any $(z_j,y)\in Z_0$ and
		\begin{equation*}
			\int_M|F|^2e^{-\varphi}c(-\psi)\leq\left(\int_0^{+\infty}c(s)e^{-s}ds\right)\sum_{j=1}^m\frac{2\pi|a_j|^2e^{-\alpha_j}}{p_jc_{\beta}(z_j)^{2(k_j+1)}}\int_Y|F_j|^2e^{-\varphi_2}.
		\end{equation*}
	\end{remark}

	Let $Z_0^1:=\{z_j:j\in \mathbb{N}_+\}$ be an infinite discrete subset of the open Riemann surface $\Omega$. Let $Y$ be an $n-1$ dimensional weakly pseudoconvex K\"{a}hler manifold. Let $M=\Omega\times Y$ be a complex manifold, and $K_M$ be the canonical line bundle on $M$. Let $\pi_1$, $\pi_2$ be the natural projections from $M$ to $\Omega$ and $Y$, and $Z_0:=\pi_1^{-1}(Z_0^1)$.
	
	Let $w_j$ be a local coordinate on a neighborhood $V_{z_j}\subset\subset\Omega$ of $z_j$ satisfying $w_j(z_j)=0$ for $z_j\in Z_0^1$, where $V_{z_j}\cap V_{z_k}=\emptyset$ for any $j,k$, $j\neq k$. Denote that $V_0:=\bigcup_{j=1}^{\infty}V_{z_j}$. 

We give an $L^2$ extension result from fibers over analytic subsets to fibrations over open Riemann surfaces, where the analytic subsets are  infinite points on open Riemann surfaces.

	\begin{theorem}\label{fib-L2ext-infinite}
		Let $k_j$ be a nonnegative integer for any $j\in\mathbb{N}_+$. Let $\psi_1$ be a negative subharmonic function on $\Omega$ satisfying that   $\frac{1}{2}v(dd^{c}\psi_1,z_j)=k_j+1>0$ for any $j\in\mathbb{N}_+$. Let $\varphi_1$ be a Lebesgue measurable function on $\Omega$ such that $\varphi_1+\psi_1$ is subharmonic on $\Omega$, $\frac{1}{2}v(dd^c(\varphi_1+\psi_1),z_j)=k_j+1$ and $\alpha_j:=(\varphi_1+\psi_1-2(k_j+1)G_{\Omega}(\cdot,z_j))(z_j)>-\infty$ for any $j$. Let $\varphi_2$ be a plurisubharmonic function on $Y$. Let $c(t)$ be a positive measurable function on $(0,+\infty)$ satisfying that $c(t)e^{-t}$ is decreasing on $(0,+\infty)$ and $\int_{0}^{+\infty}c(s)e^{-s}ds<+\infty$. Let $a_j$ be a constant for any $j$. Let $F_j$ be a holomorphic $(n-1,0)$ form on $Y$ such that $\int_Y|F_j|^2e^{-\varphi_2}<+\infty$ for any $j$.
		
		Let $f$ be a holomorphic $(n,0)$ form on $V_0\times Y$ satisfying that $f=\pi_1^*(a_jw_j^{k_j}dw_j)\wedge\pi_2^*(F_j)$ on $V_{z_j}\times Y$.
		If
		\begin{equation*}
			\sum_{j=1}^{\infty}\frac{2\pi|a_j|^2e^{-\alpha_j}}{(k_j+1)c_{\beta}(z_j)^{2(k_j+1)}}\int_Y|F_j|^2e^{-\varphi_2}<+\infty,
		\end{equation*}	
		then there exists a holomorphic $(n,0)$ form $F$ on $M$ such that $(F-f,(z_j,y))\in(\mathcal{O}(K_M)\otimes\mathcal{I}(\varphi+\psi))_{(z_j,y)}$ for any $(z_j,y)\in Z_0$ and
		\begin{equation}\label{L2result-infinite}
			\int_M|F|^2e^{-\varphi}c(-\psi)<\left(\int_0^{+\infty}c(s)e^{-s}ds\right)\sum_{j=1}^{\infty}\frac{2\pi|a_j|^2e^{-\alpha_j}}{(k_j+1)c_{\beta}(z_j)^{2(k_j+1)}}\int_Y|F_j|^2e^{-\varphi_2}.
		\end{equation}
	\end{theorem}

	\begin{remark}
		Assume that $\tilde{M}\subset M$ is an $n-$dimensional weakly pseudoconvex submanifold satisfying that $Z_0\subset \tilde{M}$. Then according to Theorem \ref{fib-L2ext-infinite} there exists a holomorphic $(n,0)$ form $F$ on $\tilde{M}$ such that $(F-f,(z_j,y))\in(\mathcal{O}(K_{\tilde{M}})\otimes\mathcal{I}(
		\varphi+\psi)_{(z_j,y)}$ for any $(z_j,y)\in Z_0$ and
		\begin{flalign*}
			\begin{split}
				\int_{\tilde{M}}|F|^2e^{-\varphi}c(-\psi)<&\int_M|F|^2e^{-\varphi}c(-\psi)\\
				\leq&\left(\int_0^{+\infty}c(s)e^{-s}ds\right)\sum_{j=1}^{\infty}\frac{2\pi|a_j|^2e^{-\alpha_j}}{(k_j+1)c_{\beta}(z_j)^{2(k_j+1)}}\int_Y|F_j|^2e^{-\varphi_2}.
			\end{split}
		\end{flalign*}
	\end{remark}
	
	\subsection{Suita conjecture and extended Suita conjecture}
	
	\
	
In this section, we present the characterizations of the fibration versions of the equality parts of Suita conjecture and extended Suita conjecture.

	Let $\Omega$  be an open Riemann surface, which admits a nontrivial Green function $G_{\Omega}$, and let $K_{\Omega}$ be the canonical (holomorphic) line bundle on $\Omega$. Let $w$ be a local coordinate on a neighborhood $V_{z_0}$ of $z_0\in\Omega$ satisfying $w(z_0)=0$. Denote the space of $L^2$ integrable holomorphic section of $K_{\Omega}$ by $A^2(\Omega,K_{\Omega},dV_{\Omega}^{-1},dV_{\Omega})$.
	Let $\{\phi_l\}_{l=1}^{+\infty}$ be a complete orthogonal system of $A^2(\Omega,K_{\Omega},dV_{\Omega}^{-1},dV_{\Omega})$ satisfying $(\sqrt{-1})\int_{M}\frac{\phi_i}{\sqrt{2}}\wedge\frac{\overline{\phi}_j}{\sqrt{2}}=\delta_i^j$. Put $\kappa_{\Omega}=\sum_{l=1}^{+\infty}\phi_l\otimes\overline\phi_l\in C^{\omega}(\Omega,K_{\Omega}\otimes\overline{K_{\Omega}})$. Denote that
	\[B_{\Omega}(z)dw\otimes\overline{dw}:=\kappa_{\Omega}|_{V_{z_0}}.\]
	Let $c_{\beta}(z)$ be the logarithmic capacity (see \cite{S-O69}) which is locally defined by
	\[c_{\beta}(z_0):=\exp\lim_{z\rightarrow z_0}(G_{\Omega}(z,z_0)-\log|w(z)|)\]
	on $\Omega$.
	In \cite{suita72}, Suita stated a conjecture as below.
	\begin{conjecture}
		$c_{\beta}(z_0)^2\le\pi B_{\Omega}(z_0)$ holds for any $z_0\in \Omega$, and equality holds if and only if $\Omega$ is conformally equivalent to the unit disc less a (possible) closed set of inner capacity zero.
	\end{conjecture}

	The inequality part of  Suita conjecture for bounded planar domain was proved by B\l ocki \cite{Blo13}, and original form of the inequality was proved by Guan-Zhou \cite{gz12}.
	The equality part of Suita conjecture was proved by Guan-Zhou \cite{guan-zhou13ap}, which completed the proof of Suita conjecture.

	Let $\Omega$  be an open Riemann surface, which admits a nontrivial Green function $G_{\Omega}$. Let $Y$ be an $(n-1)-$dimensional weakly pseudoconvex K\"ahler manifold, and let $K_Y$ be the canonical (holomorphic) line bundle on $Y$. Let $M=\Omega\times Y$ be an $n-$dimensional complex manifold. Let $\pi_1$, $\pi_2$ be the natural projections from $M$ to $\Omega$ and $Y$ respectively. Let $K_M$ be the canonical (holomorphic) line bundle on $M$.

	Denote the space of $L^2$ integrable holomorphic section of $K_M$ (resp. $K_Y$) by $A^2(M,K_M,dV_M^{-1},dV_M)$ (resp. $A^2(Y,K_Y,dV_Y^{-1},dV_Y)$).
	Let $\{\sigma_l\}_{l=1}^{+\infty}$ (resp. $\{\tau_l\}_{l=1}^{+\infty}$) be a complete orthogonal system of $A^2(M,K_M,dV_M^{-1},dV_M)$ (resp. $A^2(Y,K_Y,dV_Y^{-1},dV_Y)$) satisfying $(\sqrt{-1})^{n^2}\int_{M}\frac{\sigma_i}{\sqrt{2^n}}\wedge\frac{\overline{\sigma}_j}{\sqrt{2^n}}=\delta_i^j$ (resp. $(\sqrt{-1})^{(n-1)^2}\int_{Y}\frac{\tau_i}{\sqrt{2^{n-1}}}\wedge\frac{\overline{\tau}_j}{\sqrt{2^{n-1}}}=\delta_i^j$). Put $\kappa_M=\sum_{l=1}^{+\infty}\sigma_l\otimes\overline\sigma_l\in C^{\omega}(M,K_M\otimes\overline{K_M})$ and $\kappa_Y=\sum_{l=1}^{+\infty}\tau_l\otimes\overline\tau_l\in C^{\omega}(Y,K_Y\otimes\overline{K_Y})$.

	Let $z_0\in\Omega$, $y_0\in Y$.
	Let $w$ be a local coordinate on a neighborhood $V_{z_0}$ of $z_0$ in $\Omega$ satisfying $w(z_0)=0$. Let $\tilde w=(\tilde w_1,\ldots,\tilde w_{n-1})$ be a local coordinate on a neighborhood $U_{y_0}$ of $y_0$ in $Y$ satisfying that $\tilde w_j(y_0)=0$ for any $j\in\{1,2,\ldots,n-1\}$. Denote that
	\[B_{M}((z,y))dw\wedge d\tilde w_1\wedge\ldots d\tilde w_{n-1} \otimes\overline{dw\wedge d\tilde w_1\wedge\ldots d\tilde w_{n-1}}:=\kappa_{M}\]
	on $V_{z_0}\times U_{y_0}$ and
	\[B_{Y}(y) d\tilde w_1\wedge\ldots d\tilde w_{n-1} \otimes\overline{d\tilde w_1\wedge\ldots d\tilde w_{n-1}}:=\kappa_{Y}\]
	on $U_{y_0}$.
	Let $c_{\beta}(z_0)$ be the logarithmic capacity which is locally defined by
	\[c_{\beta}(z_0):=\exp\lim_{z\rightarrow z_0}(G_{\Omega}(z,z_0)-\log|w(z)|).\]
	
	Assume that $B_Y(y_0)>0$. We give the following theorem as a characterization of the holding of equality in the fibration version of  Suita conjecture.
	\begin{theorem}
		\label{thm:suita}
		$c_{\beta}(z_0)^2B_Y(y_0)\leq \pi B_M((z_0,y_0))$ holds, and equality holds if and only if $\Omega$ is conformally equivalent to the unit disc less a (possible) closed set of inner capacity zero.
	\end{theorem}
	
	Let $\tilde{M}\subset M$ be an $n-$dimensional complex manifold satisfying that $\{z_0\}\times Y\subset \tilde{M}$. Similar to $M$, the Bergman kernel $B_{\tilde{M}}$ can be defined. Theorem \ref{thm:suita} implies the following result.
	
	\begin{remark}
		\label{r:suita}	$c_{\beta}(z_0)^2B_Y(y_0)\leq \pi B_{\tilde{M}}((z_0,y_0))$ holds, and equality holds if and only if $\tilde{M}=M$ and $\Omega$ is conformally equivalent to the unit disc less a (possible) closed set of inner capacity zero.
	\end{remark}
	
	Let $\Omega$  be an open Riemann surface which admits a nontrivial Green function $G_{\Omega}$, and let $K_{\Omega}$ be the canonical (holomorphic) line bundle on $\Omega$. Let $w$ be a local coordinate on a neighborhood $V_{z_0}$ of $z_0\in\Omega$ satisfying $w(z_0)=0$. Let $\rho:=e^{-2u}$ on $\Omega$, where $u$ is a harmonic function on $\Omega$. Denote that
	$$B_{\Omega,\rho}dw\otimes\overline{dw}:=\sum_{l=1}^{+\infty}\sigma_l\otimes\overline{\sigma}_l|_{V_{z_0}}\in C^{\omega}(V_{z_0},K_{\Omega}\otimes\overline{K_{\Omega}}),$$
	where $\{\sigma_l\}_{l=1}^{+\infty}$ are holomorphic $(1,0)$ forms on $\Omega$ satisfying \[\sqrt{-1}\int_{\Omega}\rho\frac{\sigma_i}{\sqrt{2}}\wedge\frac{\overline{\sigma}_j}{\sqrt{2}}=\delta_i^j\]
	 and
	 \[\{F\in H^0(\Omega,K_{\Omega}):\int_{\Omega}\rho|F|^2<+\infty\ \&\ \int_{\Omega}\rho\sigma_l\wedge\overline F=0 \text{ \ for \ any \ } l\in\mathbb{N}_+\}=\{0\}.\]
	
	In \cite{Yamada}, Yamada  stated a conjecture as below (so-called extended Suita conjecture).
	\begin{conjecture}
		$c_{\beta}(z_0)^2\le\pi \rho(z_0) B_{\Omega,\rho}(z_0)$ holds for any $z_0\in \Omega$, and equality holds if and only $\chi_{-u}=\chi_{z_0}$, where $\chi_{-u}$ and $\chi_{z_0}$ are the characters associated to the functions $-u$ and $G_{\Omega}(\cdot,z_0)$ respectively.
	\end{conjecture}
	The inequality part of extended Suita conjecture  was proved by Guan-Zhou \cite{GZ15}.
	The equality part of extended Suita conjecture was proved by Guan-Zhou \cite{guan-zhou13ap}.

	Let $\rho:=e^{-2\pi_1^*(u)}$ on $M$, where $u$ is a harmonic function on $\Omega$. Denote that
	\[B_{M,\rho}d\wedge d\tilde w_1\wedge\ldots d\tilde w_{n-1} \otimes\overline{dw\wedge d\tilde w_1\wedge\ldots d\tilde w_{n-1}}:=\sum_{l=1}^{+\infty}e_l\otimes\overline{e}_l\]
	on $V_0\times Y$,
	where $\{e_l\}_{l=1}^{+\infty}$ are holomorphic $(n,0)$ forms on $M$ satisfying
	 \[(\sqrt{-1})^{n^2}\int_{M}\rho\frac{e_i}{\sqrt{2^n}}\wedge\frac{\overline{e}_j}{\sqrt{2^n}}=\delta_i^j\]
	 and
	 \[\{F\in H^0(M,K_{M}):\int_{M}\rho|F|^2<+\infty\ \&\ \int_{M}\rho e_l\wedge\overline F=0 \text{\ for \ any \ } l\in\mathbb{N}_+\}=\{0\}.\]
	
	Assume that $B_Y(y_0)>0$. We give the following theorem as a characterization of the holding of equality in the fibration version of extended Suita conjecture.
	\begin{theorem}
		\label{thm:extend}
		$c_{\beta}(z_0)^2B_Y(y_0)\leq \pi \rho(z_0) B_{M,\rho}(z_0)$ holds, and equality holds if and only if $\chi_{-u}=\chi_{z_0}$, where $\chi_{-u}$ and $\chi_{z_0}$ are the characters associated to the functions $-u$ and $G_{\Omega}(\cdot,z_0)$ respectively.
	\end{theorem}
	
	Let $\tilde{M}\subset M$ be an $n-$dimensional complex manifold satisfying that $\{z_0\}\times Y \subset \tilde{M}$. Similar to $M$, the Bergman kernel $B_{\tilde{M},\rho}$ can be defined. Theorem \ref{thm:extend} implies the following result.
	
	\begin{remark}
		\label{r:extend}
		$c_{\beta}(z_0)^2B_Y(y_0)\leq \pi B_{\tilde{M},\rho}((z_0,y_0))$ holds, and equality holds if and only if $\tilde{M}=M$ and $\chi_{-u}=\chi_{z_0}$.
	\end{remark}
	
	\section{preparation}
	\subsection{Concavity property of minimal $L^2$ integrals}
	
	\
	
	Recall the concavity property of minimal $L^2$ integrals on weakly pseudoconvex K\"{a}hler manifolds in \cite{GMY}. Let $M$ be an $n-$dimensional complex manifold. Let $X$ and $Z$ be closed subsets of $M$. A triple $(M,X,Z)$ satisfies condition $(A)$, if the following statements hold:
	
	(1). $X$ is a closed subset of $M$ and $X$ is locally negligible with respect to $L^2$ holomorphic functions, i.e. for any coordinated neighborhood $U\subset M$ and for any $L^2$ holomorphic function $f$ on $U\setminus X$, there exists an $L^2$ holomorphic function $\tilde{f}$ on $U$ such that $\tilde{f}|_{U\setminus X}=f$ with the same $L^2$ norm;
	
	(2). $Z$ is an analytic subset of $M$ and $M\setminus(X\cup Z)$ is a weakly pseudoconvex K\"{a}hler manifold.
	
	Let $(M,X,Z)$ be a triple satisfying condition $(A)$. Let $K_M$ be the canonical line bundle on $M$. Let $\psi$ be a plurisubharmonic function on $M$ such that $\{\psi<-t\}\setminus(X\cup Z)$ is a weakly pseudoconvex K\"{a}hler manifold for any $t\in\mathbb{R}$, and let $\varphi$ be a Lebesgue measurable function on $M$ such that $\varphi+\psi$ is a plurisubharmonic function on $M$. Denote $T=-\sup_M\psi$.
	
	Recall the concept of ``gain'' in \cite{GMY}. A positive measurable function $c$ on $(T,+\infty)$ is in the class $\mathcal{P}_{T,M}$ if the following two statements hold:
	
	(1). $c(t)e^{-t}$ is decreasing with respect to $t$;
	
	(2). there is a closed subset $E$ of $M$ such that $E\subset Z\cap\{\psi=-\infty\}$ and for any compact subset $K\subset M\setminus E$, $e^{-\varphi}c(-\psi)$ has a positive lower bound on $K$.
	
	Let $Z_0$ be a subset of $M$ such that $Z_0\cap\text{Supp} (\mathcal{O}/\mathcal{I}(\varphi+\psi))\neq\emptyset$. Let $U\supset Z_0$ be an open subset of $M$, and let $f$ be a holomorphic $(n,0)$ form on $U$. Let $\mathcal{F}_{z_0}\supset\mathcal{I}(\varphi+\psi)_{z_0}$ be an ideal of $\mathcal{O}_{z_0}$ for any $z_0\in Z_0$.
	
	Denote that
	\begin{flalign}
		\begin{split}
			G(t;c):=\inf\bigg\{\int_{\{\psi<-t\}}|\tilde{f}|^2e^{-\varphi}c(-\psi) : \tilde{f}\in H^0(\{\psi<-t\}, \mathcal{O}(K_M))&\\
			\& (\tilde{f}-f)\in H^0(Z_0, (\mathcal{O}(K_M)\otimes\mathcal{F})|_{Z_0})\bigg\}&,
		\end{split}
	\end{flalign}
	and
	\begin{flalign}
		\begin{split}
			\mathcal{H}^2(c,t):=\bigg\{\tilde{f} : \int_{\{\psi<-t\}}|\tilde{f}|^2e^{-\varphi}c(-\psi)<+\infty, \tilde{f}\in H^0(\{\psi<-t\}, \mathcal{O}(K_M))&\\
			\& (\tilde{f}-f)\in H^0(Z_0, (\mathcal{O}(K_M)\otimes\mathcal{F})|_{Z_0})\bigg\}&,
		\end{split}
	\end{flalign}
	where $t\in [T,+\infty)$, and $c$ is a nonnegative measurable function on $(T,+\infty)$. Here $|\tilde{f}|^2:=\sqrt{-1}^{n^2}\tilde{f}\wedge\bar{\tilde{f}}$ for any $(n,0)$ form $\tilde{f}$. And $(\tilde{f}-f)\in H^0(Z_0, (\mathcal{O}(K_M)\otimes\mathcal{F})|_{Z_0})$ means that $(\tilde{f}-f,z_0)\in (\mathcal{O}(K_M)\otimes\mathcal{F})_{z_0}$ for any $z_0\in Z_0$. If there is no holomorphic $(n,0)$ form $\tilde{f}$ on $\{\psi<-t\}$ satisfying $(\tilde{f}-f)\in H^0(Z_0, (\mathcal{O}(K_M)\otimes\mathcal{F})|_{Z_0})$, we set $G(t;c)=+\infty$. We may denote $G(t;c)$ by $G(t)$ if there are no misunderstandings.
	
	In \cite{GMY}, Guan-Mi-Yuan obtained the following concavity of $G(t;c)$.
	
	\begin{theorem}\label{Concave}
		Let $c\in\mathcal{P}_{T,M}$ such that $\int_T^{+\infty}c(s)e^{-s}ds<+\infty$. If there exists $t\in [T,+\infty)$ satisfying that $G(t)<+\infty$, then $G(h^{-1}(r))$ is concave with respect to $r\in (0,\int_T^{+\infty}c(t)e^{-t}dt)$, $\lim\limits_{t\rightarrow T+0}G(t)=G(T)$ and $\lim\limits_{t\rightarrow +\infty}G(t)=0$, where $h(t)=\int_T^{+\infty}c(t_1)e^{-t_1}dt_1$.
	\end{theorem}
	
	Guan-Mi-Yuan also obtained the following corollary of Theorem \ref{Concave}, which is a neccessary condition for the concavity degenerating to linearity.
	
	\begin{lemma}[\cite{GMY}]\label{linear}
		Let $c(t)\in\mathcal{P}_{T,M}$ such that $\int_T^{+\infty}c(s)e^{-s}ds<+\infty$. If $G(t)\in (0,+\infty)$ for some $t\geq T$ and $G(h^{-1}(r))$ is linear with respect to $r\in (0,\int_T^{+\infty}c(s)e^{-s}ds)$, where $h(t)=\int_t^{+\infty}c(l)e^{-l}dl$, then there exists a unique holomorphic $(n,0)$ form $F$ on $M$ satisfying $(F-f)\in H^0(Z_0,(\mathcal{O}(K_M)\otimes\mathcal{F})|_{Z_0})$, and $G(t;c)=\int_{\{\psi<-t\}}|F|^2e^{-\varphi}c(-\psi)$ for any $t\geq T$.
		
		Furthermore, we have
		\begin{equation}
			\int_{\{-t_2\leq\psi<-t_1\}}|F|^2e^{-\varphi}a(-\psi)=\frac{G(T_1;c)}{\int_{T_1}^{+\infty}c(t)e^{-t}dt}\int_{t_1}^{t_2}a(t)e^{-t}dt,
		\end{equation}
		for any nonnegative measurable function $a$ on $(T,+\infty)$, where $T\leq t_1<t_2\leq +\infty$.
		
		Especially, if $\mathcal H^2(\tilde{c},t_0)\subset\mathcal H^2(c,t_0)$ for some $t_0\geq T$, where $\tilde{c}$ is a nonnegative measurable function on $(T,+\infty)$, we have
		\begin{equation}
			G(t_0;\tilde{c})=\int_{\{\psi<-t_0\}}|F|^2e^{-\varphi}\tilde{c}(-\psi)=\frac{G(T_1;c)}{\int_{T_1}^{+\infty}c(s)e^{-s}ds}\int_{t_0}^{+\infty} \tilde{c}(s)e^{-s}ds.
		\end{equation}		
	\end{lemma}
	
	\begin{remark}\label{ctildec}(\cite{GMY})
		Let $c(t)\in\mathcal{P}_{T,M}$. If $\mathcal{H}^2(\tilde{c},t_1)\subset\mathcal{H}^2(c,t_1)$, then $\mathcal{H}^2(\tilde{c},t_2)\subset\mathcal{H}^2(c,t_2)$, where $t_1>t_2>T$. In the following, we give some sufficient conditions of $\mathcal{H}^2(\tilde{c},t_0)\subset\mathcal{H}^2(c,t_0)$ for $t_0>T$:
		
		(1). $\tilde{c}\in\mathcal{P}_{T,M}$ and $\lim\limits_{t\rightarrow+\infty}\frac{\tilde{c}(t)}{c(t)}>0$. Especially, $\tilde{c}\in\mathcal{P}_{T,M}$, $c$ and $\tilde{c}$ are smooth on $(T,+\infty)$ and $\frac{d}{dt}(\log\tilde{c}(t))\geq\frac{d}{dt}(\log c(t))$;
		
		(2). $\tilde{c}\in\mathcal{P}_{T,M}$, $\mathcal{H}^2(c,t_0)\neq\emptyset$ and there exists $t>t_0$ such that $\{\psi<-t\}\subset\subset\{\psi<-t_0\}$, $\{z\in\overline{\{\psi<-t\}} : \mathcal{I}(\varphi+\psi)_z\neq\mathcal{O}_z\}\subset Z_0$ and $\mathcal{F}|_{\overline{\{\psi<-t\}}}=\mathcal{I}(\varphi+\psi)|_{\overline{\{\psi<-t\}}}$.
	\end{remark}

	\subsection{Open Riemann surface case}
	\
	
	Let $M=\Omega$  be an open Riemann surface, which admits a nontrivial Green function $G_{\Omega}$.
	Let $P:\Delta\rightarrow\Omega$ be the universal covering from unit disc $\Delta$ to $\Omega$.
	It is known that for any harmonic function $u$ on $\Omega$,
	there exists a $\chi_{u}$ (the  character associate to $u$) and a multiplicative function $f_u\in\mathcal{O}^{\chi_{u}}(\Omega)$,
	such that $|f_u|=P^{*}e^{u}$.
	Let $z_0\in \Omega$.
	Recall that for the Green function $G_{\Omega}(z,z_0)$,
	there exist a $\chi_{z_0}$ and a multiplicative function $f_{z_0}\in\mathcal{O}^{\chi_{z_0}}(\Omega)$, such that $|f_{z_0}(z)|=P^{*}e^{G_{\Omega}(z,z_0)}$ .
	
	Let $Z_0:=\{z_1,z_2,...,z_m\}\subset\Omega$ be a subset of $\Omega$ satisfying that $z_j\not=z_k$ for any $j\not=k$.
	Let $w_j$ be a local coordinate on a neighborhood $V_{z_j}\Subset\Omega$ of $z_j$ satisfying $w_j(z_j)=0$ for $j\in\{1,2,...,m\}$, where $V_{z_j}\cap V_{z_k}=\emptyset$ for any $j\not=k$. Denote that $V_0:=\cup_{1\le j\le n}V_{z_j}$.
	
	Let $f$ be a holomorphic $(1,0)$ form on $V_0$, and let $f=f_1dw_j$ on $V_{z_j}$, where $f_1$ is a holomorphic function on $V_0$. Let $\psi$ be a negative subharmonic function on $\Omega$, and let $\varphi$ be a Lebesgue measurable function on $\Omega$ such that $\varphi+\psi$ is subharmonic on $\Omega$.
	
	The following Theorem gives a characterization of the concavity of $G(h^{-1}(r))$ degenerating to linearity.
	\begin{theorem}[\cite{GY-concavity3}]
		\label{thm:m-points}
		Let $c\in\mathcal{P}_{0,\Omega}.$
		Assume that $G(0)\in(0,+\infty)$ and $(\psi-2p_jG_{\Omega}(\cdot,z_j))(z_j)>-\infty$ for  $j\in\{1,2,..,m\}$, where $p_j=\frac{1}{2}v(dd^c(\psi),z_j)>0$. Then $G(h^{-1}(r))$ is linear with respect to $r$ if and only if the following statements hold:
		
		$(1)$ $\psi=2\sum_{1\le j\le m}p_jG_{\Omega}(\cdot,z_j)$;
		
		$(2)$ $\varphi+\psi=2\log|g|+2\sum_{1\le j\le m}G_{\Omega}(\cdot,z_j)+2u$ and $\mathcal{F}_{z_j}=\mathcal{I}(\varphi+\psi)_{z_j}$ for any $j\in\{1,2,...,m\}$, where $g$ is a holomorphic function on $\Omega$ such that $ord_{z_j}(g)=ord_{z_j}(f_1)$ for any $j\in\{1,2,...,m\}$ and $u$ is a harmonic function on $\Omega$;
		
		$(3)$ $\prod_{1\le j\le m}\chi_{z_j}=\chi_{-u}$, where $\chi_{-u}$ and $\chi_{z_j}$ are the  characters associated to the functions $-u$ and $G_{\Omega}(\cdot,z_j)$ respectively;
		
		$(4)$
		\[\lim_{z\rightarrow z_k}\frac{f}{gP_*\left(f_u(\prod_{1\le j\le m}f_{z_j})(\sum_{1\le j\le m}p_{j}\frac{d{f_{z_{j}}}}{f_{z_{j}}})\right)}=c_0\]
		for any $k\in\{1,2...,m\}$, where $c_0\in\mathbb{C}\backslash\{0\}$ is a constant independent of $k$.
	\end{theorem}
	
	\begin{remark}[\cite{GY-concavity3}]\label{rem-finite-1d}
		When the four statements in Theorem \ref{thm:m-points} hold,
		\begin{equation*}
			c_0gP_*\left(f_u(\prod_{1\leq l\leq m}f_{z_l})(\sum_{1\leq l\leq m}p_l\dfrac{d{f_{z_{l}}}}{f_{z_{l}}})\right)
		\end{equation*}
		is the unique holomorphic $(n,0)$ form $F$ on $M$ such that $(F-f,z_j)\in (\mathcal{O}(K_{\Omega}))_{z_j}\otimes\mathcal{F}_{z_j}$ for any $j\in\{1,2,\ldots,m\}$ and
		\begin{equation*}
			G(t)=\int_{\{\psi<-t\}}|F|^2e^{-\varphi}c(-\psi)=\left(\int_t^{+\infty}c(s)e^{-s}ds\right)\sum_{j=1}^m\frac{2\pi |a_j|^2e^{-2u(z_j)}}{p_j|d_j|^2}
		\end{equation*}
		for any $t\geq 0$, where $f_u$ is a holomorphic function on $\Delta$ such that $|f_u|=P^*(e^u)$, $f_{z_j}$ is a holomorphic function on $\Delta$ such that $|f_{z_j}|=P^*(e^{G_{\Omega}(\cdot,z_j)})$ for any $j\in\{1,2,\ldots,m\}$, $a_j:=\lim\limits_{z\rightarrow z_j}\frac{f_1}{w^{k_j}}(z_j)$
		and $d_j:=\lim\limits_{z\rightarrow z_j}\frac{g}{w^{k_j}}(z_j)$, here $k_j:=ord_{z_j}(g)=ord_{z_j}(f_1)$.
	\end{remark}
	
	\begin{remark}[\cite{GY-concavity3}]\label{r:chi}
		For any $\{z_1,z_2,..,z_m\}$, there exists a harmonic function $u$ on $\Omega$ such that $\prod_{1\le j\le m}\chi_{z_j}=\chi_{-u}$. In fact, as $\Omega$ is an open Riemann surface, then there exists a holomorphic function $\tilde{f}$ on $\Omega$ satisfying that $u:=\log|\tilde{f}|-\sum_{1\le j\le m}G_{\Omega}(\cdot,z_j)$ is harmonic on $\Omega$, which implies that $\prod_{1\le j\le m}\chi_{z_j}=\chi_{-u}$.
	\end{remark}
	
	We recall a  characterization of the holding of equality in optimal jets $L^2$ extension problem from finite points to open Riemann surfaces.
	\begin{theorem}[\cite{GY-concavity3}]\label{c:L2-1d-char}
		Let $k_j$ be a nonnegative integer for any $j\in\{1,2,...,m\}$. Let $\psi$ be a negative  subharmonic function on $\Omega$ satisfying that   $\frac{1}{2}v(dd^{c}\psi,z_j)=p_j>0$ for any $j\in\{1,2,...,m\}$. Let $\varphi$ be a Lebesgue measurable function on $\Omega$  such that $\varphi+\psi$ is subharmonic on $\Omega$, $\frac{1}{2}v(dd^c(\varphi+\psi),z_j)=k_j+1$ and $\alpha_j:=(\varphi+\psi-2(k_j+1)G_{\Omega}(\cdot,z_j))(z_j)>-\infty$ for any $j$. Let $c(t)$ be a positive measurable function on $(0,+\infty)$ satisfying $c(t)e^{-t}$ is decreasing on $(0,+\infty)$ and $\int_{0}^{+\infty}c(s)e^{-s}ds<+\infty$. Let $a_j$ be a constant for any $j$.
		
		Let $f$ be a holomorphic $(1,0)$ form on $V_0$ satisfying that $f=a_jw_j^{k_j}dw_j$ on $V_{z_j}$. Then there exists a holomorphic $(1,0)$ form $F$ on $\Omega$ such that $(F-f,z_j)\in(\mathcal{O}(K_{\Omega})\otimes\mathcal{I}(2(k_j+1)G_{\Omega}(\cdot,z_j)))_{z_j}$ and
		\begin{equation}
			\label{eq:210902a}
			\int_{\Omega}|F|^2e^{-\varphi}c(-\psi)\leq\left(\int_0^{+\infty}c(s)e^{-s}ds\right)\sum_{1\le j\le m}\frac{2\pi|a_j|^2e^{-\alpha_j}}{p_jc_{\beta}(z_j)^{2(k_j+1)}}.
		\end{equation}
		
		Moreover, equality $(\int_0^{+\infty}c(s)e^{-s}ds)\sum_{1\le j\le m}\frac{2\pi|a_j|^2e^{-\alpha_j}}{p_jc_{\beta}(z_j)^{2(k_j+1)}}=\inf\big\{\int_{\Omega}|\tilde{F}|^2e^{-\varphi}c(-\psi):\tilde{F}$ is a holomorphic $(1,0)$ form on $\Omega$ such that $(\tilde{F}-f,z_j)\in(\mathcal{O}(K_{\Omega})\otimes\mathcal{I}(2(k_j+1)G_{\Omega}(\cdot,z_j)))_{z_j}$ for any $j\big\}$ holds if and only if the following statements hold:

		$(1)$ $\psi=2\sum_{1\le j\le m}p_jG_{\Omega}(\cdot,z_j)$;
		
		$(2)$ $\varphi+\psi=2\log|g|+2\sum_{1\le j\le m}(k_j+1)G_{\Omega}(\cdot,z_j)+2u$, where $g$ is a holomorphic function on $\Omega$ such that $g(z_j)\not=0$ for any $j\in\{1,2,...,m\}$ and $u$ is a harmonic function on $\Omega$;
		
		$(3)$ $\prod_{1\le j\le m}\chi_{z_j}^{k_j+1}=\chi_{-u}$, where $\chi_{-u}$ and $\chi_{z_j}$ are the  characters associated to the functions $-u$ and $G_{\Omega}(\cdot,z_j)$ respectively;
		
		$(4)$
		\[\lim_{z\rightarrow z_k}\frac{f}{gP_*\left(f_u(\prod_{1\le j\le m}f_{z_j}^{k_j+1})(\sum_{1\le j\le m}p_{j}\frac{d{f_{z_{j}}}}{f_{z_{j}}})\right)}=c_0\]
		for any $k\in\{1,2...,m\}$, where $c_0\in\mathbb{C}\backslash\{0\}$ is a constant independent of $k$.
	\end{theorem}
	
	\begin{remark}[\cite{GY-concavity3}]\label{rem:1.2}
		When the four statements in Theorem \ref{c:L2-1d-char} hold,
		\[c_0gP_*\left(f_u(\prod_{1\le j\le m}f_{z_j}^{k_j+1})(\sum_{1\le j\le m}p_{j}\frac{d{f_{z_{j}}}}{f_{z_{j}}})\right)\]
		is the unique holomorphic $(1,0)$ form $F$ on $\Omega$ such that $(F-f,z_j)\in(\mathcal{O}(K_{\Omega})\otimes\mathcal{I}(2(k_j+1)G_{\Omega}(\cdot,z_j)))_{z_j}$ and
		\begin{equation*}
			\int_{\Omega}|F|^2e^{-\varphi}c(-\psi)\leq\left(\int_0^{+\infty}c(s)e^{-s}ds\right)\sum_{1\le j\le m}\frac{2\pi|a_j|^2e^{-\alpha_j}}{p_jc_{\beta}(z_j)^{2(k_j+1)}}.
		\end{equation*}
	\end{remark}

	Let $Z_0:=\{z_j:j\in\mathbb{Z}_{\ge1}\}\subset\Omega$ be a  discrete set of infinite points.
	Let $w_j$ be a local coordinate on a neighborhood $V_{z_j}\Subset\Omega$ of $z_j$ satisfying $w_j(z_j)=0$ for $j\in\mathbb{Z}_{\ge1}$, where $V_{z_j}\cap V_{z_k}=\emptyset$ for any $j\not=k$. Denote that $V_0:=\cup_{j\in\mathbb{Z}_{\ge1}}V_{z_j}$.
	Let $f$ be a holomorphic $(1,0)$ form on $V_0$, and let $f=f_1dw_j$ on $V_{z_j}$, where $f_1$ is a holomorphic function on $V_0$. Let $\psi$ be a negative subharmonic function on $\Omega$, and let $\varphi$ be a Lebesgue measurable function on $\Omega$ such that $\varphi+\psi$ is subharmonic on $\Omega$.
	
	The following result gives a necessary condition for $G(h^{-1}(r))$ is linear.
	
	\begin{proposition}[\cite{GY-concavity3}]
		Let $c\in\mathcal{P}_{0,\Omega}.$	\label{p:infinite}
		Assume that $G(0)\in(0,+\infty)$ and $(\psi-2p_jG_{\Omega}(\cdot,z_j))(z_j)>-\infty$ for  $j\in\mathbb{Z}_{\ge1}$, where $p_j=\frac{1}{2}v(dd^c(\psi),z_j)>0$. Assume that $G(h^{-1}(r))$ is linear with respect to $r$. Then the following statements hold:
		
		$(1)$ $\psi=2\sum_{j\in\mathbb{Z}_{\ge1}}p_jG_{\Omega}(\cdot,z_j)$;

		$(2)$ $\varphi+\psi=2\log|g|$ and $\mathcal{F}_{z_j}=\mathcal{I}(\varphi+\psi)_{z_j}$ for any $j\in\mathbb{Z}_{\ge1}$, where $g$ is a holomorphic function on $\Omega$ such that $ord_{z_j}(g)=ord_{z_j}(f_1)+1$ for any $j\in\mathbb{Z}_{\ge1}$;
		
		$(3)$  $\frac{p_j}{ord_{z_j}g}\lim_{z\rightarrow z_j}\frac{dg}{f}=c_0$ for any $j\in\mathbb{Z}_{\ge1}$, where $c_0\in\mathbb{C}\backslash\{0\}$ is a constant independent of $j$;
		
		$(4)$ $\sum_{j\in\mathbb{Z}_{\ge1}}p_j<+\infty$.
	\end{proposition}

	\subsection{Basic properties of the Green functions}
	\
	
	In this subection, we recall some basic properties of the Green functions. Let $\Omega$ be an open Riemann surface, which admits a nontrivial Green function $G_{\Omega}$, and let $z_0\in\Omega$.
	
	\begin{lemma}[see \cite{S-O69}, see also \cite{Tsuji}] 	\label{l:green-sup}
		Let $w$ be a local coordinate on a neighborhood of $z_0$ satisfying $w(z_0)=0$.  $G_{\Omega}(z,z_0)=\sup_{v\in\Delta_{\Omega}^*(z_0)}v(z)$, where $\Delta_{\Omega}^*(z_0)$ is the set of negative subharmonic function on $\Omega$ such that $v-\log|w|$ has a locally finite upper bound near $z_0$. Moreover, $G_{\Omega}(\cdot,z_0)$ is harmonic on $\Omega\backslash\{z_0\}$ and $G_{\Omega}(\cdot,z_0)-\log|w|$ is harmonic near $z_0$.
	\end{lemma}

	\begin{lemma}[see \cite{GY-concavity3}]
		\label{l:green-sup2}
		Let $K=\{z_j:j\in\mathbb{Z}_{\ge1}\,\&\,j<\gamma \}$ be a discrete subset of $\Omega$, where $\gamma\in\mathbb{Z}_{>1}\cup\{+\infty\}$. Let $\psi$ be a negative subharmonic function on $\Omega$ such that $\frac{1}{2}v(dd^c\psi,z_j)\geq p_j$ for any $j$, where $p_j>0$ is a constant. Then $2\sum_{1\le j< \gamma}p_jG_{\Omega}(\cdot,z_j)$ is a subharmonic function on $\Omega$ satisfying that $2\sum_{1\le j<\gamma }p_jG_{\Omega}(\cdot,z_j)\ge\psi$ and $2\sum_{1\le j<\gamma }p_jG_{\Omega}(\cdot,z_j)$ is harmonic on $\Omega\backslash K$.
	\end{lemma}
	
	\begin{lemma}[see \cite{GY-concavity}]\label{l:G-compact}
		For any  open neighborhood $U$ of $z_0$, there exists $t>0$ such that $\{G_{\Omega}(z,z_0)<-t\}$ is a relatively compact subset of $U$.
	\end{lemma}
	
	\begin{lemma}[see \cite{GY-concavity3}]
		\label{l:psi=G}
		Let $\Omega$ be an open Riemann surface which admits a nontrivial Green function $G_{\Omega}$. Let $Z_0'=\{z_j: j\in\mathbb{Z}_+ \& j<\gamma\}$ be a discrete subset of $\Omega$, where $\gamma\in \mathbb{Z}_+\cup\{+\infty\}$. Let $\psi$ be a negative plurisubharmonic function on $\Omega$ satisfying $\frac{1}{2}v(dd^c\psi,z_0)\geq p_j>0$ for any $j$, where $p_j$ is a constant. Assume that $\psi\not\equiv 2\sum_{1\le k<\gamma}p_jG_{\Omega}(\cdot,z_j)$. Let $l(t)$ be a positive Lebesgue measurable function on $(0,+\infty)$ satisfying $l$ is decreasing on $(0,+\infty)$ and $\int_0^{+\infty}l(t)dt<+\infty$. Then there exists a Lebesgue measurable subset $V$ of $M$ such that $l(-\psi(z))<l(-2\sum_{1\le k<\gamma}p_jG_{\Omega}(z,z_j))$ for any  $z\in V$ and $\mu(V)>0$, where $\mu$ is the Lebesgue measure on $\Omega$.
	\end{lemma}
	
	\begin{lemma}[see \cite{GY-concavity3}]
		\label{green-approx} There exists a sequence of open Riemann surfaces $\{\Omega_l\}_{l\in\mathbb{Z}^+}$ such that $z_0\in\Omega_l\Subset\Omega_{l+1}\Subset\Omega$, $\cup_{l\in\mathbb{Z}^+}\Omega_l=\Omega$, $\Omega_l$ has a smooth boundary $\partial\Omega_l$ in $\Omega$  and $e^{G_{\Omega_l}(\cdot,z_0)}$ can be smoothly extended to a neighborhood of $\overline{\Omega_l}$ for any $l\in\mathbb{Z}^+$, where $G_{\Omega_l}$ is the Green function of $\Omega_l$. Moreover, $\{{G_{\Omega_l}}(\cdot,z_0)-G_{\Omega}(\cdot,z_0)\}$ is decreasingly convergent to $0$ on $\Omega$.
	\end{lemma}
	
	\subsection{Some other required results}
	\
	
	\begin{lemma}[\cite{GMY}]\label{dbarequa}
	Let $c$ be a positive function on $(0,+\infty)$, such that $\int_{0}^{+\infty}c(t)e^{-t}dt<+\infty$ and $c(t)e^{-t}$ is decreasing on $(0,+\infty)$.
	Let $B\in(0,+\infty)$ and $t_{0}\geq 0$ be arbitrarily given.
	Let $M$ be an $n-$dimensional weakly pseudoconvex K\"ahler  manifold.
	Let $\psi<0$ be a plurisubharmonic function on $M$.
	Let $\varphi$ be a plurisubharmonic function on $M$.
	Let $F$ be a holomorphic $(n,0)$ form on $\{\psi<-t_{0}\}$,
	such that
	\begin{equation}
		\int_{K\cap\{\psi<-t_{0}\}}|F|^{2}<+\infty
	\end{equation}
	for any compact subset $K$ of $M$, and
	\begin{equation}
		\int_{M}\frac{1}{B}\mathbb{I}_{\{-t_{0}-B<\psi<-t_{0}\}}|F|^{2}e^{-\varphi}\leq C<+\infty.
	\end{equation}
	Then there exists a
	holomorphic $(n,0)$ form $\tilde{F}$ on $M$, such that
	\begin{equation}
		\begin{split}
			\int_{M}&|\tilde{F}-(1-b_{t_0,B}(\psi))F|^{2}e^{-\varphi+v_{t_0,B}(\psi)}c(-v_{t_0,B}(\psi))\leq C\int_{0}^{t_{0}+B}c(t)e^{-t}dt.
		\end{split}
	\end{equation}
	where $b_{t_0,B}(t)=\int_{-\infty}^{t}\frac{1}{B}\mathbb{I}_{\{-t_{0}-B< s<-t_{0}\}}ds$ and
	$v_{t_0,B}(t)=\int_{-t_0}^{t}b_{t_0,B}(s)ds-t_0$.
	\end{lemma}
	
	Follow the assumptions of Lemma \ref{Concave} in the following two lemmas.
	\begin{lemma}[\cite{GMY}]\label{G(t)=0}
		The following three statements are equivalent:
		
		(1). $f\in H^0(Z_0,(\mathcal{O}(K_M)\otimes\mathcal{F})|_{Z_0})$;
		
		(2). $G(t)=0$ for some $t\geq T$;
		
		(3). $G(t)=0$ for any $t\geq T$.
	\end{lemma}

	\begin{lemma}[\cite{GMY}]\label{F_t}
		Assume that $G(t)<+\infty$ for some $t\in [T,+\infty)$. Then there exists a unique holomorphic $(n,0)$ form $F_t$ on $\{\psi<-t\}$ satisfying
		\begin{equation}
			\int_{\{\psi<-t\}}|F_t|^2e^{-\varphi}c(-\psi)=G(t)
		\end{equation}
		and $(F_t-f)\in H^0(Z_0,(\mathcal{O}(K_M)\otimes\mathcal{F})|_{Z_0})$.
		
		Furthermore, for any holomorphic $(n,0)$ form $\hat{F}$ on $\{\psi<-t\}$ satisfying
		\begin{equation}
			\int_{\{\psi<-t\}}|\hat{F}|^2e^{-\varphi}c(-\psi)<+\infty
		\end{equation}
		and $(\hat{F}-f)\in H^0(Z_0,(\mathcal{O}(K_M)\otimes\mathcal{F})|_{Z_0})$, we have the following equality,
		\begin{flalign}
			\begin{split}
				&\int_{\{\psi<-t\}}|F_t|^2e^{-\varphi}c(-\psi)+\int_{\{\psi<-t\}}|\hat{F}-F_t|^2e^{-\varphi}c(-\psi)\\
				=&\int_{\{\psi<-t\}}|\hat{F}|^2e^{-\varphi}c(-\psi).
			\end{split}	
		\end{flalign}
	\end{lemma}
	
	We also need the following lemmas.
	\begin{lemma}[see \cite{G-R}]\label{module}
		Let $N$ be a submodule of $\mathcal{O}_{\mathbb{C}^n,o}^q$, $q\in\mathbb{Z}_+$. Let $\{f_j\}\subset\mathcal{O}_{\mathbb{C}^n}(U)^q$ be a sequence of $q-$tuples holomorphic functions in an open neighborhood $U$ of the origin $o$. Assume that $\{f_j\}$ converges uniformly in $U$ towards a $q-$tuples $f\in\mathcal{O}_{\mathbb{C}^n,o}^q$, and assume furthermore that all the germs $(f_j,o)$ belong to $N$. Then $(f,o)\in N$.
	\end{lemma}
	
	\begin{lemma}[see \cite{guan-zhou13ap}]\label{s_K}
		Let $M$ be a complex manifold. Let $S$ be an analytic subset of $M$. Let $\{g_j\}_{j=1,2,\ldots}$ be a sequence of nonnegative Lebesgue measurable functions on $M$, which satisfies that $g_j$ are almost everywhere convergent to $g$ on $M$ when $j\rightarrow +\infty$, where $g$ is a nonnegative Lebesgue measurable function on $M$. Assume that for any compact subset $K$ of $M\setminus S$, there exist $s_K\in (0,+\infty)$ and $C_K\in (0,+\infty)$ such that
		\begin{equation}
			\int_Kg_j^{-s_K}dV_M\leq C_K
		\end{equation}
		for any $j$, where $dV_M$ is a continuous volume form on $M$.
		
		Let $\{F_j\}_{j=1,2,\ldots}$ be a sequence of holomorphic $(n,0)$ form on $M$. Assume that $\liminf\limits_{j\rightarrow +\infty}\int_M|F_j|^2g_j\leq C$, where $C$ is a positive constant. Then there exists a subsequence $\{F_{j_l}\}_{l=1,2,\ldots}$, which satisfies that $\{F_{j_l}\}$ is uniformly convergent to a holomorphic $(n,0)$ form $F$ on $M$ on any compact subset of $M$ when $l\rightarrow +\infty$, such that
		\begin{equation}
			\int_M|F|^2g\leq C.
		\end{equation}
	\end{lemma}
	
	\begin{lemma}[\cite{FN1980}]\label{l:FN1} 
		Let $X$ be a Stein manifold and $\varphi \in PSH(X)$. Then there exists a sequence
		$\{\varphi_{n}\}_{n=1,\cdots}$ of smooth strongly plurisubharmonic functions such that
		$\varphi_{n} \downarrow \varphi$.
	\end{lemma}
	
	\begin{lemma}[see \cite{GY-concavity3}]\label{tildecincreasing}
		If $c(t)$ is a positive measurable function on $(T,+\infty)$ such
		that $c(t)e^{-t}$ is decreasing on $(T,+\infty)$ and $\int_{T_1}^{+\infty}c(s)e^{-s}ds<+\infty$ for some $T_1>T$, then there exists a positive measurable function $\tilde{c}$ on $(T,+\infty)$ satisfying the
		following statements:
		
		(1). $\tilde{c}\geq c$ on $(T,+\infty)$;
		
		(2). $\tilde{c}(t)e^{-t}$ is strictly decreasing on $(T,+\infty)$ and $\tilde{c}$ is increasing on $(a,+\infty)$,	where $a>T$ is a real number;
		
		(3). $\int_{T_1}^{+\infty}\tilde{c}(s)e^{-s}ds<+\infty$.
		
		Moreover, if $\int_T^{+\infty}c(s)e^{-s}ds<+\infty$ and $c\in\mathcal{P}_T$, we can choose $\tilde{c}$ satisfying the above conditions, $\int_T^{+\infty}\tilde{c}(s)e^{-s}ds<+\infty$ and $\tilde{c}\in\mathcal{P}_T$.
	\end{lemma}
	
	\begin{lemma}
		\label{c(t)e^{-at}}
		Let $c(t)$ be a positive measurable function on $(0,+\infty)$, and let $a\in\mathbb{R}$. Assume that $\int_{t}^{+\infty}c(s)e^{-s}ds\in(0,+\infty)$ when $t$ near $+\infty$. Then we have
		
		$(1)$ $\lim_{t\rightarrow+\infty}\frac{\int_{t}^{+\infty}c(s)e^{-as}ds}{\int_t^{+\infty}c(s)e^{-s}ds}=1$ if and only if $a=1$;
		
		$(2)$ $\lim_{t\rightarrow+\infty}\frac{\int_{t}^{+\infty}c(s)e^{-as}ds}{\int_t^{+\infty}c(s)e^{-s}ds}=0$ if and only if $a>1$;
		
		$(3)$ $\lim_{t\rightarrow+\infty}\frac{\int_{t}^{+\infty}c(s)e^{-as}ds}{\int_t^{+\infty}c(s)e^{-s}ds}=+\infty$ if and only if $a<1$.
	\end{lemma}
	\begin{proof}
		If $a=1$, it clear that $\lim_{t\rightarrow+\infty}\frac{\int_{t}^{+\infty}c(s)e^{-as}ds}{\int_t^{+\infty}c(s)e^{-s}ds}=1$.
		
		If $a>1$, then $c(s)e^{-as}\le e^{(1-a)s_0} c(s)e^{-s}$ for $s\ge s_0>0$, which implies that
		\[\limsup_{t\rightarrow+\infty}\frac{\int_{t}^{+\infty}c(s)e^{-as}ds}{\int_t^{+\infty}c(s)e^{-s}ds}\le e^{(1-a)s_0}.\]
		Let $s_0\rightarrow+\infty$, we have \[\lim_{t\rightarrow+\infty}\frac{\int_{t}^{+\infty}c(s)e^{-as}ds}{\int_t^{+\infty}c(s)e^{-s}ds}=0.\]
		
		If $a<1$, then $c(s)e^{-as}\ge e^{(1-a)s_0} c(s)e^{-s}$ for $a>s_0>0$, which implies that \[\liminf_{t\rightarrow+\infty}\frac{\int_{t}^{+\infty}c(s)e^{-as}ds}{\int_t^{+\infty}c(s)e^{-s}ds}\ge e^{(1-a)s_0}.\]
		Let $s_0\rightarrow+\infty$, we have \[\lim_{t\rightarrow+\infty}\frac{\int_{t}^{+\infty}c(s)e^{-as}ds}{\int_t^{+\infty}c(s)e^{-s}ds}=+\infty.\]
	\end{proof}

	\section{some required lemmas}
	Before the proofs of the main results, we need some lemmas, and we present them in the following.
	
	\subsection{Some lemmas in the local case}
	\
	We need the following lemmas in the local cases.
	
	Let $X=\Delta^m$,$Y=\Delta^{n-m}$ be the unit polydiscs in $\mathbb{C}^m$ and $\mathbb{C}^{n-m}$. Let $M=X\times Y$. Let $\pi_1$, $\pi_2$ be the natural projections from $M$ to $X$ and $Y$. Let $\psi_1$ be a negative plurisubharmonic function on $X$ such that $\psi_1(o)=-\infty$, where $o$ is the origin of $\mathbb{C}^m$, and let $\varphi_1$ be a Lebesgue measurable function on $X$ such that $\varphi_1+\psi_1$ is plurisubharmonic on $X$. Let $\varphi_2$ be a plurisubharmonic function on $Y$. Denote that $\psi:=\pi_1^*(\psi_1)$, $\varphi:=\pi_1^*(\varphi_1)+\pi_2^*(\varphi_2)$. Let $f(z,w)$ be a holomorphic function on $M$, where $z=(z_1,\ldots,z_m)\in X$, $w=(w_1,\ldots,w_{n-m})\in Y$.
	
	\begin{lemma}\label{local-germ}
		Assume that $f\in \mathcal{I}(\varphi+\psi)_{(o,w)}$ for any $(o,w)\in Y_o$, then $(f(\cdot,w),o)\in\mathcal{I}(\varphi_1+\psi_1)_o$ on $X$ for any $w\in Y$, and $(f(o,\cdot),w)\in\mathcal{I}(\varphi_2)_w$ on $Y$ for any $w\in Y$.
	\end{lemma}
	
	\begin{proof}
		Since $f\in \mathcal{I}(\varphi+\psi)_{(o,w)}$, we can find some $r>0$ such that
		\begin{equation*}
			\int_{\Delta^n((o,w),r)}|f|^2e^{-\varphi-\psi}<+\infty,
		\end{equation*}
		where $\Delta^n((o,w),r)$ is the polydisc centered on $(o,w)$ with radius $r$. Then according to the Fubini's Theorem, we have
		\begin{flalign*}
			\begin{split}
				&\int_{z\Delta^m(o,r)}|f(z,w)|^2e^{-\varphi_1-\psi_1}\\
				\leq&\frac{1}{(\pi r^2)^{n-m}}\int_{w'\in\Delta^{n-m}(w,r)}\left(\int_{z\in\Delta^m(o,r)}|f(z,w')|^2e^{-\varphi_1-\psi_1}\right)\\
				\leq&\frac{C}{(\pi r^2)^{n-m}}\int_{w'\in\Delta^{n-m}(w,r)}\left(\int_{z\in\Delta^m(o,r)}|f(z,w')|^2e^{-\varphi_1-\psi_1}\right)e^{-\varphi_2}\\
				=&\frac{C}{(\pi r^2)^{n-m}}\int_{(z,w')\in\Delta^n((o,w),r)}|f(z,w')|^2e^{-\varphi-\psi}<+\infty,
			\end{split}
		\end{flalign*}
		where $C$ is a positive constant. And
		\begin{flalign*}
			\begin{split}
				&\int_{w'\in\Delta^{n-m}(w,r)}|f(o,w')|^2e^{-\varphi_2}\\
				\leq&\frac{1}{(\pi r^2)^m}\int_{w'\in\Delta^{n-m}(w,r)}\left(\int_{z\in\Delta^m(o,r)}|f(z,w')|^2\right)e^{-\varphi_2}\\
				\leq&\frac{\sup_{\Delta^m(o,r)}e^{\varphi_1+\psi_1}}{(\pi r^2)^m}\int_{(z,w')\in\Delta^n((o,w),r)}|f(z,w')|^2e^{-\varphi-\psi}<+\infty.
			\end{split}
		\end{flalign*}
		It means that $(f(\cdot,w),o)\in\mathcal{I}(\varphi_1+\psi_1)_o$ on $X$, and $(f(o,\cdot),w)\in\mathcal{I}(\varphi_2)_w$ on $Y$.	
	\end{proof}	
	
	\begin{lemma}\label{e-varphic-psi}
		Assume that
		\begin{equation*}
			\int_M|f|^2e^{-\varphi}c(-\psi)<+\infty.
		\end{equation*}
		Then for any $w\in \Delta^{n-m}$,
		\begin{equation*}
			\int_{z\in\Delta^m}|f(z,w)|^2e^{-\varphi_1}c(-\psi_1)<+\infty.
		\end{equation*}
	\end{lemma}
	
	\begin{proof}
		According to the Fubini's Theorem, we have
		\begin{flalign*}
			\begin{split}
				&\int_{z\in\Delta^m}|f(z,w)|^2e^{-\varphi_1}c(-\psi_1)\\
				\leq&\frac{1}{(\pi r^2) ^{n-m}}\int_{w'\in\Delta^{n-m}(w,r)}\left(\int_{z\in\Delta^m}|f(z,w')|^2e^{-\varphi_1}c(-\psi_1)\right)\\
				\leq&\frac{e^T}{(\pi r^2) ^{n-m}}\int_{w'\in\Delta^{n-m}(w,r)}\left(\int_{z\in\Delta^m}|f(z,w')|^2e^{-\varphi_1}c(-\psi_1)\right)e^{-\varphi_2}\\
				\leq&\frac{e^T}{(\pi r^2)^{n-m}}\int_{\Delta^n}|f|^2e^{-\varphi}c(-\psi)<+\infty,
			\end{split}
		\end{flalign*}
		where $r>0$ such that $\Delta^{n-m}(w,r)\subset\subset\Delta^{n-m}$, and $T:=-\sup_{w'\in\Delta^{n-m}(w,r)}\varphi_2(w')$.
	\end{proof}
	
	\begin{lemma}\label{f1zf2w}
		Let $f_1(z)$ be a holomorphic function on  $X$ such that $(f_1,o)\in \mathcal{I}(\varphi_1+\psi_1)_o$, and $f_2(w)$ be a holomorphic function on $Y$ such that $(f_2,w)\in\mathcal{I}(\varphi_2)_w$ for any $w\in Y$. Let $\tilde{f}(z,w)=f_1(z)f_2(w)$ on $M$, then $(\tilde{f},(o,w))\in\mathcal{I}(\varphi+\psi)_{(o,w)}$ for any $(o,w)\in Y_o$.
	\end{lemma}
	
	\begin{proof}
		According to $(f_1,o)\in \mathcal{I}(\varphi_1+\psi_1)_o$ and $(f_2,w)\in\mathcal{I}(\varphi_2)_w$, we can find some $r>0$ such that
		\begin{equation*}
			\int_{\Delta^m(o,r)}|f_1|^2e^{-\varphi_1-\psi_1}<+\infty,
		\end{equation*}
		and
		\begin{equation*}
			\int_{\Delta^{n-m}(w,r)}|f_2|^2e^{-\varphi_2}<+\infty.
		\end{equation*}
		Then using Fubini's Theorem, we get
		\begin{equation*}
			\int_{\Delta^n((o,w),r)}|\tilde{f}|^2e^{-\varphi-\psi}=\int_{\Delta^m(o,r)}|f_1|^2e^{-\varphi_1-\psi_1}\int_{\Delta^{n-m}(w,r)}|f_2|^2e^{-\varphi_2}<+\infty,
		\end{equation*}
		which means that $(\tilde{f},(o,w))\in\mathcal{I}(\varphi+\psi)_{(o,w)}$.
	\end{proof}	
	
	Let $X=\Delta^m$ be the unit polydisc, where the coordinate is $w=(w_1,\ldots,w_m)$. Let $Y$ be an $(n-m)$ dimensional complex manifold, and let $M=X\times Y$.
	
	\begin{lemma}\label{decomp}
		For any $(n,0)$ holomorphic form $F$ on $M$, there exists a unique sequence of $(n-m,0)$ holomorphic forms $\{F_{\alpha}\}_{\alpha\in\mathbb{N}^m}$ on $Y$ such that
		\begin{equation*}
			F=\sum_{\alpha\in\mathbb{N}^m}\pi_1^*(w^{\alpha}dw)\wedge \pi_2^*(F_{\alpha}),
		\end{equation*}
		where the right side is uniformly convergent on any compact subset of $U$. Here $\pi_1$, $\pi_2$ are the natural projections from $M$ to $X$ and $Y$, $dw=dw_1\wedge\cdots\wedge dw_m$, and $w^{\alpha}=w_1^{\alpha_1}\cdots w_n^{\alpha_m}$ for any $\alpha=(\alpha_1,\cdots,\alpha_m)\in\mathbb{N}^m$.
	\end{lemma}
	
	\begin{proof}
		Firstly we consider the local case. Assume that $Y=\Delta^{n-m}$, and the coordinate is $w'=(w_1',\ldots,w_{n-m}')$. Then there exists a holomorphic function $\tilde{F}(w,w')$ on $M=\Delta^n$ such that
		\begin{equation*}
			F=\tilde{F}(w,w')dw\wedge dw'.
		\end{equation*}
		Considering the Taylor's expansion of $\tilde{F}$, we can assume that
		\begin{flalign*}
			\begin{split}
			\tilde{F}(w,w')&=\sum_{\alpha\in\mathbb{N}^m,\alpha'\in \mathbb{N}^{n-m}}d_{\alpha,\alpha'}w^\alpha{w'}^{\alpha'}\\
			&=\sum_{\alpha\in\mathbb{N}^m}w^{\alpha}\left(\sum_{\alpha'\in\mathbb{N}^{n-m}}d_{\alpha,\alpha'}{w'}^{\alpha'}\right)\\
			&=\sum_{\alpha\in\mathbb{N}^m}w^{\alpha}\tilde{F}_{\alpha}(w'),
			\end{split}
		\end{flalign*}
		where
		\begin{equation*}
			\tilde{F}_{\alpha}(w')=\frac{1}{\alpha!}\cdot\left(\frac{\partial^{\alpha}\tilde{F}(w,w')}{\partial w^{\alpha}}\bigg|_{w=0}\right)(w'),
		\end{equation*}		
		and the summations are uniformly convergent on any compact subset of $M$. Let $F_{\alpha}=\tilde{F}_{\alpha}(w')dw'$ on $Y=\Delta^{n-m}$, then we have
		\begin{equation*}
			F=\sum_{\alpha\in\mathbb{N}^m}\pi_1^*(w^{\alpha}dw)\wedge \pi_2^*(F_{\alpha}).
		\end{equation*}
		
		Secondly, we need to prove that $F_{\alpha}$ is independent of the choices of the local coordinates of $Y$. Assume that $z'=(z_1'\ldots,z_{n-m}')$ is another coordinate on $Y=\Delta^{n-m}$, and $F=\tilde{F}_0(w,z')dw\wedge dz'$. Then we have $\tilde{F}(w,w'(z'))dw'=\tilde{F}_0(w,z')dz'$, which induces that
		\begin{flalign*}
			\begin{split}
				F_{\alpha}&=\frac{1}{\alpha!}\cdot\left(\frac{\partial^{\alpha}\tilde{F}(w,w')}{\partial w^{\alpha}}\bigg|_{w=0}\right)(w')dw'\\
				&=\frac{1}{\alpha!}\cdot\left(\frac{\partial^{\alpha} \tilde{F}_0(w,z')}{\partial w^{\alpha}}\bigg|_{w=0}\right)(z')dz'.
			\end{split}
		\end{flalign*}
		It means that $F_{\alpha}$ is independent of the choices of the coordinates for any $\alpha\in\mathbb{N}^m$.
		
		For general $Y$, we can find holomorphic $(n-m,0)$ forms $F_{\alpha}$ on $Y$ such that $F=\sum_{\alpha\in\mathbb{N}^m}\pi_1^*(w^{\alpha}dw)\wedge \pi_2^*(F_{\alpha})$ with gluing.
		
		Finally, for the uniqueness, in the local case, we have
		\begin{equation*}
			\tilde{F}(w,w')dw\wedge dw'=F=\sum_{\alpha\in\mathbb{N}^m}\pi_1^*(w^{\alpha}dw)\wedge\pi_2^*(F_{\alpha}),
		\end{equation*}
		which implies that
		\begin{equation*}
			F_{\alpha}(w')=\frac{1}{\alpha!}\cdot\left(\frac{\partial^\alpha \tilde{F}(w,w')}{\partial w^{\alpha}}|_{w=0}\right)(w')dw'.
		\end{equation*}
		Then we get the uniqueness in the local case, which can also implies the uniqueness for the general $Y$.
	\end{proof}
	
	Let $\tilde{M}$ be an $n-$dimensional complex submanifold of $M=\Delta^m\times Y$ such that $\{o\}\times Y\subset \tilde{M}$, where $o$ is the origin of $\Delta^m$.
	\begin{lemma}\label{decomp-tildeM}
		For any holomorphic $(n,0)$ form $F$ on $\tilde{M}$, there exist a unique sequence of holomorphic $(n-m,0)$ forms $\{F_{\alpha}\}_{\alpha\in\mathbb{N}}$ on $Y$ and a neighborhood $U\subset \tilde{M}$ of $\{o\}\times Y$ such that
		\begin{equation*}
			 F=\sum_{\alpha\in\mathbb{N}^m}\pi_1^*(w^{\alpha}dw)\wedge \pi_2^*(F_{\alpha})
		\end{equation*}
	on $U$, where the right side is uniformly convergent on any compact subset of $U$.
	\end{lemma}
	
	\begin{proof}
	For any open subset $V$ of $Y$ with $V\subset\subset Y$, there exists $r_V\in (0,1)$ such that $\Delta_{r_V}^m\subset \tilde{M}$. It follows from Lemma \ref{decomp} that there exists a unique sequence of holomorphic $(n-m,0)$ forms $\{F_{\alpha,V}\}_{\alpha\in\mathbb{N}}$ on $V$ such that
		\begin{equation*}
	F=\sum_{\alpha\in\mathbb{N}^m}\pi_1^*(w^{\alpha}dw)\wedge \pi_2^*(F_{\alpha,V})
	\end{equation*}	
	on $\Delta_{r_V}^m\times V$, where the right side is uniformly convergent on any compact subset of $\Delta_{r_V}^m\times V$. According to the uniqueness of the decomposition in Lemma \ref{decomp}, we get that there exists a unique sequence of holomorphic $(n-m,0)$ forms $\{F_{\alpha}\}_{\alpha\in\mathbb{N}}$ on $Y$ and a neighborhood $U\subset \tilde{M}$ of $\{o\}\times Y$, such that \begin{equation*}
		F=\sum_{\alpha\in\mathbb{N}^m}\pi_1^*(w^{\alpha}dw)\wedge \pi_2^*(F_{\alpha})
	\end{equation*}
	on $U$, where the right side is uniformly convergent on any compact subset of $U$.
	\end{proof}
	
	Let $X$ be an $n_1-$dimensional complex manifold, and let $Y$ be an $n_2-$dimensional complex manifold. Let $M=X\times Y$ be an $n-$dimensional complex manifold, where $n=n_1+n_2$. Let $\pi_1$ and $\pi_2$ be the natural projections from $M$ to $X$ and $Y$ respectively.
	\begin{lemma}
		\label{decom-product}
		Let $F\not\equiv0$ be a holomorphic $(n,0)$ form on $M$. Let $f_1$ be a holomorphic $(n_1,0)$ form on an open subset $U$ of $X$, and let $f_2$ be a holomorphic $(n_2,0)$ form on an open subset $V$ of $Y$. If
		$$F=\pi_1^*(f_1)\wedge\pi_2^*(f_2)$$
		on $U\times V$, there exist a holomorphic $(n_1,0)$ form $F_1$ on $X$ and a holomorphic $(n_2,0)$ form $F_2$ on $Y$ such that $F_1=f_1$ on $U$, $F_2=f_2$ on $V$, and
		$$F=\pi_1^*(F_1)\wedge\pi_2^*(F_2)$$
		on $M$.
	\end{lemma}
	\begin{proof}
		As $F\not\equiv0$ and $F=\pi_1^*(f_1)\wedge\pi_2^*(f_2)$, we have $f_1\not\equiv0$ and $f_2\not\equiv0$. Choosing $x\in U$ such that $f_1(x)\not=0$, there exists a local coordinate $w=(w_1,\ldots,w_{n_1})$ on a neighborhood $U_1\subset U$ satisfying $w(x)=0$ and $w(U_1)=\Delta^{n_1}$. Then $f_1=\sum_{\alpha\in\mathbb{N}^{n_1}}b_{\alpha}w^{\alpha}dw_1\wedge\ldots\wedge dw_{n_1}$, where $b_{\alpha}$ is a constant for any $\alpha$ and $b_{0}\not=0$, hence
		$$F=\sum_{\alpha\in\mathbb{N}^{n_1}}b_{\alpha}\pi_1^*(w^{\alpha}dw_1\wedge\ldots\wedge dw_{n_1})\wedge\pi_2^*(f_2)$$
		on $U_1\times V$. It follows from Lemma \ref{decomp} that there exists a sequence of holomorphic $(n_2,0)$ forms $\{h_{\alpha}\}_{\alpha\in\mathbb{N}^{n_1}}$ on $Y$ such that
		$$F=\sum_{\alpha\in\mathbb{N}^{n_1}}\pi_1^*(w^{\alpha}dw_1\wedge\ldots\wedge dw_{n_1})\wedge\pi_2^*(h_{\alpha})$$
		on $U_1\times Y$. It follows from the uniqueness of the decomposition in Lemma \ref{decomp} that $b_{\alpha}f_2=h_{\alpha}$ on $V$, thus there exists a holomorphic $(n_2,0)$ form $F_2$ on $Y$ such that $F_2=f_2$ on $V$. With a similar discussion, we know that there exists a holomorphic $(n_1,0)$ form on $X$ such that $F_1=f_1$ on $X$. As $F=\pi_1^*(f_1)\wedge\pi_2^*(f_2)$ on $U\times V$, we have $F=\pi_1^*(F_1)\wedge\pi_2^*(F_2)$ on $M$.
	\end{proof}
	
	Let $\Omega=\Delta$ be the unit disk in $\mathbb{C}$, where the coordinate is $w$. Let $Y=\Delta^{n-1}$ be the unit polydisc in $\mathbb{C}^n$, where the coordinate is $w'=(w'_1,\ldots,w'_{n-1})$. Let $M=\Omega\times Y$. Let $\pi_1$, $\pi_2$ be the natural projections from $M$ to $\Omega$ and $Y$.
	
	Let $\psi_1=2p\log|w|+\psi_0$ on $\Omega$, where $p>0$ and $\psi_0$ is a negative subharmonic function on $\Omega$ with $\psi_0(0)>-\infty$. Let $\varphi_1$ be a Lebesgue measurable function on $\Omega$ such that $\varphi_1+\psi_1=2\log|g|+2\log|w|+2u$, where $g$ is a holomorphic function on $\Omega$ with $ord(g)_0=k_0\in\mathbb{N}$, $u$ is a subharmonic function on $\Omega$ such that $v(dd^cu,z)\in[0,1)$ for any $z\in \Omega$. Let $\varphi_2$ be a plurisubharmonic function on $Y$. Let $\psi:=\pi_1^*(\psi_1)$ and $\varphi:=\pi_1^*(\varphi_1)+\pi_2^*(\varphi_2)$ on $M$.
	Let $F$ be a holomorphic $(n,0)$ form on $M$, where
	\begin{equation*}
		F=\sum_{j=k}^{\infty}\pi_1^*(w^jdw)\wedge \pi_2^*(F_j)
	\end{equation*}
	according to Lemma \ref{decomp}. Here $k\in\mathbb{N}$ and $F_j$ is a holomorphic $(n-1,0)$ form on $Y$ for any $j\geq k$.
	
	\begin{lemma}\label{k>k0}
		Let $c$ be a positive measurable function on $(0,+\infty)$ such that $c(t)e^{-t}$ is decreasing on $(0,+\infty)$, $c$ is increasing near $+\infty$, and $\int_0^{+\infty}c(s)e^{-s}ds<+\infty$. Assume that $k>k_0$, and
		\[\int_M|F|^2e^{-\varphi}c(-\psi)<+\infty.\]
		Then
		\[(F,(0,y))\in(\mathcal{O}(K_M))_{(0,y)}\otimes\mathcal{I}(\varphi+\psi)_{(0,y)}\]
		for any $y\in Y$.
	\end{lemma}
	
	\begin{proof}
		It follows from $\Omega$ is a Stein manifold  and Lemma \ref{l:FN1} that there exist smooth subharmonic functions $u_l$ and $\Psi_l$ on $\Omega$ such that $u_l$ are decreasingly convergent to $u$, and $\Psi_l$ are decreasingly convergent to $\psi_0$.
		
		Then there exists some $t_1\geq 0$ such that $c$ is increasing on $(t_1,+\infty)$, and for any $t>t_1$,
		\begin{flalign}\label{k>k0-1}
			\begin{split}
				&\int_{\{\psi<-t\}}|F|e^{-\varphi}c(-\psi)\\
				=&\int_{\{\pi_1^*(2p\log|w|+\psi_0)<-t\}}|F|^2e^{\pi_1^*(-2\log|g|-2u+(2p-2)\log|w|+\psi_0)-\pi_2^*(\varphi_2)}c(\pi_1^*(-2p\log|w|-\psi_0))\\
				\geq&\int_{\{\pi_1^*(2p\log|w|+\Psi_l)<-t\}}|F|^2e^{\pi_1^*(-2\log|g|-2u_l+(2p-2)\log|w|+\psi_0)-\pi_2^*(\varphi_2)}c(\pi_1^*(-2p\log|w|-\Psi_l)).
			\end{split}
		\end{flalign}
		
		For any $\varepsilon>0$, there exists $\rho>0$ such that
		
		(1).
		\begin{equation*}
			\sup_{|z|<\rho} 2|u_l(z)-u_l(0)|<\varepsilon;
		\end{equation*}
		
		(2).
		\begin{equation*}
			\sup_{|z|<\rho} |\Psi_l(z)-\Psi_l(0)|<\varepsilon.
		\end{equation*}
		
		Then there exists $t'>0$ such that for any $t>t'$, we have
		\begin{flalign}\label{k>k0-2}
			\begin{split}
				&\int_{\{\psi<-t\}}|F|e^{-\varphi}c(-\psi)\\
				\geq&\int_{\{\pi_1^*(2p\log|w|+\Psi_l)<-t\}}|F|^2e^{\pi_1^*(-2\log|g|-2u_l+(2p-2)\log|w|+\psi_0)-\pi_2^*(\varphi_2)}c(\pi_1^*(-2p\log|w|-\Psi_l))\\
				\geq&\int_{\{2p\log|w|+\Psi_l(0)<-t-\varepsilon,w'\in Y\}}|w^k\tilde{h}(w,w')|^2|w|^{2p-2}e^{-2\log|g(w)|-2u_l(0)-\varepsilon+\psi_0(w)}e^{-\varphi_2(w')}\\
				&\cdot c(-2p\log|w|-\Psi_l(0)-\varepsilon)|dw\wedge dw'|^2\\
				\geq&\int_{w'\in Y}\left(\int_0^{e^{-\frac{t+\Psi_l(0)+\varepsilon}{2p}}}\int_0^{2\pi}2r^{2k+1}\left|\frac{\tilde{h}(re^{i\theta},w')}{g(re^{i\theta})}\right|^2r^{2p-2}c(-2p\log r-\Psi_l(0)-\varepsilon)e^{\psi_0(re^{i\theta})}drd\theta\right)\\
				&\cdot e^{-2u_l(0)-\varepsilon}e^{-\varphi_2(w')}|dw'|^2\\
				\geq&\frac{2\pi}{p|d|^2}e^{-2u_l(0)-\varepsilon+\psi_0(0)-\left(\frac{k-k_0}{p}+1\right)(\Psi_l(0)+\varepsilon)}\int_t^{+\infty}c(s)e^{-\left(\frac{k-k_0}{p}+1\right)s}ds\\
				&\cdot\int_{w'\in Y}|\tilde{h}(0,w')|^2e^{-\varphi_2(w')}|dw'|^2,
			\end{split}
		\end{flalign}
		where $\tilde{h}(w,w')$ is a holomorphic function on $M$ such that $F=w^k\tilde{h}(w,w')dw\wedge dw'$ on $M$, and $d:=\lim\limits_{w\rightarrow 0}\frac{g(w)}{w^{k_0}}$. Since $k>k_0$ and $\int_0^{+\infty}c(s)e^{-s}ds<+\infty$, we have
		\begin{equation*}
			\int_t^{+\infty}c(s)e^{-\left(\frac{k-k_0}{p}+1\right)s}ds<+\infty.
		\end{equation*}
		Then it follows from $\int_{\{\psi<-t\}}|F|e^{-\varphi}c(-\psi)<+\infty$ and inequality (\ref{k>k0-2}) that
		\begin{equation}\label{k>k0-3}
			\int_Y|F_k|^2e^{-\varphi_2}=\int_{w'\in Y}|\tilde{h}(0,w')|^2e^{-\varphi_2(w')}|dw'|^2<+\infty.
		\end{equation}
		In addition, we get that $w^k\in\mathcal{I}(\varphi_1+\psi_1)_0$ according to $k>k_0$. Then it follows from Lemma \ref{local-germ} that
		\begin{equation}
			(\pi_1^*(w^kdw)\wedge\pi_2^*(F_k),(0,y))\in(\mathcal{O}(K_M))_{(0,y)}\otimes\mathcal{I}(\varphi+\psi)_{(0,y)}.
		\end{equation}
		
		According to the Fubini's Theorem and inequality (\ref{k>k0-3}), there exists $r>0$ such that
		\begin{flalign}
			\begin{split}
				&\int_{\Delta(0,r)\times Y}|\pi_1^*(w^kdw)\wedge \pi_2^*(F_k)|^2e^{-\varphi}c(-\psi)\\
				=&\int_{\Delta(0,r)}|w^k|^2e^{-\varphi_1}c(-\psi_1)|dw|^2\cdot\int_Y|F_k|^2e^{-\varphi_2}\\
				\leq&C\int_{\Delta(0,r)}|w^k|^2e^{-\varphi_1-\psi_1}|dw|^2\cdot\int_Y|F_k|^2e^{-\varphi_2}\\
				<&+\infty,
			\end{split}
		\end{flalign}
		where $C$ is a positive constant independent of $F$. Then we have
		\begin{equation*}
			\int_{\Delta(0,r)\times Y}|F-\pi_1^*(w^kdw)\wedge\pi_2^*(F_k)|e^{-\varphi}c(-\psi)<+\infty,
		\end{equation*}
		since $\int_{\Delta(0,r)\times Y}|F|e^{-\varphi}c(-\psi)<+\infty$. Note that
		\begin{equation*}
			F-\pi_1^*(w^kdw)\wedge\pi_2^*(F_k)=\sum_{j=k+1}^{\infty}\pi_1^*(w^jdw)\wedge\pi_2^*(F_j)
		\end{equation*}
		and $k+1>k_0$. Using the same methods we can get that
		\begin{equation*}
			(\pi_1^*(w^{k+1}dw)\wedge\pi_2^*(F_{k+1}),(0,y))\in(\mathcal{O}(K_M))_{(0,y)}\otimes\mathcal{I}(\varphi+\psi)_{(0,y)}
		\end{equation*}
		for any $y\in Y$ and
		\begin{equation*}
			\int_{\Delta(0,r')\times Y}\left|F-\pi_1^*(w^kdw)\wedge\pi_2^*(F_k)-\pi_1^*(w^{k+1}dw)\wedge\pi_2^*(F_{k+1})\right|^2e^{-\varphi}c(-\psi)<+\infty
		\end{equation*}
		for some $r'\in (0,r)$. Now with inductions, we can get that
		\begin{equation*}
			(\pi_1^*(w^jdw)\wedge\pi_2^*(F_j),(0,y))\in(\mathcal{O}(K_M))_{(0,y)}\otimes\mathcal{I}(\varphi+\psi)_{(0,y)}
		\end{equation*}
		for any $j\geq k$, $y\in Y$. Then it follows from Lemma \ref{module} that
		\begin{equation*}
			(F,(0,y))=(\sum_{j=k}^{\infty}\pi_1^*(w^jdw)\wedge\pi_2^*(F_j),(0,y))\in(\mathcal{O}(K_M))_{(0,y)}\otimes\mathcal{I}(\varphi+\psi)_{(0,y)}
		\end{equation*}
		for any $y\in Y$.
	\end{proof}

	\subsection{Some other lemmas}
	\
	The following lemmas will be used in the proofs of the main results and their applications.
	
	Let $Z_0':=\{z_j:j\in \mathbb{N}_+ \ \& \ j<\gamma\}$ be a discrete subset of the open Riemann surface $\Omega$, where $\gamma\in\mathbb{N}_+\cup\{+\infty\}$. Let $Y$ be an $n-1$ dimensional weakly pseudoconvex K\"{a}hler manifold. Let $M=\Omega\times Y$ be a complex manifold, and $K_M$ be the canonical line bundle on $M$. Let $\pi_1$, $\pi_2$ be the natural projections from $M$ to $\Omega$ and $Y$. Let $\psi_1$ be a subharmonic function on $\Omega$ such that $p_j=\frac{1}{2}v(dd^c\psi_1,z_j)>0$, and let $\varphi_1$ be a Lebesgue measurable function on $\Omega$ such that $\varphi_1+\psi_1$ is subharmonic on $\Omega$. Let $\varphi_2$ be a plurisubharmonic function on $Y$. Denote that $\psi:=\pi_1^*(\psi_1)$, $\varphi:=\pi_1^*(\varphi_1)+\pi_2^*(\varphi_2)$. Using the Weierstrass Theorem on open Riemann surfaces (see \cite{OF81}) and
	the Siu’s Decomposition Theorem, we have
	\[\varphi_1+\psi_1=2\log |g_0|+2u_0,\]
	where $g_0$ is a holomorphic function on $\Omega$ and $u_0$ is a subharmonic function on $\Omega$ such that $v(dd^cu_0,z)\in [0,1)$ for any $z\in\Omega$.
	
	Let $w_j$ be a local coordinate on a neighborhood $V_{z_j}\subset\subset\Omega$ of $z_j$ satisfying $w_j(z_j)=0$ for $z_j\in Z'_0$, where $V_{z_j}\cap V_{z_k}=\emptyset$ for any $j,k$, $j\neq k$. Denote that $V_0:=\bigcup_{1\leq j<\gamma}V_{z_j}$. Assume that $g_0=d_jw_j^{k_j}h_j$ on $V_{z_j}$, where $d_j$ is a constant, $k_j$ is a nonnegative integer, and $h_j$ is a holomorphic function on $V_{z_j}$ such that $h_j(z_j)=1$ for any $j$ ($1\leq j<\gamma$).
	
	Let $c(t)$ be a positive measurable function on $(0,+\infty)$ satisfying that $c(t)e^{-t}$ is decreasing and $\int_0^{+\infty}c(t)e^{-t}dt<+\infty$. We can get the following lemmas.
	
	\begin{lemma}\label{L2ext-finite-f}
		Let $F$ be a holomorphic $(n,0)$ form on $V_0\times Y$ such that $F=\pi_1^*(w_j^{\tilde{k}_j}f_jdw_j)\wedge \pi_2^*(\tilde{F}_j)$ on $V_{z_j}\times Y$, where $\tilde{k_j}$ is a nonnegative integer, $f_j$ is a holomorphic function on $V_{z_j}$ such that $f_j(z_j)=a_j\in\mathbb{C}\setminus\{0\}$, and $\tilde{F}_j$ is a holomorphic $(n-1,0)$ form on $Y$ for any $j$ ($1\leq j<\gamma$).
		 
		Denote that $I_F:=\{j : 1\leq j<\gamma\& \tilde{k}_j+1-k_j\leq 0\}$. Assume that $\tilde{k}_j+1-k_j=0$ and $u_0(z_j)>-\infty$ for $j\in I_F$. If
		\begin{equation}
			\int_Y|\tilde{F}_j|^2e^{-\varphi_2}<+\infty
		\end{equation}
		for any $j$ with $1\leq j<\gamma$, and
		\begin{equation}
			\sum_{j\in I_F}\frac{2\pi|a_j|^2 e^{-2u_0(z_j)}}{p_j|d_j|^2}\int_Y|\tilde{F}_j|^2e^{-\varphi_2}<+\infty,
		\end{equation}
		then there exists a holomorphic $(n,0)$ form $\tilde{F}$ on $M$, such that $(\tilde{F}-F,(z_j,y))\in(\mathcal{O}(K_M)\otimes\mathcal{I}(\varphi+\psi))_{(z_j,y)}$ for any $j$ ($1\leq j<\gamma$) and $y\in Y$, and
		\begin{equation}
			\int_M|\tilde{F}|^2e^{-\varphi}c(-\psi)\leq \left(\int_0^{+\infty}c(s)e^{-s}ds\right)\sum_{j\in I_F}\frac{2\pi|a_j|^2 e^{-2u_0(z_j)}}{p_j|d_j|^2}\int_Y|\tilde{F}_j|^2e^{-\varphi_2}.	
		\end{equation}
	\end{lemma}

	\begin{proof}
		Lemma \ref{l:green-sup2} tells us that $\psi_1\leq 2\sum_{1\leq j<\gamma}p_jG_{\Omega}(\cdot,z_j)$. Since $c(t)e^{-t}$ is decreasing on $(0,+\infty)$, we have
		\begin{equation*}
			e^{-\varphi}c(-\psi)\leq e^{-\pi_1^*\left(\varphi_1+\psi_1-2\sum_{1\leq j<\gamma}p_jG_{\Omega}(\cdot,z_j)\right)-\pi_2^*(\varphi_2)}c\left(-\pi_1^*(2\sum_{1\leq j<\gamma}p_jG_{\Omega}(\cdot,z_j))\right).
		\end{equation*}
		Thus we can assume that $\psi_1=2\sum_{1\leq j<\gamma}p_jG_{\Omega}(\cdot,z_j)$. 
		
		The following remark shows that we only need to prove Lemma \ref{L2ext-finite-f} when $Z_0'$ is a finite set.
		
		\begin{remark}\label{rem-L2ext}
		It follows from Lemma \ref{green-approx} that there exists a sequence of Riemann surfaces $\{\Omega_l\}_{l\in\mathbb{N}_+}$ such that $\Omega_l\subset\subset\Omega_{l+1}\subset\subset\Omega$ for any $l$, $\bigcup_{l\in\mathbb{N}_+}\Omega_l=\Omega$, and $\{G_{\Omega_l}(\cdot,z_0)-G_{\Omega}(\cdot,z_0)\}$ is decreasingly convergent to $0$ for any $z_0\in\Omega$. As $Z'_0$ is a discrete subset of $\Omega$, $Z_l:=\Omega_l\cap Z'_0$ is a set of finite points. Denote that
		\begin{equation*}
			\psi_{1,l}:=2\sum_{z_j\in Z_l}p_jG_{\Omega_l}(\cdot,z_j),
		\end{equation*}
		and
		\begin{equation*}
			\varphi_{1,l}:=\varphi_1+\psi_1-\psi_{1,l}.
		\end{equation*}
		Then we have $\varphi_{1,l}+\psi_{1,l}=\varphi_1+\psi_1$ on $\Omega_l$. Denote that $\varphi'_l:=\pi_1^*(\varphi_{1,l})$, and $\psi'_l:=\pi_1^*(\psi_{1,l})$. Denote that $I_l:=I_F\cap \{j : z_j\in Z_l\}$. After this remark we will prove that there exists a holomorphic $(n,0)$ form $F_l$ on $M_l$ such that $(F_l-F,(z_j,y))\in (\mathcal{O}(K_M)\otimes\mathcal{I}(\varphi'_l+\psi'_l))_{(z_j,y)}= (\mathcal{O}(K_M)\otimes\mathcal{I}(\varphi+\psi))_{(z_j,y)}$ for any $z_j\in Z_l$, $y\in Y$ and
		\begin{equation*}
			\int_{M_l}|F_l|^2e^{-\varphi'_l}c(-\psi'_l)\leq \left(\int_0^{+\infty}c(t)e^{-t}dt\right)\sum_{j\in I_l}\frac{2\pi|a_j|^2 e^{-2u_0(z_j)}}{p_j|d_j|^2}\int_Y|\tilde{F}_j|^2e^{-\varphi_2}
		\end{equation*}
		where $M_l:=\Omega_l\times Y$.
		
		Since $\psi\leq\psi'_l$ and $c(t)e^{-t}$ is decreasing on $(0,+\infty)$, we have
		\begin{flalign}
			\begin{split}
				&\int_{M_l}|F_l|^2e^{-\varphi}c(-\psi)\\
				\leq&\int_{M_l}|F_l|^2e^{-\varphi'_l}c(-\psi'_l)\\
				\leq&\left(\int_0^{+\infty}c(s)e^{-s}ds\right)\sum_{j\in I_l}2\pi\frac{|a_j|^2e^{-2u_0(z_j)}}{p_j|d_j|^2}\int_Y|\tilde{F}_j|^2e^{-\varphi_2}.
			\end{split}
		\end{flalign}
		Note that $\psi$ is smooth on $M\setminus (Z'_0\times Y)$. For any compact subset $K$ of $M\setminus (Z'_0\times Y)$, there exsits $s_K>0$ such that $\int_Ke^{-s_K\psi}dV_M<+\infty$, where $dV_M$ is a continuous volume form on $M$. Then we have
		\begin{equation*}
			\int_K\left(\frac{e^{\varphi}}{c(-\psi)}\right)^{s_K}dV_M=\int_K\left(\frac{e^{\varphi+\psi}}{c(-\psi)}\right)^{s_K}e^{-s_K\psi}dV_M\leq C\int_Ke^{-s_K\psi}dV_M <+\infty,
		\end{equation*}
		where $C$ is a constant. It follows from Lemma \ref{s_K} and the diagonal method that there exists a subsequence of $\{F_l\}$, denoted also by $\{F_l\}$, which is uniformly convergent to a holomorphic $(n,0)$ form $\tilde{F}$ on $M$ on any compact subset of $M$ and
		\begin{flalign}
			\begin{split}
				&\int_M|\tilde{F}|^2e^{-\varphi}c(-\psi)\\
				\leq&\lim_{l\rightarrow+\infty}\int_{M_l}|F_l|^2e^{-\varphi}c(-\psi)\\
				\leq&\lim_{l\rightarrow+\infty}\left(\int_0^{+\infty}c(s)e^{-s}ds\right)\sum_{j\in I_l}2\pi\frac{|a_j|^2e^{-2u_0(z_j)}}{p_j|d_j|^2}\int_Y|\tilde{F}_j|^2e^{-\varphi_2}.\\
				=&\left(\int_0^{+\infty}c(s)e^{-s}ds\right)\sum_{j\in I_F}2\pi\frac{|a_j|^2e^{-2u_0(z_j)}}{p_j|d_j|^2}\int_Y|\tilde{F}_j|^2e^{-\varphi_2}.
			\end{split}
		\end{flalign}
		Since $\{F_l\}$ is uniformly convergent to $\tilde{F}$ on any compact subset of $M$ and $(F_l-F,(z_j,y))\in(\mathcal{O}(K_M)\otimes\mathcal{I}(\varphi+\psi))_{(z_j,y)}$ for any $l$, $z_j\in Z'_l$, and $y\in Y$. Then it follows from Lemma \ref{module} that $(\tilde{F}-F,(z_j,y))\in(\mathcal{O}(K_M)\otimes\mathcal{I}(\varphi+\psi))_{(z_j,y)}$ for any $z_j\in Z'_0$, and $y\in Y$.	
		\end{remark}

		According to the above discussions, we assume that $\gamma=m+1$ and $I_F=\{1,2,\ldots,m_1\}$, where $m$ is a positive integer and $m_1<m$ is a nonnegative integer ($I_F=\emptyset$ if and only is $m_1=0$). And we can replace $\Omega_l$, $\psi'_l$, $\varphi'_l$ by $\Omega$, $\psi$, $\varphi_1$.
		
		Since $\Omega$ is a Stein manifold, there exist smooth subharmonic functions $u_l$ on $\Omega$, which are decreasingly convergent to $u_0$ with respect to $l$.
		
		Using Lemma \ref{l:G-compact}, we know that there exists $t_0>0$ such that $\{\psi_1<-t\}\subset\subset V_0$ for any $t\geq t_0$, which implies that $\int_{\{\psi<-t\}}|F|^2<+\infty$ according to the Fubini's Theorem and $\int_Y|\tilde{F}_j|^2e^{-\varphi_2}<+\infty$ for any $j$, $1\leq j<\gamma$. Using Lemma \ref{dbarequa}, there exists a holomorphic $(n,0)$ form $F_{l,t}$ on $M$, such that
		\begin{flalign}\label{L2ext-1}
			\begin{split}
				&\int_M|F_{l,t}-(1-b_{t,1}(\psi))F|^2e^{\pi_1^*(-2\log |g_0|-2u_l+v_{t,1}(\psi_1))-\pi_2^*(\varphi_2)}c(-v_{t,1}(\psi))\\
				\leq&\left(\int_0^{t+1}c(s)e^{-s}ds\right)\int_M\mathbb{I}_{\{-t-1<\psi<-t\}}|F|^2e^{\pi_1^*(-2\log|g_0|-2u_l)-\pi_2^*(\varphi_2)}
			\end{split}
		\end{flalign}
		for any $t\geq t_0$, where according to the Fubini's Theorem we can know that
		\begin{equation*}
			\int_M\mathbb{I}_{\{-t-1<\psi<-t\}}|F|^2e^{\pi_1^*(-2\log|g_0|-2u_l)-\pi_2^*(\varphi_2)}<+\infty.
		\end{equation*} 
		Note that $b_{t,1}(s)=0$ for large enough $s$, then $(F_{l,t}-F,(z_j,y))\in(\mathcal{O}(K_M)\otimes \mathcal{I}(\pi_1^*(2\log|g_0|)+\pi_2^*(\varphi_2)))_{(z_j,y)}=\mathcal{ O}(K_M)\otimes \mathcal{I}(\varphi+\psi)_{(z_j,y)}$ for any $j\in \{1,2,\ldots,m\}$.
		
		For any $y\in Y$, we can choose some neighborhood $U_y\subset\subset Y$ such that $\tilde{F}_j=\tilde{h}_j(w')dw'$ on $U_y$, where $w'=(w'_1,\ldots,w'_{n-1})$ is a local coordinate on $U_y$ such that $w'_1(y)=\cdots=w'_{n-1}(y)=0$, and $\tilde{h}_j$ is a holomorphic function on $U_y$.
		
		For any $\varepsilon>0$, there exists $t_1>t_0$ such that
		
		(1). for any $j\in\{1,2,\ldots,m\}$,
		\begin{equation*}
			\sup_{z\in\{\psi_1<-t_1\}\cap V_{z_j}}|g_1(z)-g_1(z_j)|<\varepsilon,
		\end{equation*}
		where $g_1$ is a smooth function on $V_0$ satisfying that
		\begin{equation*}
			g_1|_{V_{z_j}}=\psi_1-2p_j\log |w_j|;
		\end{equation*}
		
		(2). for any $j\in\{1,2,\ldots,m\}$,
		\begin{equation*}
			\sup_{z\in\{\psi_1<-t_1\}\cap V_{z_j}} 2|u_l(z)-u_l(z_j)|<\varepsilon;
		\end{equation*}
		
		(3). for any $j\in\{1,2,\ldots,m\}$,
		\begin{equation*}
			\sup_{z\in\{\psi_1<-t_1\}\cap V_{z_j}} |f_j(z)-a_j|<\varepsilon.
		\end{equation*}
		
		Note that $\tilde{k}_j+1-k_j=0$ for $1\leq j\leq m_1$ and $\tilde{k}_j+1-k_j>0$ for $m_1<j\leq m$. With direct calculating, we have
		\begin{flalign*}
			\begin{split}
				&\limsup_{t\rightarrow+\infty}\int_{\Omega\times U_y}\mathbb{I}_{\{-t-1\leq\psi_1<-t\}}|F|^2e^{\pi_1^*(-2\log|g_0|-2u_l)-\pi_2^*(\varphi_2)}\\
				\leq&\limsup_{t\rightarrow+\infty}\sum_{1\leq j\leq m}\int_{\{-t-1-\varepsilon<2p_j\log |w_j|+g_1(z_j)<-t+\varepsilon,w'\in U_y\}}\left|\frac{\tilde{h}_j(w')}{d_j}\right|^2(|a_j|+\varepsilon)^2\\
				&\cdot|w_j|^{2(\tilde{k}_j-k_j)}e^{-2u_l(z_j)+\varepsilon}e^{-\varphi_2(w')}|dw'|^2\\
				=&\limsup_{t\rightarrow+\infty}\sum_{1\leq j\leq m}\int_{w'\in U_y}\int_{\{-t-1-\varepsilon<2p_j\log |w_j|+g_1(z_j)<-t+\varepsilon\}}\left|\frac{\tilde{h}_j(w')}{d_j}\right|^2(|a_j|+\varepsilon)^2\\
				&\cdot|w_j|^{2(\tilde{k}_j-k_j)}e^{-2u_l(z_j)+\varepsilon}e^{-\varphi_2(w')}|dw'|^2\\
				\leq&\limsup_{t\rightarrow+\infty}\sum_{1\leq j\leq m}\int_{w'\in U_y}\left|\frac{\tilde{h}_j(w')}{d_j}\right|^2(|a_j|+\varepsilon)^2e^{-2u_l(z_j)+\varepsilon}e^{-\varphi_2(w')}|dw'|^2\\
				&\cdot 4\pi\int_{e^{-\frac{t+1+\varepsilon+g_1(z_j)}{2p_j}}}^{e^{-\frac{t-\varepsilon+g_1(z_j)}{2p_j}}}r^{2(\tilde{k}_j-k_j)+1}dr\\
				=&\sum_{1\leq j\leq m}\int_{w'\in U_y}\left|\frac{\tilde{h}_j(w')}{d_j}\right|^2(|a_j|+\varepsilon)^2e^{-2u_l(z_j)+\varepsilon}e^{-\varphi_2(w')}|dw'|^2\\
				&\cdot 4\pi\limsup_{t\rightarrow+\infty}\int_{e^{-\frac{t+1+\varepsilon+g_1(z_j)}{2p_j}}}^{e^{-\frac{t-\varepsilon+g_1(z_j)}{2p_j}}}r^{2(\tilde{k}_j-k_j)+1}dr\\
				=&\sum_{1\leq j\leq m_1}2\pi\int_{w'\in U_y}\left|\frac{\tilde{h}_j(w')}{d_j}\right|^2(|a_j|+\varepsilon)^2e^{-2u_l(z_j)+\varepsilon}e^{-\varphi_2(w')}|dw'|^2\frac{1+2\varepsilon}{p_j}.
			\end{split}
		\end{flalign*}
		Letting $\varepsilon\rightarrow 0+$, we get that
		\begin{flalign}
			\begin{split}
				&\limsup_{t\rightarrow+\infty}\int_{\Omega\times U_y}\mathbb{I}_{\{-t-1\leq\psi_1<-t\}}|F|^2e^{\pi_1^*(-2\log|g_0|-2u_l)-\pi_2^*(\varphi_2)}\\
				\leq&\sum_{1\leq j\leq m_1}2\pi\frac{|a_j|^2e^{-2u_l(z_j)}}{p_j|d_j|^2}\int_{w'\in U_y}|\tilde{h}_j(w')|^2e^{-\varphi_2(w')}|dw'|^2.
			\end{split}
		\end{flalign}
		According to the arbitrariness of $y$ and $U_y$, it follows from the above calculations in local case that
		\begin{flalign}\label{L2ext-2}
			\begin{split}
				&\limsup_{t\rightarrow+\infty}\int_M\mathbb{I}_{\{-t-1\leq\psi_1<-t\}}|F|^2e^{\pi_1^*(-2\log|g_0|-2u_l)-\pi_2^*(\varphi_2)}\\
				\leq&\sum_{1\leq j\leq m_1}2\pi\frac{|a_j|^2e^{-2u_l(z_j)}}{p_j|d_j|^2}\int_Y|\tilde{F}_j|^2e^{-\varphi_2}<+\infty.
			\end{split}
		\end{flalign}
		Since $v_{t,1}(\psi)\geq\psi$ and $c(t)e^{-t}$ is decreasing, combining inequality (\ref{L2ext-1}) with inequality (\ref{L2ext-2}), we have
		\begin{flalign}\label{L2ext-3}
			\begin{split}
				&\limsup_{t\rightarrow+\infty}\int_M|F_{l,t}-(1-b_{t,1}(\psi))F|^2e^{\pi_1^*(-2\log |g_0|-2u_l+\psi_1)-\pi_2^*(\varphi_2)}c(-\psi)\\
				\leq&\limsup_{t\rightarrow+\infty}\int_M|F_{l,t}-(1-b_{t,1}(\psi))F|^2e^{\pi_1^*(-2\log |g_0|-2u_l+v_{t,1}(\psi_1))-\pi_2^*(\varphi_2)}c(-v_{t,1}(\psi))\\
				\leq&\limsup_{t\rightarrow+\infty}\left(\int_0^{t+1}c(s)e^{-s}ds\right)\int_M\mathbb{I}_{\{-t-1\leq\psi_1<-t\}}|F|^2e^{-\pi_1^*(2\log|g_0|-2u_l)-\pi_2^*(\varphi_2)}\\
				\leq&\left(\int_0^{+\infty}c(s)e^{-s}ds\right)\sum_{1\leq j\leq m_1}2\pi\frac{|a_j|^2e^{-2u_l(z_j)}}{p_j|d_j|^2}\int_Y|\tilde{F}_j|^2e^{-\varphi_2}<+\infty.
			\end{split}
		\end{flalign}
		As $\tilde{k}_j-k_j=-1$ for $1<j\leq m_1$, $\tilde{k}_j-k_j\geq 0$ for $m_1<j\leq m$, $\{\psi_1<-t_0\}\subset\subset V_0$ and $\psi_1=2\sum_{1\leq j<m}p_jG_{\Omega}(\cdot,z_j)$, we have
		\begin{flalign}\label{|F|^2}
			\begin{split}
			&\int_{\{\psi<-t'\}}|F|^2e^{\pi_1^*(-2\log |g_0|+\psi_1)-\pi_2^*(\varphi_2)}c(-\psi)\\
			\leq&\sum_{j=1}^m\int_{\{\psi_1<-t'\}\cap V_{z_j}}|w_j^{\tilde{k}_j}f_jdw_j|^2e^{-2\log |g_0|+\psi_1}c(-\psi_1)\cdot\int_Y|F_j|^2e^{-\varphi_2}\\
			\leq&C\sum_{j=1}^m\int_{\{2p_j\log|w_j|<-C'\}}|w_j|^{-2}e^{2p_j\log|w_j|}c(-2p_j\log|w_j|)\cdot\int_Y|F_j|^2e^{-\varphi_2}<+\infty,
			\end{split}
		\end{flalign}
		for some $t'>0$, where $C,C'$ are positive constants. Inequality (\ref{|F|^2}) implies that
		\begin{equation*}
			\limsup_{t\rightarrow+\infty}\int_M|(1-b_{t,1}(\psi))F|^2e^{\pi_1^*(-2\log |g_0|-2u_l+\psi_1)-\pi^*_2(\varphi_2)}c(-\psi)<+\infty.
		\end{equation*}
		Thus we have
		\begin{equation*}
			\limsup_{t\rightarrow+\infty}\int_M|F_{l,t}|^2e^{\pi_1^*(-2\log |g_0|-2u_l+\psi_1)-\pi^*_2(\varphi_2)}c(-\psi)<+\infty.
		\end{equation*}
		According to Lemma \ref{s_K}, we obtain that there exists a subsequence of $\{F_{l,t}\}_{t\rightarrow+\infty}$ (also denoted by $\{F_{l,t}\}_{t\rightarrow+\infty}$) compactly convergent to a holomorphic $(n,0)$ form on $M$ denoted by $F_l$. Then it follows from inequality (\ref{L2ext-3}) and Fatou's Lemma that
		\begin{flalign}
			\begin{split}
				&\int_M|F_l|^2e^{\pi_1^*(-2\log |g_0|-2u_l+\psi_1)-\pi^*_2(\varphi_2)}c(-\psi)\\
				=&\int_M\liminf_{t\rightarrow+\infty}|F_l-(1-b_{t,1}(\psi))F|^2e^{\pi_1^*(-2\log |g_0|-2u_l+\psi_1)-\pi^*_2(\varphi_2)}c(-\psi)\\
				\leq&\liminf_{t\rightarrow+\infty}\int_M|F_l-(1-b_{t,1}(\psi))F|^2e^{\pi_1^*(-2\log |g_0|-2u_l+\psi_1)-\pi^*_2(\varphi_2)}c(-\psi)\\
				\leq&\left(\int_0^{+\infty}c(s)e^{-s}ds\right)\sum_{1\leq j\leq m_1}2\pi\frac{|a_j|^2e^{-2u_l(z_j)}}{p_j|d_j|^2}\int_Y|\tilde{F}_j|^2e^{-\varphi_2}<+\infty.
			\end{split}
		\end{flalign}
		Since $u_l$ decreasingly converges to $u_0$, which implies that $\lim\limits_{l\rightarrow+\infty}u_l(z_j)=u_0(z_j)>-\infty$ for $1\leq j\leq m_1$, then we have
		\begin{flalign}
			\begin{split}
				&\limsup_{l\rightarrow+\infty}\int_M|F_l|^2e^{\pi_1^*(-2\log |g_0|-2u_l+\psi_1)-\pi^*_2(\varphi_2)}c(-\psi)\\
				\leq&\left(\int_0^{+\infty}c(s)e^{-s}ds\right)\sum_{1\leq j\leq m_1}2\pi\frac{|a_j|^2e^{-2u_0(z_j)}}{p_j|d_j|^2}\int_Y|\tilde{F}_j|^2e^{-\varphi_2}<+\infty.
			\end{split}
		\end{flalign}
		According to Lemma \ref{s_K}, we obtain that there exists a subsequence of $\{F_l\}$ (also denoted by $\{F_l\}$) compactly convergent to a holomorphic $(n,0)$ form on $M$ denoted by $\tilde{F}$, which satisfies that $(\tilde{F}-F,(z_j,y))\in(\mathcal{O}(K_M)\otimes\mathcal{I}(\varphi+\psi))_{(z_j,y)}$ for any $j\in\{1,2,\ldots,m\}$, $y\in Y$, and
		\begin{flalign}
			\begin{split}
				&\int_M|\tilde{F}|^2e^{-\varphi}c(-\psi)\\
				=&\int_M|\tilde{F}|^2e^{\pi_1^*(-2\log |g_0|-2u_0+\psi_1)-\pi^*_2(\varphi_2)}c(-\psi)\\
				\leq&\left(\int_0^{+\infty}c(s)e^{-s}ds\right)\sum_{1\leq j\leq m_1}2\pi\frac{|a_j|^2e^{-2u_0(z_j)}}{p_j|d_j|^2}\int_Y|\tilde{F}_j|^2e^{-\varphi_2}.
			\end{split}
		\end{flalign}
		
	Finally combining with Remark \ref{rem-L2ext}, the proof of Lemma \ref{L2ext-finite-f} is completed. 	
	\end{proof}
	
	\begin{lemma}\label{fanxiangineq}
		Assume that $(\psi_1-2p_jG_{\Omega}(\cdot,z_j))(z_j)>-\infty$ for any $j$ ($1\leq j<\gamma$) and $c$ is increasing near $+\infty$. Let $F$ be a holomorphic $(n,0)$ form on $M$ such that $F=\sum_{l=\tilde{k}_j}\pi_1^*(w_j^ldw_j)\wedge F_{j,l}$ on $V_{z_j}\times Y$ for any $j$ ($1\leq j<\gamma$) according to Lemma \ref{decomp}, where $\tilde{k}_j$ is a nonnegative integer, $F_{j,l}$ is a holomorphic $(n-1,0)$ form on $V_{z_j}\times Y$ for any $j,l$, and $\tilde{F}_j:=F_{j,\tilde{k}_j}\not\equiv 0$ on $Y$. Denote that
		\begin{equation*}
			I_F:=\{j:1\leq j<\gamma \ \& \ \tilde{k}_j+1-k_j\leq 0\}.
		\end{equation*}
		Assume that
		\begin{equation*}
			\liminf_{t\rightarrow+\infty}\frac{\int_{\{\psi<-t\}}|F|^2e^{-\varphi}c(-\psi)}{\int_t^{+\infty}c(s)e^{-s}ds}<+\infty.
		\end{equation*}
		 Then $\tilde{k}_j+1-k_j=0$ for any $j\in I_F$, and
		\begin{equation}
			\liminf_{t\rightarrow+\infty}\frac{\int_{\{\psi<-t\}}|F|^2e^{-\varphi}c(-\psi)}{\int_t^{+\infty}c(s)e^{-s}ds}\geq\sum_{j\in I_F}\frac{2\pi e^{-2u_0(z_j)}}{p_j|d_j|^2}\int_Y|\tilde{F}_j|^2e^{-\varphi_2}.
		\end{equation}
	\end{lemma}
	
	\begin{proof}
		According to the Siu's Decomposition Theorem and Lemma \ref{l:green-sup}, we can assume that
		\begin{equation}
			\psi_1=2\sum_{1\leq j<\gamma}p_jG_{\Omega}(\cdot,z_j)+\psi_0,
		\end{equation}
		where $\psi_0$ is a negative subharmonic function on $\Omega$ such that $\psi_0(z_j)>-\infty$ for any $j$, $1\leq j<\gamma$.
		
		It follows from $\Omega$ is a Stein manifold and Lemma \ref{l:FN1} that there exist smooth subharmonic functions $u_l$ and $\Psi_l$ on $\Omega$ such that $u_l$ are decreasingly convergent to $u$, and $\Psi_l$ are decreasingly convergent to $\psi_0$.
		
		Denote that
		\begin{equation*}
			G:=2\sum_{1\leq j<\gamma}p_jG_{\Omega}(\cdot,z_j).
		\end{equation*}
		Then there exists some $t_1\geq 0$ such that $c$ is increasing on $(t_1,+\infty)$ and for any $t>t_1$,
		\begin{flalign}\label{fanxiang-1}
			\begin{split}
				&\int_{\{\psi<-t\}}|F|^2e^{-\varphi}c(-\psi)\\
				=&\int_{\{\pi_1^*(G+\psi_0)<-t\}}|F|^2e^{\pi_1^*(-2\log|g_0|-2u_0+G+\psi_0)-\pi_2^*(\varphi_2)}c(\pi_1^*(-G-\psi_0))\\
				\geq&\int_{\{\pi_1^*(G+\Psi_l)<-t\}}|F|^2e^{\pi_1^*(-2\log|g_0|-2u_l+G+\psi_0)-\pi_2^*(\varphi_2)}c(\pi_1^*(-G-\Psi_l)).
			\end{split}
		\end{flalign}
		
		For any $\varepsilon>0$, $m\in\mathbb{N}$, there exists $s_0>0$ such that
		
		(1). for any $j\in\{1,2,\ldots,m\}$,
		\begin{equation*}
			U_j:=\{|w_j(z)|<s_0:z\in V_{z_j}\}\subset\subset V_{z_j};
		\end{equation*}
		
		(2). for any $j\in\{1,2,\ldots,m\}$, there exists a holomorphic function $\tilde{g}_j$ on $U_j$ such that $|\tilde{g}_j|^2=e^{\frac{G}{p_j}}$;
		
		(3). for any $j\in\{1,2,\ldots,m\}$,
		\begin{equation*}
			\sup_{z\in U_j} 2|u_l(z)-u_l(z_j)|<\varepsilon;
		\end{equation*}
		
		(4). for any $j\in\{1,2,\ldots,m\}$,
		\begin{equation*}
			\sup_{z\in U_j} |H_j(z)-H_j(z_j)|<\varepsilon,
		\end{equation*}
		where $H_j:=G-2p_j\log|w_j|+\Psi_l+\varepsilon$ is smooth function on $U_j$.
		
		Note that $G+\Psi_l\leq 2p_j\log|w_j|+H_j(z_j)$ on $U_j$, then it follows from Lemma \ref{l:G-compact} that there exists $t'>t_1$ such that
		\begin{equation*}
			(\{G+\Psi_l<-t'\}\cap U_j)\subset\subset U_j.
		\end{equation*}
		Then for $t>t'$, following from inequality (\ref{fanxiang-1}) we get
		\begin{flalign}\label{fanxiang-2}
			\begin{split}
				&\int_{\{\psi<-t\}}|F|e^{-\varphi}c(-\psi)\\
				\geq&\int_{\{\pi_1^*(G+\Psi_l)<-t\}}|F|^2e^{\pi_1^*(-2\log|g_0|-2u_l+G+\psi_0)-\pi_2^*(\varphi_2)}c(\pi_1^*(-G-\Psi_l))\\
				\geq&\sum_{1\leq j\leq m}\int_{(\{G+\Psi_l<-t\}\cap U_j)\times Y}|F|^2e^{\pi_1^*(-2\log|g_0|-2u_l+G+\psi_0)-\pi_2^*(\varphi_2)}c(\pi_1^*(-G-\Psi_l))\\
				\geq&\sum_{1\leq j\leq m}\int_{(\{2p_j\log |w_j|+H_j(z_j)<-t\}\cap U_j)\times Y}|F|^2e^{\pi_1^*(-2\log|g_0|-2u_l+G+\psi_0)-\pi_2^*(\varphi_2)}c(\pi_1^*(-G-\Psi_l))\\
				\geq&\sum_{1\leq j\leq m}\int_{(\{2p_j\log |w_j|+H_j(z_j)<-t\}\cap U_j)\times Y}|F|^2|\pi_1^*(\tilde{g}_j)|^{2p_j}e^{\pi_1^*(-2\log|g_0|-2u_l(z_j)-\varepsilon+\psi_0)-\pi_2^*(\varphi_2)}\\
				&\cdot c(\pi_1^*(-2p_j\log|w_j|-H_j(z_j))).
			\end{split}
		\end{flalign}
		
		For any $y\in Y$, we can choose some neighborhood $U_y\subset\subset Y$ such that $F=w_j^{\tilde{k}_j}\tilde{h}_j(w_j,w')dw_j\wedge dw'$ on $V_j\times U_y$, where $w'=(w'_1,\ldots,w'_{n-1})$ is a local coordinate on $U_y$ such that $w'_1(y)=\cdots=w'_{n-1}(y)=0$, and $\tilde{h}_j$ is a holomorphic function on $V_j\times U_y$. Note that $ord_{z_j}\tilde{g}_j=1$ for any $j$ ($1\leq j\leq m$), then we can assume that $\tilde{g}_j=\hat{d}_jw_j\hat{h}_j$ on $U_j$, where $\hat{d}_j\neq 0$ are constants, $\hat{h}_j$ are holomorphic functions on $U_j$ such that $\hat{h}_j(z_j)=1$ for $j\in \{1,2,\ldots,m\}$.
		
		Then it follows from inequality (\ref{fanxiang-2}) that
		\begin{flalign}
			\begin{split}
				&\int_{\{\psi_1<-t\}\times U_y}|F|^2e^{-\varphi}c(-\psi)\\
				\geq&\sum_{1\leq j\leq m}\int_{\{w_j\in\{2p_j\log |w_j|+H_j(z_j)<-t\}\cap U_j,w'\in U_y\}}|\tilde{h}_j(w_j,w')|^2|\tilde{g}_j(w_j)|^{2p_j}|w_j|^{2\tilde{k}_j}\\
				&\cdot e^{-2\log|g_0(w_j)|-2u_l(z_j)-\varepsilon+\psi_0(w_j))-\varphi_2(w')}c(-2p_j\log|w_j|-H_j(z_j))|dw_j\wedge dw'|^2\\
				\geq&\sum_{1\leq j\leq m}\frac{2|\hat{d}_j|^{2p_j}e^{-2u_l(z_j)-\varepsilon}}{|d_j|^2}\int_{w'\in U_y}\bigg(\int_0^{e^{-\frac{t+H_j(z_j)}{2p_j}}}\int_0^{2\pi}|r|^{2(\tilde{k}_j+p_j-k_j)+1}\\
				&\cdot\left|\frac{\tilde{h}_j(re^{i\theta},w')}{h_j(re^{i\theta})}\right|^2|\hat{h}_j(re^{i\theta})|^{2p_j}e^{\psi_0(re^{i\theta})}c(-2p_j\log r-H_j(z_j))d\theta dr\bigg)e^{-\varphi_2(w')}|dw'|^2\\
				\geq&\sum_{1\leq j\leq m}\frac{2\pi|\hat{d}_j|^{2p_j}}{p_j|d_j|^2}e^{-2u_l(z_j)-\varepsilon+\psi_0(z_j)-\frac{\tilde{k}_j-k_j+1+p_j}{p_j}H_j(z_j)}\\
				&\cdot\int_{w'\in U_y}\left(|\tilde{h}_j(z_j,w')|^2\int_t^{+\infty}c(s)e^{-\left(\frac{\tilde{k}_j+1-k_j}{p_j}+1\right)s}ds\right)e^{-\varphi_2(w')}|dw'|^2.
			\end{split}
		\end{flalign}
	where $k_j=ord_{z_j}(g)$, and $d_j=\lim_{z\rightarrow z_j}(g/w_j^{k_j})(z)$.
	
		Denote that
		\begin{equation*}
			I_m:=\{j\in\{1,2,\ldots,m\}:\tilde{k}_j+1-k_j\leq 0\}.
		\end{equation*}
		As
		 \[\liminf_{t\rightarrow+\infty}\frac{\int_{\{\psi<-t\}}|F|^2e^{-\varphi}c(-\psi)}{\int_t^{+\infty}c(s)e^{-s}ds}<+\infty,\]
		 $\int_0^{+\infty}c(s)e^{-s}ds<+\infty$, and $\tilde{h}_j(z_j,w')\not\equiv 0$ on $U_y$, we have
		\begin{equation}\label{k_j}
			\tilde{k}_j+1-k_j=0, \ \forall j\in I_m
		\end{equation}
		according to Lemma \ref{c(t)e^{-at}}. And we have
		\begin{flalign}\label{fanxiang-3}
			\begin{split}
				&\liminf_{t\rightarrow+\infty}\frac{\int_{\{\psi_1<-t\}\times U_y}|F|^2e^{-\varphi}c(-\psi)}{\int_t^{+\infty}c(s)e^{-s}ds}\\
				\geq&\sum_{j\in I_m}\frac{2\pi|\hat{d}_j|^{2p_j}}{p_j|d_j|^2}e^{-2u_l(z_j)-\varepsilon+\psi_0(z_j)-\frac{\tilde{k}_j-k_j+p_j+1}{p_j}H_j(z_j)}\\
				&\cdot\int_{w'\in U_y}\left(|\tilde{h}_j(z_j,w')|^2\cdot\lim_{t\rightarrow+\infty}\frac{\int_t^{+\infty}c(s)e^{-\left(\frac{\tilde{k}_j+1-k_j}{p_j}+1\right)s}ds}{\int_t^{+\infty}c(s)e^{-s}ds}\right)e^{-\varphi_2(w')}|dw'|^2\\
				=&\sum_{j\in I_m}\frac{2\pi|\hat{d}_j|^{2p_j}}{p_j|d_j|^2}e^{-2u_l(z_j)-\varepsilon+\psi_0(z_j)-H_j(z_j)}\int_{w'\in U_y}|\tilde{h}_j(z_j,w')|^2e^{-\varphi_2(w')}|dw'|^2.
			\end{split}
		\end{flalign}
		Note that
		\begin{flalign*}
			\begin{split}
				H_j(z_j)=&\Psi_l(z_j)+\varepsilon+\lim_{z\rightarrow z_j}(G-2p_j\log|w_j|)\\
				=&\Psi_l(z_j)+\varepsilon+\log\left(\lim_{z\rightarrow z_j}\frac{|\tilde{g}|^{2p_j}}{|w_j|^{2p_j}}\right)\\
				=&\Psi_l(z_j)+\varepsilon+2p_j\log|\hat{d}_j|.
			\end{split}
		\end{flalign*}
		Letting $\varepsilon\rightarrow 0+$, from inequality (\ref{fanxiang-3}) we get
		\begin{flalign}\label{fanxiang-4}
			\begin{split}
				\liminf_{t\rightarrow+\infty}\frac{\int_{\{\psi_1<-t\}\times U_y}|F|^2e^{-\varphi}c(-\psi)}{\int_t^{+\infty}c(s)e^{-s}ds}\geq&\sum_{j\in I_m}\frac{2\pi}{p_j|d_j|^2}e^{-2u_l(z_j)+\psi_0(z_j)-\Psi_l(z_j)}\\
				&\cdot\int_{w'\in U_y}|\tilde{h}_j(z_j,w')|^2e^{-\varphi_2(w')}|dw'|^2.
			\end{split}
		\end{flalign}
		Letting $l\rightarrow +\infty$, from inequality (\ref{fanxiang-4}) we get
		\begin{flalign}\label{fanxiang-5}
			\begin{split}
				\liminf_{t\rightarrow+\infty}\frac{\int_{\{\psi_1<-t\}\times U_y}|F|^2e^{-\varphi}c(-\psi)}{\int_t^{+\infty}c(s)e^{-s}ds}\geq&\sum_{j\in I_m}\frac{2\pi}{p_j|d_j|^2}e^{-2u_0(z_j)}\\
				&\cdot\int_{w'\in U_y}|\tilde{h}_j(z_j,w')|^2e^{-\varphi_2(w')}|dw'|^2.
			\end{split}
		\end{flalign}
		Letting $m\rightarrow +\infty$, from inequality (\ref{fanxiang-5}) we get
		\begin{flalign}\label{fanxiang-6}
			\begin{split}
				\liminf_{t\rightarrow+\infty}\frac{\int_{\{\psi_1<-t\}\times U_y}|F|^2e^{-\varphi}c(-\psi)}{\int_t^{+\infty}c(s)e^{-s}ds}\geq&\sum_{j\in I_F}\frac{2\pi}{p_j|d_j|^2}e^{-2u_0(z_j)}\\
				&\cdot\int_{w'\in U_y}|\tilde{h}_j(z_j,w')|^2e^{-\varphi_2(w')}|dw'|^2.
			\end{split}
		\end{flalign}
	 And we also have $\tilde{k}_j+1-k_j=0$ for any $j\in I_F$ from equality (\ref{k_j}). According to the arbitrariness of $y$ and $U_y$, from the above calculations in the local cases we get
		\begin{equation*}
			\liminf_{t\rightarrow+\infty}\frac{\int_{\{\psi<-t\}}|F|^2e^{-\varphi}c(-\psi)}{\int_t^{+\infty}c(s)e^{-s}ds}\geq\sum_{j\in I_F}\frac{2\pi e^{-2u_0(z_j)}}{p_j|d_j|^2}\int_Y|\tilde{F}_j|^2e^{-\varphi_2}.
		\end{equation*}
	 Especially, we have $\int_Y|\tilde{F}_j|^2e^{-\varphi_2}<+\infty$ and $u_0(z_j)>-\infty$ for any $j\in I_F$.
	\end{proof}

	\section{Proof of Theorem \ref{one-p} and Remark \ref{rem-p}}
	In this section, we give the proof of Theorem \ref{one-p} and Remark \ref{rem-p}.
	
	\subsection{Proof of the sufficiency in Theorem \ref{one-p} and Remark \ref{rem-p}}
	\
	we give the proof of the sufficiency in Theorem \ref{one-p} and the proof of Remark \ref{rem-p}.
	
	\begin{proof}[Proof of the sufficiency in Theorem \ref{one-p} and Remark \ref{rem-p}]
		It follows from the statement (2) in Theorem \ref{one-p} that we can assume that $f=\pi_1^*(w^kdw)\wedge \pi_2^*(f_{Y})$ on $V_{z_0}\times Y$. Denote that
		\begin{equation*}
			F:=c_0\pi_1^*(gP_*(f_udf_{z_0}))\wedge\pi_2^*(f_{Y}),
		\end{equation*}
		where $f_u$ is a holomorphic function on $\Delta$ such that $|f_u|=P^*(e^u)$, $f_{z_0}$ is a holomorphic function on $\Delta$ such that $|f_{z_0}|=P^*(e^{G_{\Omega}(\cdot,z_0)})$, and
		\begin{equation*}
			c_0:=\lim_{z\rightarrow z_0}\frac{w^kdw}{gP_*(f_udf_{z_0})}\in\mathbb{C}\setminus\{0\}.
		\end{equation*}
		Then Lemma \ref{f1zf2w} implies that $(F-f,(z_0,y))\in (\mathcal{O}(K_M))_{(z_0,y)}\otimes\mathcal{I}(\varphi+\psi)_{(z_0,y)}$ for any $(z_0,y)\in Z_0$.
		
		For any $y\in Y$, let $w'=(w_1',\ldots,w_{n-1}')$ be a local coordinate on a neighborhood $U_y$ of $y$ satisfying $w'_j(y)=0$ for any $j\in \{1,\ldots,n-1\}$. Assume that $f_{Y}=h(w') dw'$ on $U_y$, where $h$ is a holomorphic function on $U_y$, $dw'=dw_1'\wedge\cdots\wedge dw_{n-1}'$. Then we have
		\begin{equation}
			f=w^kh(w')dw\wedge dw'	
		\end{equation}
		on $V_{z_0}\times U_y$, and	
		\begin{equation}
			F=c_0\pi_1^*(gP_*(f_udf_{z_0}))\wedge h(w')dw'
		\end{equation}
		on $\Omega\times U_y$.
		
		For any $w'\in U_y$, $t\geq 0$, denote that
		\begin{flalign*}
			\begin{split}
				G_{w'}(t):=\inf\bigg\{\int_{\{\psi_1<-t\}}|\tilde{f}|^2e^{-\varphi_1}c(-\psi_1) : (\tilde{f}-h(w')w^kdw,z_0)\in (\mathcal{O}(K_{\Omega}))_{z_0}\otimes\mathcal{I}(\varphi_1+\psi_1)_{z_0}&\\
				\& \ \tilde{f}\in H^0(\{\psi_1<-t\},\mathcal{O}(K_{\Omega}))\bigg\}.&
			\end{split}
		\end{flalign*}
		Denote that
		\begin{equation*}
			F_{w'}:=c_0h(w')gP_*(f_udf_{z_0})
		\end{equation*}
		on $\Omega$. Then it follows from Theorem \ref{thm:m-points} and Remark \ref{rem-finite-1d} that $G_{w'}(h^{-1}(r))$ is linear and
		\begin{flalign}\label{Gw't}
			\begin{split}
				G_{w'}(t)&=\int_{\{\psi_1<-t\}}|F_{w'}|^2e^{-\varphi_1}c(-\psi_1)\\
				&=\left(\int_t^{+\infty}c(s)e^{-s}ds\right)\frac{2\pi e^{-2u(z_0)}}{p|d|^2}|h(w')|^2
			\end{split}
		\end{flalign}
		for any $t\geq 0$, $w'\in U_y$, where
		\begin{equation*}
			d:=\lim_{z\rightarrow z_0}\frac{g}{w^k}(z).
		\end{equation*}
		Then following from the Fubini's Theorem, we have
		\begin{flalign*}
			\begin{split}
				\int_{\{\psi_1<-t\}\times U_y}|F|^2e^{-\varphi}c(-\psi)&=\left(\int_t^{+\infty}c(s)e^{-s}ds\right)\frac{2\pi e^{-2u(z_0)}}{p|d|^2}\int_{U_y}|h(w')|^2e^{-\varphi_2(w')}|dw'|^2\\
				&=\left(\int_t^{+\infty}c(s)e^{-s}ds\right)\frac{2\pi e^{-2u(z_0)}}{p|d|^2}\int_{U_y}|f_Y|^2e^{-\varphi_2}.
			\end{split}
		\end{flalign*}
		According to the arbitrariness of $y$ and $U_y$, for general $Y$ and any $t\geq 0$ we have
		\begin{equation}
			\int_{\{\psi<-t\}}|F|^2e^{-\varphi}c(-\psi)=\left(\int_t^{+\infty}c(s)e^{-s}ds\right)\frac{2\pi e^{-2u(z_0)}}{p|d|^2}\int_Y|f_Y|^2e^{-\varphi_2}.
		\end{equation}
		
		Assume that $t\geq 0$. According to Lemma \ref{F_t}, let $\tilde{F}$ be the holomorphic $(n,0)$ form on $\{\psi<-t\}$ such that $(\tilde{F}-f,(z_0,y))\in (\mathcal{O}(K_M))_{(z_0,y)}\otimes\mathcal{I}(\varphi+\psi)_{(z_0,y)}$ for any $(z_0,y)\in Z_0$ and
		\begin{equation*}
			\int_{\{\psi<-t\}}|\tilde{F}|^2e^{-\varphi}c(-\psi)=G(t;c).
		\end{equation*}
		Following from Lemma \ref{decomp}, we assume that
		\begin{equation*}
			\tilde{F}=\sum_{j=\tilde{k}}^{\infty}\pi_1^*(w^jdw)\wedge\pi_2^*(\tilde{F}_j)
		\end{equation*}
		on $V_{z_0}\times Y$, where $\tilde{k}\in\mathbb{N}$, $\tilde{F}_j$ is a holomorphic $(n-1,0)$ form on $Y$ for any $j\geq \tilde{k}$, and $\tilde{F}_{\tilde{k}}\not\equiv 0$. Now we prove that $\tilde{k}=k$ and $\tilde{F}_{\tilde{k}}=f_Y$.
		
		For any $y\in Y$, let $w'=(w_1',\ldots,w_{n-1}')$ be a local coordinate on a neighborhood $U_y$ of $y$ satisfying $w'_j(y)=0$ for any $j\in \{1,\ldots,n-1\}$. Assume that $f_{Y}=h(w') dw'$ on $U_y$, $\tilde{F}=w^{\tilde{k}}\tilde{h}(w,w')dw\wedge dw'$ on $V_{z_0}\times U_y$, where $h$ is a holomorphic function on $U_y$, $\tilde{h}$ is a holomorphic function on $V_{z_0}\times U_y$. Then we have
		\begin{equation*}
			f=w^kh(w')dw\wedge dw', \ \tilde{F}=w^{\tilde{k}}\tilde{h}(w,w')dw\wedge dw'
		\end{equation*}
		on $V_{z_0}\times U_y$, and $\tilde{F}_{\tilde{k}}=\tilde{h}(z_j,w')dw'$ on $U_y$. According to Lemma \ref{local-germ}, we have $w^{\tilde{k}}\tilde{h}(w,w')-w^kh(w')\in\mathcal{I}(\varphi_1+\psi_1)_{z_0}$, which implies that $\tilde{k}=k$ and $\tilde{h}(z_0,w')=h(w')$ if $h(w')\neq 0$. Since $\int_Y|f_Y|e^{-\varphi_2}\in (0,+\infty)$ and $h$ is holomorphic, we have $\tilde{k}=k$ and $\tilde{h}(z_0,w')=h(w')$ for any $w'\in U_y$. According to the arbitrariness of $y$ and $U_y$, we get that $\tilde{k}=k$ and $\tilde{F}_{\tilde{k}}=f_Y$.
		
		Besides, following from Lemma \ref{decomp}, we can also assume that
		\begin{equation*}
			\tilde{F}=\sum_{\alpha'\in\mathbb{N}^{n-1}}\pi_1^*(\tilde{F}_{\alpha'})\wedge\pi_2^*({w'}^{\alpha'}dw')
		\end{equation*}
		on $\{\psi_1<-t\}\times U_y$, where $\tilde{F}_{\alpha'}$ is a holomorphic $(1,0)$ form on $\{\psi_1<-t\}$ for any $\alpha'\in\mathbb{N}^{n-1}$. According to the discussions above, we can assume that
		\begin{equation*}
			\tilde{F}_{\alpha'}=w^k\tilde{f}_{\alpha'}dw
		\end{equation*}
		on $V_{z_0}\cap\{\psi_1<-t\}$, where $\tilde{f}_{\alpha'}$ is a holomorphic function on $V_{z_0}\cap\{\psi_1<-t\}$. And we also have
		\begin{equation*}
			\sum_{\alpha'\in\mathbb{N}^{n-1}}\tilde{f}_{\alpha'}(z_0){w'}^{\alpha'}=\tilde{h}(z_0,w')=h(w'),
		\end{equation*}
		which implies that
		\begin{equation*}
			\sum_{\alpha'\in\mathbb{N}^{n-1}}\tilde{f}_{\alpha'}(z_0){w'}^{\alpha'}dw'=f_Y
		\end{equation*}
		on $U_y$. Then by the definition of $G_{w'}(t)$, it follows from equality (\ref{Gw't}) that
		\begin{equation*}
			\int_{\{\psi_1<-t\}}\left|\sum_{\alpha'\in\mathbb{N}^{n-1}}\tilde{F}_{\alpha'}{w'}^{\alpha'}\right|^2e^{-\varphi_1}c(-\psi_1)\geq \left(\int_t^{+\infty}c(s)e^{-s}ds\right)\frac{2\pi e^{-2u(z_0)}}{p|d|^2}|h(w')|^2.
		\end{equation*}
		Then we have
		\begin{flalign*}
			\begin{split}
				&\int_{\{\psi_1<-t\}\times U_y}|\tilde{F}|^2e^{-\varphi}c(-\psi)\\
				=&\int_{\{\psi_1<-t\}\times U_y}\left|\sum_{\alpha'\in\mathbb{N}^{n-1}}\pi_1^*(\tilde{F}_{\alpha'})\wedge\pi_2^*({w'}^{\alpha'}dw')\right|^2e^{-\varphi}c(-\psi)\\
				=&\int_{w'\in U_y}\left(\int_{\{\psi_1<-t\}}\left|\sum_{\alpha'\in\mathbb{N}^{n-1}}\tilde{F}_{\alpha'}{w'}^{\alpha'}\right|^2e^{-\varphi_1}c(-\psi_1)\right)e^{-\varphi_2(w')}|dw'|^2\\
				\geq&\int_{w'\in U_y}\left(\int_t^{+\infty}c(s)e^{-s}ds\right)\frac{2\pi e^{-2u(z_0)}}{p|d|^2}|h(w')|^2e^{-\varphi_2(w')}|dw'|^2\\
				=&\int_{\{\psi_1<-t\}\times U_y}|F|^2e^{-\varphi}c(-\psi).
			\end{split}
		\end{flalign*}
		According to the arbitrariness of $y$ and $U_y$, we get that
		\begin{equation*}
			\int_{\{\psi<-t\}}|\tilde{F}|^2e^{-\varphi}c(-\psi)\geq\int_{\{\psi<-t\}}|F|^2e^{-\varphi}c(-\psi),
		\end{equation*}
		which implies that $F\equiv\tilde{F}$ on $\{\psi<-t\}$ by the choice of $\tilde{F}$.
		
		Now we get that
		\begin{equation}
			G(t;c)=\int_{\{\psi<-t\}}|F|^2e^{-\varphi}c(-\psi)=\left(\int_t^{+\infty}c(s)e^{-s}ds\right)\frac{2\pi e^{-2u(z_0)}}{p|d|^2}\int_Y|f_Y|^2e^{-\varphi_2}
		\end{equation}
		for any $t\geq 0$. Then $G(h^{-1}(r))$ is linear with respect to $r\in (0,\int_0^{+\infty}c(s)e^{-s}ds]$, and $F=c_0\pi_1^*(gP_*(f_udf_{z_0}))\wedge\pi_2^*(f_{Y})$ is the unique holomorphic $(n,0)$ form $F$ on $M$ such that $(F-f,(z_0,y))\in (\mathcal{O}(K_M))_{(z_0,y)}\otimes\mathcal{I}(\varphi+\psi)_{(z_0,y)}$ for any $(z_0,y)\in Z_0$ and
		\begin{equation*}
			G(t)=\int_{\{\psi<-t\}}|F|^2e^{-\varphi}c(-\psi)=\left(\int_t^{+\infty}c(s)e^{-s}ds\right)\frac{2\pi e^{-2u(z_0)}}{p|d|^2}\int_{Y}|f_{Y}|^2e^{-\varphi_2}
		\end{equation*}
		for any $t\geq 0$. Then the proof of the sufficiency in Theorem \ref{one-p} and the proof of Remark \ref{rem-p} are done.
	\end{proof}
	
	\subsection{Proof of the necessity in Theorem \ref{one-p}}
	\
	we give the proof of the necessity in Theorem \ref{one-p}.
	
	\begin{proof}[Proof of the necessity in Theorem \ref{one-p}]
		Assume that $G(h^{-1}(r))$ is linear with respect to $r\in (0,\int_0^{+\infty}c(s)e^{-s}ds]$. Then according to Lemma \ref{linear}, there exists a unique holomorphic $(n,0)$ form $F$ on $M$ satisfying $(F-f)\in H^0(Z_0,(\mathcal{O}(K_M)\otimes\mathcal{I}(\varphi+\psi))|_{Z_0})$, and $G(t;c)=\int_{\{\psi<-t\}}|F|^2e^{-\varphi}c(-\psi)$ for any $t\geq 0$. Then according to Lemma \ref{linear}, Remark \ref{ctildec} and Lemma \ref{tildecincreasing}, we can assume that $c$ is increasing near $+\infty$.
		
		Following from Lemma \ref{decomp}, we assume that
		\begin{equation*}
			F=\sum_{j=k}^{\infty}\pi_1^*(w^jdw)\wedge \pi_2^*(F_j)
		\end{equation*}
		on $V_{z_0}\times Y$, where $k\in\mathbb{N}$, $F_j$ is a holomorphic $(n-1,0)$ form on $Y$ for any $j\geq k$, and $F_k\not\equiv 0$.
		
		Using the Weierstrass Theorem on open Riemann surfaces (see \cite{OF81}) and
		the Siu’s Decomposition Theorem, we have
		\[\varphi_1+\psi_1=2\log |g|+2\log G_{\Omega}(\cdot,z_0)+2u,\]
		where $g$ is a holomorphic function on $\Omega$ and $u$ is a subharmonic function on $\Omega$ such that $v(dd^cu,z)\in [0,1)$ for any $z\in\Omega$. Denote that $k_0:=ord_{z_0}(g)$.
		
		We prove that $k=k_0$. Firstly, as $G(h^{-1}(r))$ is linear and $c$ is increasing near $+\infty$, using Lemma \ref{fanxiangineq}, we can get that $k\geq k_0$ and $\int_Y|F_k|^2e^{-\varphi_2}<+\infty$. Secondly, if $k>k_0$, then it follows from Lemma \ref{k>k0} that $(F,(z_0,y))\in(\mathcal{O}(K_M)\otimes\mathcal{I}(\varphi+\psi))_{(z_0,y)}$ for any $(z_0,y)\in Z_0$, which contradicts to $G(0)\neq 0$ according to Lemma \ref{G(t)=0}. Thus we get that $k=k_0$.
		
		Since $k=k_0$, according to Lemma \ref{e-varphic-psi}, Lemma \ref{f1zf2w} and Lemma \ref{k>k0}, we have
		\begin{equation}
			(F-\pi_1^*(w^kdw)\wedge\pi_2^*(F_k),(z_0,y))\in(\mathcal{O}(K_M)\otimes\mathcal{I}(\varphi+\psi))_{(z_0,y)}
		\end{equation}
		for any $(z_0,y)\in Z_0$. Then it follows from $(F-f,(z_0,y))\in(\mathcal{O}(K_M)\otimes\mathcal{I}(\varphi+\psi))_{(z_0,y)}$ that $f=\pi_1^*(w^kdw)\wedge\pi_2^*(F_k)+f_0$, where $(f_0,(z_0,y))\in (\mathcal{O}(K_M))_{(z_0,y)}\otimes\mathcal{I}(\varphi+\psi)_{(z_0,y)}$ for any $(z_0,y)\in Z_0$.
		
		For $t\geq 0$, denote that	
		\begin{flalign*}
			\begin{split}
				G_{\Omega}(t):=\inf\bigg\{\int_{\{\psi_1<-t\}}|\tilde{f}|^2e^{-\varphi_1}c(-\psi_1) : (\tilde{f}-w^kdw,z_0)\in (\mathcal{O}(K_{\Omega}))_{z_0}\otimes\mathcal{I}(\varphi_1+\psi_1)_{z_0}&\\
				\& \ \tilde{f}\in H^0(\{\psi_1<-t\},\mathcal{O}(K_{\Omega}))\bigg\}.&
			\end{split}
		\end{flalign*}
		Then it follows from Lemma \ref{F_t} that there exists a unique holomorphic $(1,0)$ form $\tilde{f}_t$ on $\{\psi_1<-t\}$ such that $(\tilde{f}_t-w^kdw,z_0)\in (\mathcal{O}(K_{\Omega}))_{z_0}\otimes\mathcal{I}(\varphi_1+\psi_1)_{z_0}$ and
		\begin{equation*}
			G_{\Omega}(t)=\int_{\{\psi_1<-t\}}|\tilde{f}_t|^2e^{-\varphi_1}c(-\psi_1)
		\end{equation*}
		for any $t\geq 0$. Following from Lemma \ref{L2ext-finite-f} (when $Y$ is a single point), we have
		\begin{equation*}
			\frac{\int_{\{\psi_1<-t\}}|\tilde{f}_t|^2e^{-\varphi_1}c(-\psi_1)}{\int_t^{+\infty}c(s)e^{-s}ds}=\frac{G_{\Omega}(t)}{\int_t^{+\infty}c(s)e^{-s}ds}\leq \frac{2\pi e^{-2u(z_0)}}{p|d|^2}.
		\end{equation*}
		Denote that $\tilde{F}_t:=\pi_1^*(\tilde{f}_t)\wedge\pi_2^*(F_k)$. Then according to the Fubini's theorem we can get that
		\begin{equation}\label{ne-1}
			\frac{G(t)}{\int_t^{+\infty}c(s)e^{-s}ds}\leq\frac{\int_{\{\psi<-t\}}|\tilde{F}_t|^2e^{-\varphi}c(-\psi)}{\int_t^{+\infty}c(s)e^{-s}ds}\leq \frac{2\pi e^{-2u(z_0)}}{p|d|^2}\int_Y|F_k|^2e^{-\varphi},
		\end{equation}
		which means that
		\begin{equation}\label{ne-2}
			\int_{\{\psi<-t\}}|F|^2e^{-\varphi}c(-\psi)\leq\left(\int_t^{+\infty}c(s)e^{-s}ds\right)\frac{2\pi e^{-2u(z_0)}}{p|d|^2}\int_Y|F_k|^2e^{-\varphi}.
		\end{equation}
		Besides, according to $ord_{z_0}(g)=k$ and Lemma \ref{fanxiangineq}, we have
		\begin{equation}\label{ne-3}
			\int_{\{\psi<-t\}}|F|^2e^{-\varphi}c(-\psi)\geq\left(\int_t^{+\infty}c(s)e^{-s}ds\right)\frac{2\pi e^{-2u(z_0)}}{p|d|^2}\int_Y|F_k|^2e^{-\varphi}
		\end{equation}
		for any $t\geq 0$. Combining inequality (\ref{ne-1}), inequality (\ref{ne-2}) and inequality (\ref{ne-3}), we have
		\begin{flalign}\label{ne-4}
			\begin{split}
				&\frac{G(t)}{\int_t^{+\infty}c(s)e^{-s}ds}\\
				=&\frac{\int_{\{\psi<-t\}}|F|^2e^{-\varphi}c(-\psi)}{\int_t^{+\infty}c(s)e^{-s}ds}\\
				=&\frac{\int_{\{\psi<-t\}}|F_t|^2e^{-\varphi}c(-\psi)}{\int_t^{+\infty}c(s)e^{-s}ds}\\
				=&\frac{2\pi e^{-2u(z_0)}}{p|d|^2}\int_Y|F_k|^2e^{-\varphi}.
			\end{split}
		\end{flalign}
		According to the uniqueness of $F$, we have $F\equiv F_t$ on $\{\psi<-t\}$. In addition, we have
		\begin{equation*}
			\int_{\{\psi<-t\}}|F_t|^2e^{-\varphi}c(-\psi)=\int_{\{\psi_1<-t\}}|\tilde{f}_t|^2e^{-\varphi_1}c(-\psi_1)\cdot\int_Y|F_k|^2e^{-\varphi_2}
		\end{equation*}
		according to the Fubini's Theorem. It means that
		\begin{equation*}
			G_{\Omega}(t)=\int_{\{\psi_1<-t\}}|\tilde{f}_t|^2e^{-\varphi_1}c(-\psi_1)=\left(\int_t^{+\infty}c(s)e^{-s}ds\right)\frac{2\pi e^{-2u(z_0)}}{p|d|^2}
		\end{equation*}
		for any $t\geq 0$, i.e. $G_{\Omega}(h^{-1}(r))$ is linear with respect to $r\in (0,\int_0^{+\infty}c(s)e^{-s}ds]$. It follows from Theorem \ref{thm:m-points} that the statements (1), (3) and (4) in Theorem \ref{one-p} hold. Denote that $f_Y:=F_k$, then $\int_Y|f_Y|^2e^{-\varphi_2}\in (0,+\infty)$.
		
		Now the proof of the necessity of Theorem \ref{one-p} is done.
		
	\end{proof}
	
	\section{Proofs of Theorem \ref{finite-p} and Remark \ref{rem-finite}}
	In this section, we give the proofs of Theorem \ref{finite-p} and Remark \ref{rem-finite}.
	
	\subsection{Proofs of the sufficiency in Theorem \ref{finite-p} and Remark \ref{rem-finite}}
	\
	we give the proof of the sufficiency in Theorem \ref{finite-p} and the proof of Remark \ref{rem-finite}.
	
	\begin{proof}[Proofs of the sufficiency in Theorem \ref{finite-p} and Remark \ref{rem-finite}]
		We may assume that $f=\pi_1^*(a_jw_j^{k_j}dw_j)\wedge \pi_2^*(f_Y)$ on $V_{z_j}\times Y$ for any $j\in\{1,2,\ldots,m\}$ according to the statement (2) in Theorem \ref{finite-p}. Denote that
		\begin{equation*}
			F:=c_0\pi_1^*\left(gP_*\left(f_u\left(\prod\limits_{l=1}^mf_{z_l}\right)\left(\sum\limits_{l=1}^mp_l\dfrac{d{f_{z_{l}}}}{f_{z_{l}}}\right)\right)\right)\wedge\pi_2^*(f_{Y}),
		\end{equation*}
		where $f_u$ is a holomorphic function on $\Delta$ such that $|f_u|=P^*(e^u)$, $f_{z_j}$ is a holomorphic function on $\Delta$ such that $|f_{z_j}|=P^*(e^{G_{\Omega}(\cdot,z_j)})$ for any $j\in\{1,2,\ldots,m\}$, and $c_0$ is the constant in statement (5) of Theorem \ref{finite-p}.
		
		Then Theorem \ref{thm:m-points}, Remark \ref{rem-finite-1d} and Lemma \ref{f1zf2w} implie that $(F-f,(z_j,y))\in (\mathcal{O}(K_M))_{(z_j,y)}\otimes\mathcal{I}(\varphi+\psi)_{(z_j,y)}$ for any $(z_j,y)\in Z_0$.
		
		For any $y\in Y$, let $w'=(w_1',\ldots,w_{n-1}')$ be a local coordinate on a neighborhood $U_y$ of $y$ satisfying $w'_l(y)=0$ for any $l\in \{1,\ldots,n-1\}$. Assume that $f_{Y}=h(w') dw'$ on $U_y$, where $h$ is a holomorphic function on $U_y$, $dw'=dw_1'\wedge\cdots\wedge dw_{n-1}'$. Then we have
		\begin{equation}
			f=a_jw_j^{k_j}h(w')dw_j\wedge dw'
		\end{equation}
		on $V_{z_j}\times U_y$, and	
		\begin{equation}
			F=c_0\pi_1^*\left(gP_*\left(f_u\left(\prod\limits_{l=1}^mf_{z_l}\right)\left(\sum\limits_{l=1}^mp_l\dfrac{d{f_{z_{l}}}}{f_{z_{l}}}\right)\right)\right)\wedge h(w')dw'
		\end{equation}
		on $\Omega\times U_y$.
		
		For any $w'\in U_y$, $t\geq 0$, denote that
		\begin{flalign*}
			\begin{split}
				G_{w'}(t):=\inf\bigg\{\int_{\{\psi_1<-t\}}|\tilde{f}|^2e^{-\varphi_1}c(-\psi_1) : (\tilde{f}-a_jh(w')w_j^{k_j}dw,z_j)\in (\mathcal{O}(K_{\Omega}))_{z_j}\otimes\mathcal{I}(\varphi_1+\psi_1)_{z_j}&\\
				\text{for \ any\ } j\in\{1,2,\ldots,m\}\& \ \tilde{f}\in H^0(\{\psi_1<-t\},\mathcal{O}(K_{\Omega}))\bigg\}.&
			\end{split}
		\end{flalign*}
		Denote that
		\begin{equation*}
			F_{w'}:=c_0h(w')gP_*\left(f_u\left(\prod\limits_{l=1}^mf_{z_l}\right)\left(\sum\limits_{l=1}^mp_l\dfrac{d{f_{z_{l}}}}{f_{z_{l}}}\right)\right)
		\end{equation*}
		on $\Omega$. Then it follows from Theorem \ref{thm:m-points} and Remark \ref{rem-finite-1d} that $G_{w'}(h^{-1}(r))$ is linear and
		\begin{flalign}\label{Gw't-finite}
			\begin{split}
				G_{w'}(t)&=\left(\int_t^{+\infty}c(s)e^{-s}ds\right)\left(\sum_{j=1}^m\frac{2\pi |a_j|^2e^{-2u(z_j)}}{p_j|d_j|^2}\right)|h(w')|^2
			\end{split}
		\end{flalign}
		for any $t\geq 0$, $w'\in U_y$, where
		\begin{equation*}
			d_j:=\lim_{z\rightarrow z_j}\frac{g}{w_j^{k_j}}(z).
		\end{equation*}
		Then following from the Fubini's Theorem, we have
		\begin{flalign*}
			\begin{split}
				\int_{\{\psi_1<-t\}\times U_y}|F|^2e^{-\varphi}c(-\psi)&=\left(\int_t^{+\infty}c(s)e^{-s}ds\right)\left(\sum_{j=1}^m\frac{2\pi |a_j|^2e^{-2u(z_j)}}{p_j|d_j|^2}\right)\int_{U_y}|f_Y|^2e^{-\varphi_2}.
			\end{split}
		\end{flalign*}
		According to the arbitrariness of $y$ and $U_y$, for general $Y$ and any $t\geq 0$ we have
		\begin{equation}
			\int_{\{\psi<-t\}}|F|^2e^{-\varphi}c(-\psi)=\left(\int_t^{+\infty}c(s)e^{-s}ds\right)\left(\sum_{j=1}^m\frac{2\pi |a_j|^2e^{-2u(z_j)}}{p_j|d_j|^2}\right)\int_Y|f_Y|^2e^{-\varphi_2}.
		\end{equation}
		
		Assume that $t\geq 0$. According to Lemma \ref{F_t}, let $\tilde{F}$ be the holomorphic $(n,0)$ form on $\{\psi<-t\}$ such that $(\tilde{F}-f,(z_j,y))\in (\mathcal{O}(K_M))_{(z_j,y)}\otimes\mathcal{I}(\varphi+\psi)_{(z_j,y)}$ for any $(z_j,y)\in Z_0$ and
		\begin{equation*}
			\int_{\{\psi<-t\}}|\tilde{F}|^2e^{-\varphi}c(-\psi)=G(t;c).
		\end{equation*}
		Following from Lemma \ref{decomp}, we assume that
		\begin{equation*}
			\tilde{F}=\sum_{l=\tilde{k}_j}^{\infty}\pi_1^*(w_j^ldw_j)\wedge \pi_2^*(F_{j,l})
		\end{equation*}
		on $V_{z_j}\times Y$, where $\tilde{k}_j\in\mathbb{N}$, $\tilde{F}_{j,l}$ is a holomorphic $(n-1,0)$ form on $Y$ for any $l\geq \tilde{k}_j$, and $\tilde{F}_{j,\tilde{k}_j}\not\equiv 0$. Now we prove that $\tilde{k}_j=k_j$ and $\tilde{F}_{j,\tilde{k}_j}=a_jf_Y$.
		
		For any $y\in Y$, let $w'=(w_1',\ldots,w_{n-1}')$ be a local coordinate on a neighborhood $U_y$ of $y$ satisfying $w'_j(y)=0$ for any $j\in \{1,\ldots,n-1\}$. Assume that $f_{Y}=h(w') dw'$ on $U_y$, $F=w_j^{\tilde{k}_j}\tilde{h}_j(w_j,w')dw'$ on $V_{z_j}\times U_y$, where $h$ is a holomorphic function on $U_y$, $\tilde{h}_j$ is a holomorphic function on $V_{z_j}\times U_y$. Then we have
		\begin{equation*}
			f=a_jw_j^{k_j}h(w')dw_j\wedge dw', \ \tilde{F}=w_j^{\tilde{k}_j}\tilde{h}_j(w_j,w')dw_j\wedge dw'
		\end{equation*}
		on $V_{z_j}\times U_y$, and $\tilde{F}_{j,\tilde{k}_j}=\tilde{h}_j(z_j,w')dw'$ on $U_y$. According to Lemma \ref{local-germ}, we have $w_j^{\tilde{k}_j}\tilde{h}_j(w_j,w')-a_jw_j^{k_j}h(w')\in\mathcal{I}(\varphi_1+\psi_1)_{z_j}$, which implies that $\tilde{k}_j=k_j$ and $\tilde{h}_j(z_j,w')=a_jh(w')$ if $h(w')\neq 0$. Since $\int_Y|f_Y|^2e^{-\varphi_2}\in (0,+\infty)$ and $h$ is holomorphic, we have $\tilde{k}_j=k_j$ and $\tilde{h}_j(z_j,w')=a_jh(w')$ for any $w'\in U_y$. According to the arbitrariness of $y$ and $U_y$, we get that $\tilde{k}_j=k_j$ and $\tilde{F}_{j,\tilde{k}_j}=a_jf_Y$.
		
		Besides, following from Lemma \ref{decomp}, we can also assume that
		\begin{equation*}
			\tilde{F}=\sum_{\alpha'\in\mathbb{N}^{n-1}}\pi_1^*(\tilde{F}_{\alpha'})\wedge\pi_2^*({w'}^{\alpha'}dw')
		\end{equation*}
		on $\{\psi_1<-t\}\times U_y$, where $\tilde{F}_{\alpha'}$ is a holomorphic $(1,0)$ form on $\{\psi_1<-t\}$ for any $\alpha'\in\mathbb{N}^{n-1}$. According to the discussions above, we can assume that
		\begin{equation*}
			\tilde{F}_{\alpha'}=w_j^{k_j}\tilde{f}_{j,\alpha'}dw_j
		\end{equation*}
		on $V_{z_j}\cap\{\psi_1<-t\}$, where $\tilde{f}_{j,\alpha'}$ is a holomorphic function on $V_{z_j}\cap\{\psi_1<-t\}$. And we also have
		\begin{equation*}
			\sum_{\alpha'\in\mathbb{N}^{n-1}}\tilde{f}_{j,\alpha'}(z_j){w'}^{\alpha'}=\tilde{h}_j(z_j,w')=a_jh(w'),
		\end{equation*}
		which implies that
		\begin{equation*}
			\sum_{\alpha'\in\mathbb{N}^{n-1}}\tilde{f}_{j,\alpha'}(z_j){w'}^{\alpha'}dw'=a_jf_Y
		\end{equation*}
		on $U_y$. Then by the definition of $G_{w'}(t)$, it follows from equality (\ref{Gw't-finite}) that
		\begin{flalign*}
			\begin{split}
				\int_{\{\psi_1<-t\}}\left|\sum_{\alpha'\in\mathbb{N}^{n-1}}\tilde{F}_{\alpha'}{w'}^{\alpha'}\right|^2e^{-\varphi_1}c(-\psi_1)\geq& \left(\int_t^{+\infty}c(s)e^{-s}ds\right)\\
				&\cdot\left(\sum_{j=1}^m\frac{2\pi |a_j|^2e^{-2u(z_j)}}{p_j|d_j|^2}\right)|h(w')|^2.
			\end{split}
		\end{flalign*}
		Then we have
		\begin{flalign*}
			\begin{split}
				&\int_{\{\psi_1<-t\}\times U_y}|\tilde{F}|^2e^{-\varphi}c(-\psi)\\
				=&\int_{\{\psi_1<-t\}\times U_y}\left|\sum_{\alpha'\in\mathbb{N}^{n-1}}\pi_1^*(\tilde{F}_{\alpha'})\wedge\pi_2^*({w'}^{\alpha'}dw')\right|^2e^{-\varphi}c(-\psi)\\
				=&\int_{w'\in U_y}\left(\int_{\{\psi_1<-t\}}\left|\sum_{\alpha'\in\mathbb{N}^{n-1}}\tilde{F}_{\alpha'}{w'}^{\alpha'}\right|^2e^{-\varphi_1}c(-\psi_1)\right)e^{-\varphi_2(w')}|dw'|^2\\
				\geq&\int_{w'\in U_y}\left(\int_t^{+\infty}c(s)e^{-s}ds\right)\cdot\left(\sum_{j=1}^m\frac{2\pi |a_j|^2e^{-2u(z_j)}}{p_j|d_j|^2}\right)|h(w')|^2e^{-\varphi_2(w')}|dw'|^2\\
				=&\int_{\{\psi_1<-t\}\times U_y}|F|^2e^{-\varphi}c(-\psi).
			\end{split}
		\end{flalign*}
		According to the arbitrariness of $y$ and $U_y$, we get that
		\begin{equation*}
			\int_{\{\psi<-t\}}|\tilde{F}|^2e^{-\varphi}c(-\psi)\geq\int_{\{\psi<-t\}}|F|^2e^{-\varphi}c(-\psi),
		\end{equation*}
		which implies that $F\equiv\tilde{F}$ on $\{\psi<-t\}$ by the choice of $\tilde{F}$.
		
		Now we get that
		\begin{equation}
			G(t)=\int_{\{\psi<-t\}}|F|^2e^{-\varphi}c(-\psi)=\left(\int_t^{+\infty}c(s)e^{-s}ds\right)\left(\sum_{j=1}^m\frac{2\pi |a_j|^2e^{-2u(z_j)}}{p_j|d_j|^2}\right)\int_Y|f_Y|^2e^{-\varphi_2}
		\end{equation}
		for any $t\geq 0$. Then $G(h^{-1}(r))$ is linear with respect to $r\in (0,\int_0^{+\infty}c(s)e^{-s}ds]$, and
		\[F=c_0\pi_1^*\left(gP_*\left(f_u\left(\prod\limits_{l=1}^mf_{z_l}\right)\left(\sum\limits_{l=1}^mp_l\dfrac{d{f_{z_{l}}}}{f_{z_{l}}}\right)\right)\right)\wedge\pi_2^*(f_{Y})\]
		is the unique holomorphic $(n,0)$ form $F$ on $M$ such that $(F-f,(z_j,y))\in (\mathcal{O}(K_M))_{(z_j,y)}\otimes\mathcal{I}(\varphi+\psi)_{(z_j,y)}$ for any $(z_j,y)\in Z_0$ and
		\begin{equation*}
			G(t)=\int_{\{\psi<-t\}}|F|^2e^{-\varphi}c(-\psi)=\left(\int_t^{+\infty}c(s)e^{-s}ds\right)\left(\sum_{j=1}^m\frac{2\pi |a_j|^2e^{-2u(z_j)}}{p_j|d_j|^2}\right)\int_Y|f_Y|^2e^{-\varphi_2}
		\end{equation*}
		for any $t\geq 0$. Then the proof of the sufficiency in Theorem \ref{finite-p} and the proof of Remark \ref{rem-finite} are done.
		
	\end{proof}
	
	\subsection{Proof of the necessity in Theorem \ref{finite-p}}
	\
	we give the proof of the necessity in Theorem \ref{finite-p}.
	
	\begin{proof}[Proof of the necessity in Theorem \ref{finite-p}]
		Assume that $G(h^{-1}(r))$ is linear with respect to $r\in (0,\int_0^{+\infty}c(s)e^{-s}ds]$. Then according to Lemma \ref{linear}, there exists a unique holomorphic $(n,0)$ form $F$ on $M$ satisfying $(F-f)\in H^0(Z_0,(\mathcal{O}(K_M)\otimes\mathcal{I}(\varphi+\psi))|_{Z_0})$, and $G(t;c)=\int_{\{\psi<-t\}}|F|^2e^{-\varphi}c(-\psi)$ for any $t\geq 0$. Then according to Lemma \ref{linear}, Remark \ref{ctildec} and Lemma \ref{tildecincreasing}, we can assume that $c$ is increasing near $+\infty$.
		
		Following from Lemma \ref{decomp}, for any $j\in\{1,2,\ldots,m\}$, we assume that
		\begin{equation*}
			F=\sum_{l=\tilde{k}_j}^{\infty}\pi_1^*(w_j^ldw_j)\wedge \pi_2^*(F_{j,l})
		\end{equation*}
		on $V_{z_j}\times Y$, where $\tilde{k}_j\in\mathbb{N}$, $F_{j,l}$ is a holomorphic $(n-1,0)$ form on $Y$ for any $l\geq \tilde{k}_j$, and $\tilde{F}_j:=F_{j,\tilde{k}_j}\not\equiv 0$.
		
		Using the Weierstrass Theorem on open Riemann surfaces (see \cite{OF81}) and
		the Siu’s Decomposition Theorem, we have
		\[\varphi_1+\psi_1=2\log |g_0|+2u,\]
		where $g_0$ is a holomorphic function on $\Omega$ and $u$ is a subharmonic function on $\Omega$ such that $v(dd^cu,z)\in [0,1)$ for any $z\in\Omega$. Denote that $k_j:=ord_{z_j}(g_0)$, and $d_j:=\lim_{z\rightarrow z_j}(g_0/w_j^{k_j})(z)$.
		
		We prove that $\psi_1=2\sum_{j=1}^mp_jG_{\Omega}(\cdot,z_j)$. As $G(h^{-1}(r))$ is linear and $c$ is increasing near $+\infty$, using Lemma \ref{fanxiangineq}, we can get that $\tilde{k}_j-k_j+1=0$ for any $j\in I_F$ and 
		\begin{equation}\label{geqfinite}
			\frac{\int_M|F|^2e^{-\varphi}c(-\psi)}{\int_0^{+\infty}c(s)e^{-s}ds}=\liminf_{t\rightarrow+\infty}\frac{\int_{\{\psi<-t}|F|^2e^{-\varphi}c(-\psi)}{\int_t^{+\infty}c(s)e^{-s}ds}\geq \sum_{j\in I_F}\frac{2\pi e^{-2u_(z_j)}}{p_j|d_j|^2}\int_Y|\tilde{F}_j|^2e^{-\varphi_2},
		\end{equation}
		where $I_F:=\{j:\tilde{k}_j-k_j+1\leq 0\}$. Especially, $\int_Y|\tilde{F}_j|^2e^{-\varphi_2}<+\infty$ and $u_0(z_j)>-\infty$ for any $j\in I_F$. Then according to Lemma \ref{e-varphic-psi}, and Lemma \ref{k>k0}, we have
		\begin{equation*}
			(F-\pi_1^*(w_j^{\tilde{k}_j}dw_j)\wedge\pi_2^*(\tilde{F}_j),(z_j,y))\in(\mathcal{O}(K_M)\otimes\mathcal{I}(\varphi+\psi))_{(z_j,y)}
		\end{equation*}
		for any $j\in I_F$, $y\in Y$. And according to Lemma \ref{k>k0} we also have
		\begin{equation*}
			(F,(z_j,y))\in(\mathcal{O}(K_M)\otimes\mathcal{I}(\varphi+\psi))_{(z_j,y)}
		\end{equation*}
		for any $j\in\{1,2,\ldots,m\}\setminus I_F$, $y\in Y$. Denote that $\tilde{\psi}_1:=2\sum_{j=1}^mp_jG_{\Omega}(\cdot,z_j)$, $\tilde{\varphi}_1:=\varphi_1+\psi_1-\tilde{\psi}_1$, $\tilde{\psi}:=\pi_1^*(\tilde{\psi}_1)$, and $\tilde{\varphi}:=\pi_1^*(\tilde{\varphi}_1)+\pi_2^*(\varphi_2)$. Then according to Lemma \ref{L2ext-finite-f}, there exists a holomorphic $(n,0)$ form $\tilde{F}$ on $M$ such that $(\tilde{F}-\pi_1^*(w_j^{\tilde{k}_j}dw_j)\wedge\pi_2^*(\tilde{F}_j),(z_j,y))\in(\mathcal{O}(K_M)\otimes\mathcal{I}(\tilde{\varphi}+\tilde{\psi}))_{(z_j,y)}$ for any $j\in I_F$, $y\in Y$, $(\tilde{F},(z_j,y))\in(\mathcal{O}(K_M)\otimes\mathcal{I}(\tilde{\varphi}+\tilde{\psi}))_{(z_j,y)}$ for any $j\in\{1,2,\ldots,m\}\setminus I_F$, $y\in Y$, and 
		\begin{equation}\label{leqfinite}
			\int_M|\tilde{F}|^2e^{-\tilde{\varphi}}c(-\tilde{\psi})\leq \left({\int_0^{+\infty}c(s)e^{-s}ds}\right)\sum_{j\in I_F}\frac{2\pi e^{-2u_(z_j)}}{p_j|d_j|^2}\int_Y|\tilde{F}_j|^2e^{-\varphi_2}
		\end{equation}
		(Here in Lemma \ref{L2ext-finite-f} letting $\tilde{F}_j\equiv 0$ for $j\notin I_F$). Then $(\tilde{F}-F,(z_j,y))\in(\mathcal{O}(K_M)\otimes\mathcal{I}(\varphi+\psi))_{(z_j,y)}$ for any $(z_j,y)\in Z_0$. Combining inequality (\ref{geqfinite}) with inequality (\ref{leqfinite}), we have that
		\begin{equation*}
			\int_M|\tilde{F}|^2e^{-\varphi}c(-\psi)\geq \int_M|F|^2e^{-\varphi}c(-\psi)\geq \int_M|\tilde{F}|^2e^{-\tilde{\varphi}}c(-\tilde{\psi}).
		\end{equation*}
		According to Lemma \ref{l:psi=G}, we get that $\psi_1=2\sum_{j=1}^mp_jG_{\Omega}(\cdot,z_j)$.
		
		It follows from Lemma \ref{l:G-compact} that there exists $s_0>0$ such that $\{\psi_1<-s_0\}\subset\subset V_0=\bigcup_{j=1}^mV_{z_j}$, and $U_j:=V_{z_j}\cap\{\psi_1<-s_0\}$ is conformally equivalent to the unit disc. Denote that
		\begin{flalign*}
			\begin{split}
				G_j(t):=\inf\bigg\{\int_{(\{\psi_1<-t\}\cap U_j)\times Y}|\tilde{f}|^2e^{-\varphi}c(-\psi) : (\tilde{f}-f,(z_j,y))\in (\mathcal{O}(K_M)\otimes\mathcal{I}(\varphi+\psi))_{(z_j,y)}&\\
				\& \ \tilde{f}\in H^0((\{\psi_1<-t\}\cap U_j)\times Y,\mathcal{O}(K_M))\bigg\},&
			\end{split}
		\end{flalign*}	
		where $t>s_0$ and $j\in\{1,2,\ldots,m\}$. Theorem \ref{Concave} shows that $G_j(h^{-1}(r))$ is concave respect to $r\in (0,\int_{s_0}^{+\infty}c(s)e^{-s}ds]$ for any $j\in\{1,2,\ldots,m\}$, where $h(t)=\int_t^{+\infty}c(s)e^{-s}ds$. Since
		\begin{equation}\label{G=sumG_j}
			G(h^{-1}(r))=\sum_{j=1}^mG_j(h^{-1}(r))
		\end{equation}
		and $G(h^{-1}(r))$ is linear with respect to $r\in (0,\int_{s_0}^{+\infty}c(s)e^{-s}ds]$, then $G_j(h^{-1}(r))$ is linear with respect to $(0,\int_{s_0}^{+\infty}c(s)e^{-s}ds]$ and
		\begin{equation}\label{Gj-F}
			G_j(h^{-1}(r))=\int_{(\{\psi_1<-t\}\cap U_j)\times Y}|F|^2e^{-\varphi}c(-\psi)
		\end{equation}
		for any $j\in \{1,2,\ldots,m\}$. As $G(0)\in (0,+\infty)$, according to equality (\ref{G=sumG_j}) and Lemma \ref{G(t)=0} there exists $j_0\in\{1,2,\ldots,m\}$ such that $G_{j_0}(s_0)\neq 0$. Then it follows from Theorem \ref{one-p}, Remark \ref{rem-p} and equality (\ref{Gj-F}) that there exist a holomorphic $(1,0)$ form $f_0$ on $U_{j_0}$ and a holomorphic $(n-1,0)$ form $F_Y$ on $Y$ such that $\int_Y|F_Y|^2e^{-\varphi_2}\in(0,+\infty)$ and
		\begin{equation*}
			F=\pi_1^*(f_0)\wedge\pi_2^*(F_Y)
		\end{equation*}
		on $U_{j_0}\times Y$. By Lemma \ref{decom-product}, there exists a holomorphic $(1,0)$ form $F_0$ on $\Omega$ such that $F_0=f_0$ on $U_{j_0}$ and
		\begin{equation}\label{FF0FY}
			F=\pi_1^*(F_0)\wedge\pi_2^*(F_Y)
		\end{equation}
		on $M$.
		
		For any $t\geq 0$, denote that
		\begin{flalign*}
			\begin{split}
				G_{\Omega}(t):=\inf\bigg\{\int_{\{\psi_1<-t\}}|\tilde{f}|^2e^{-\varphi_1}c(-\psi_1) : (\tilde{f}-F_0,z_j)\in (\mathcal{O}(K_{\Omega}))_{z_j}\otimes\mathcal{I}(\varphi_1+\psi_1)_{z_j}&\\
				\text{for \ any\ } j\in\{1,2,\ldots,m\}\& \ \tilde{f}\in H^0(\{\psi_1<-t\},\mathcal{O}(K_{\Omega}))\bigg\}.&
			\end{split}
		\end{flalign*}
		Then it follows from Lemma \ref{F_t} that there exists a unique holomorphic $(1,0)$ form $\tilde{f}_t$ on $\{\psi_1<-t\}$ such that $(\tilde{f}_t-F_0,z_j)\in (\mathcal{O}(K_{\Omega}))_{z_j}\otimes\mathcal{I}(\varphi_1+\psi_1)_{z_j}$ for any $j\in\{1,2,\ldots,m\}$ and
		\begin{equation*}
			G_{\Omega}(t)=\int_{\{\psi_1<-t\}}|\tilde{f}_t|^2e^{-\varphi_1}c(-\psi_1).
		\end{equation*}
		According to the definition of $G_{\Omega}(t)$, we have
		\begin{equation}\label{tildefgeqF0}
			\int_{\{\psi_1<-t\}}|\tilde{f}_t|^2e^{-\varphi}c(-\psi)=G_{\Omega}(t)\leq\int_{\{\psi_1<-t\}}|F_0|^2e^{-\varphi}c(-\psi).
		\end{equation}
		Denote that $\tilde{F}_t:=\pi_1^*(\tilde{f}_t)\wedge\pi_2^*(F_Y)$. Lemma \ref{f1zf2w} shows that
		$(\tilde{F}_t-F,(z_j,y))\in (\mathcal{O}(K_M))_{(z_j,y)}\otimes\mathcal{I}(\varphi+\psi)_{(z_j,y)}$
		for any $(z_j,y)\in Z_0$, which implies that
		\begin{equation*}
			\int_{\{\psi<-t\}}|\tilde{F}_t|^2e^{-\varphi}c(-\psi)\geq\int_{\{\psi<-t\}}|F|^2e^{-\varphi}c(-\psi).
		\end{equation*}
		However, it follows from equality (\ref{FF0FY}), inequality (\ref{tildefgeqF0}) and the Fubini's Theorem that
		\begin{flalign*}
			\begin{split}
				&\int_{\{\psi<-t\}}|\tilde{F}_t|^2e^{-\varphi}c(-\psi)\\
				=&\int_{\{\psi_1<-t\}}|\tilde{f}_t|^2e^{-\varphi_1}c(-\psi_1)\cdot\int_Y|F_Y|^2e^{-\varphi_2}\\
				\leq&\int_{\{\psi_1<-t\}}|F_0|^2e^{-\varphi_1}c(-\psi_1)\cdot\int_Y|F_Y|^2e^{-\varphi_2}\\
				=&\int_{\{\psi<-t\}}|F|^2e^{-\varphi}c(-\psi).
			\end{split}
		\end{flalign*}
		Thus we have
		\begin{equation*}
			\int_{\{\psi<-t\}}|\tilde{F}_t|^2e^{-\varphi}c(-\psi)=\int_{\{\psi<-t\}}|F|^2e^{-\varphi}c(-\psi),
		\end{equation*}
		then $F\equiv \tilde{F}_t$ on $\{\psi<-t\}$, which implies that $F_0\equiv\tilde{f}_t$ on $\{\psi_1<-t\}$. Then we have
		\begin{flalign*}
			\begin{split}
				G(t)&=\int_{\{\psi<-t\}}|F|^2e^{-\varphi}c(-\psi)\\
				&=\int_{\{\psi_1<-t\}}|F_0|^2e^{-\varphi_1}c(-\psi_1)\cdot\int_Y|F_Y|^2e^{-\varphi_2}\\
				&=G_{\Omega}(t)\cdot\int_Y|F_Y|^2e^{-\varphi_2}
			\end{split}
		\end{flalign*}
		for any $t\geq 0$, which implies that $G_{\Omega}(h^{-1}(r))$ is linear with respect to $r\in (0,\int_0^{+\infty}c(s)e^{-s}ds]$. According to Theorem \ref{thm:m-points} and Remark \ref{rem-finite-1d}, we have
		
		(1). $\varphi_1+\psi_1=2\log|g|+2\sum_{j=1}^mG_{\Omega}(\cdot,z_j)+2u$, where $g$ is a holomorphic function on $\Omega$ such that $ord_{z_j}(g)=ord_{z_j}(f_1)$ for any $j\in\{1,2,...,m\}$ and $u$ is a harmonic function on $\Omega$, here $F_0=f_1dw_j$ on $V_{z_j}$;
		
		(2). $\prod_{j=1}^m\chi_{z_j}=\chi_{-u}$, where $\chi_{-u}$ and $\chi_{z_j}$ are the characters associated to the functions $-u$ and $G_{\Omega}(\cdot,z_j)$ respectively;
		
		(3).
		\[F_0=c_0gP_*\left(f_u(\prod_{j=1}^mf_{z_j})(\sum_{j=1}^mp_{j}\frac{d{f_{z_{j}}}}{f_{z_{j}}})\right)\]
		where $c_0\in\mathbb{C}\backslash\{0\}$ is a constant, $f_u$ is a holomorphic function on $\Delta$ such that $|f_u|=P^*(e^u)$, and $f_{z_j}$ is a holomorphic function on $\Delta$ such that $|f_{z_j}|=P^*(e^{G_{\Omega}(\cdot,z_j)})$ for any $j\in\{1,2,\ldots,m\}$;
		
		(4). \begin{equation*}
			G_{\Omega}(t)=\int_{\{\psi<-t\}}|F_0|^2e^{-\varphi_1}c(-\psi_1)=\left(\int_t^{+\infty}c(s)e^{-s}ds\right)\sum_{j=1}^m\frac{2\pi |a_j|^2e^{-2u(z_j)}}{p_j|d_j|^2}
		\end{equation*}
		for any $t\geq 0$, where $a_j:=\lim\limits_{z\rightarrow z_j}\frac{F_0}{w^{k_j}dw_j}(z)$
		and $d_j:=\lim\limits_{z\rightarrow z_j}\frac{g}{w^{k_j}}(z)$, here $k_j:=ord_{z_j}(g)$.
		
		Then we have
		\begin{equation*}
			c_0=\lim_{z\rightarrow z_j}\frac{a_jw_j^{k_j}dw_j}{gP_*\left(f_u\left(\prod\limits_{l=1}^mf_{z_l}\right)\left(\sum\limits_{l=1}^mp_l\dfrac{d{f_{z_{l}}}}{f_{z_{l}}}\right)\right)},
		\end{equation*}
		and $a_j\neq 0$. Thus the proofs of statements (1), (3), (4) and (5) in Theorem \ref{finite-p} are done. Besides, we have
		\begin{equation*}
			(F_0-a_jw_j^{k_j}dw_j)\in(\mathcal{O}(K_{\Omega})\otimes\mathcal{I}(\varphi_1+\psi_1))_{z_j}
		\end{equation*}
		for any $j\in\{1,2,\ldots,m\}$, which implies that
		\begin{equation*}
			(F-\pi_1^*(a_jw_j^{k_j}dw_j)\wedge\pi_2^*(F_Y))\in(\mathcal{O}(K_M)\otimes\mathcal{I}(\varphi+\psi))_{(z_j,y)}
		\end{equation*}
		for any $(z_j,y)\in Z_0$ according to $\int_Y|F_Y|^2e^{-\varphi_2}<+\infty$ and Lemma \ref{f1zf2w}. Thus we have $f=\pi_1^*(a_jw_j^{k_j}dw_j)\wedge \pi_2^*(f_{Y})+f_j$ on $V_{z_j}\times Y$ for any $j\in\{1,2,\ldots,m\}$, where $(f_j,(z_j,y))\in (\mathcal{O}(K_M))_{(z_j,y)}\otimes\mathcal{I}(\varphi+\psi)_{(z_j,y)}$ for any $j\in\{1,2,\ldots,m\}$ and $y\in Y$. Then the proof of statement (2) in Theorem \ref{finite-p} is done, and the proof of necessity in Theorem \ref{finite-p} is completed.
	\end{proof}
	
	\section{Proof of Proposition \ref{infinite-p}}
	
	In this section, we give the proof of Proposition \ref{infinite-p}.
	
	\begin{proof}[Proof of Proposition \ref{infinite-p}]
		Assume that $G(h^{-1}(r))$ is linear with respect to $r\in (0,\int_0^{+\infty}c(s)e^{-s}ds]$. Then according to Lemma \ref{linear}, there exists a unique holomorphic $(n,0)$ form $F$ on $M$ satisfying $(F-f)\in H^0(Z_0,(\mathcal{O}(K_M)\otimes\mathcal{I}(\varphi+\psi))|_{Z_0})$, and $G(t;c)=\int_{\{\psi<-t\}}|F|^2e^{-\varphi}c(-\psi)$ for any $t\geq 0$. Then according to Lemma \ref{linear}, Remark \ref{ctildec} and Lemma \ref{tildecincreasing}, we can assume that $c$ is increasing near $+\infty$.
		
		Following from Lemma \ref{decomp}, for any $j\in\{1,2,\ldots,m\}$, we assume that
		\begin{equation*}
			F=\sum_{l=\tilde{k}_j}^{\infty}\pi_1^*(w_j^ldw_j)\wedge \pi_2^*(F_{j,l})
		\end{equation*}
		on $V_{z_j}\times Y$, where $\tilde{k}_j\in\mathbb{N}$, $F_{j,l}$ is a holomorphic $(n-1,0)$ form on $Y$ for any $l\geq \tilde{k}_j$, and $\tilde{F}_j:=F_{j,\tilde{k}_j}\not\equiv 0$.
		
		Using the Weierstrass Theorem on open Riemann surfaces (see \cite{OF81}) and
		the Siu’s Decomposition Theorem, we have
		\[\varphi_1+\psi_1=2\log |g_0|+2u,\]
		where $g_0$ is a holomorphic function on $\Omega$ and $u$ is a subharmonic function on $\Omega$ such that $v(dd^cu,z)\in [0,1)$ for any $z\in\Omega$. Denote that $k_j:=ord_{z_j}(g_0)$, and $d_j:=\lim_{z\rightarrow z_j}(g_0/w_j^{k_j})(z)$.
		
		We prove that $\psi_1=2\sum_{j=1}^{\infty}p_jG_{\Omega}(\cdot,z_j)$. As $G(h^{-1}(r))$ is linear and $c$ is increasing near $+\infty$, using Lemma \ref{fanxiangineq}, we can get that $\tilde{k}_j-k_j+1=0$ for any $j\in I_F$ and 
		\begin{equation}\label{geqinfinite}
			\frac{\int_M|F|^2e^{-\varphi}c(-\psi)}{\int_0^{+\infty}c(s)e^{-s}ds}=\liminf_{t\rightarrow+\infty}\frac{\int_{\{\psi<-t}|F|^2e^{-\varphi}c(-\psi)}{\int_t^{+\infty}c(s)e^{-s}ds}\geq \sum_{j\in I_F}\frac{2\pi e^{-2u_(z_j)}}{p_j|d_j|^2}\int_Y|\tilde{F}_j|^2e^{-\varphi_2},
	\end{equation}
where $I_F:=\{j\in\mathbb{N}_+:\tilde{k}_j- k_j+1\leq 0\}$. Especially, $\int_Y|\tilde{F}_j|^2e^{-\varphi_2}<+\infty$ and $u_0(z_j)>-\infty$ for any $j\in I_F$. Then according to Lemma \ref{e-varphic-psi}, and Lemma \ref{k>k0}, we have
\begin{equation*}
(F-\pi_1^*(w_j^{\tilde{k}_j}dw_j)\wedge\pi_2^*(\tilde{F}_j),(z_j,y))\in(\mathcal{O}(K_M)\otimes\mathcal{I}(\varphi+\psi))_{(z_j,y)}
\end{equation*}
for any $j\in I_F$, $y\in Y$. And according to Lemma \ref{k>k0} we also have
\begin{equation*}
(F,(z_j,y))\in(\mathcal{O}(K_M)\otimes\mathcal{I}(\varphi+\psi))_{(z_j,y)}
\end{equation*}
for any $j\in\mathbb{N}_+\setminus I_F$, $y\in Y$. Denote that $\tilde{\psi}_1:=2\sum_{j=1}^{\infty}p_jG_{\Omega}(\cdot,z_j)$, $\tilde{\varphi}_1:=\varphi_1+\psi_1-\tilde{\psi}_1$, $\tilde{\psi}:=\pi_1^*(\tilde{\psi}_1)$, and $\tilde{\varphi}:=\pi_1^*(\tilde{\varphi}_1)+\pi_2^*(\varphi_2)$. Then according to Lemma \ref{L2ext-finite-f}, there exists a holomorphic $(n,0)$ form $\tilde{F}$ on $M$ such that $(\tilde{F}-\pi_1^*(w_j^{\tilde{k}_j}dw_j)\wedge\pi_2^*(\tilde{F}_j),(z_j,y))\in(\mathcal{O}(K_M)\otimes\mathcal{I}(\tilde{\varphi}+\tilde{\psi}))_{(z_j,y)}$ for any $j\in I_F$, $y\in Y$, $(\tilde{F},(z_j,y))\in(\mathcal{O}(K_M)\otimes\mathcal{I}(\tilde{\varphi}+\tilde{\psi}))_{(z_j,y)}$ for any $j\in\mathbb{N}_+\setminus I_F$, $y\in Y$, and 
\begin{equation}\label{leqinfinite}
\int_M|\tilde{F}|^2e^{-\tilde{\varphi}}c(-\tilde{\psi})\leq \left({\int_0^{+\infty}c(s)e^{-s}ds}\right)\sum_{j\in I_F}\frac{2\pi e^{-2u_(z_j)}}{p_j|d_j|^2}\int_Y|\tilde{F}_j|^2e^{-\varphi_2}
\end{equation}
(Here in Lemma \ref{L2ext-finite-f} letting $\tilde{F}_j\equiv 0$ for $j\notin I_F$). Then $(\tilde{F}-F,(z_j,y))\in(\mathcal{O}(K_M)\otimes\mathcal{I}(\tilde{\varphi}+\tilde{\psi}))_{(z_j,y)}$ for any $(z_j,y)\in Z_0$. Combining inequality (\ref{geqinfinite}) with inequality (\ref{leqinfinite}), we have that
\begin{equation*}
\int_M|\tilde{F}|^2e^{-\varphi}c(-\psi)\geq \int_M|F|^2e^{-\varphi}c(-\psi)\geq \int_M|\tilde{F}|^2e^{-\tilde{\varphi}}c(-\tilde{\psi}).
\end{equation*}
According to Lemma \ref{l:psi=G}, we get that $\psi_1=2\sum_{j=1}^{\infty}p_jG_{\Omega}(\cdot,z_j)$.

		For $p\in Z_0^1$ (without loss of generality, we can assume $p=z_1$), there exists $s_1>0$ such that $\{\psi_1<-s_1\}\cap \partial U_1=\emptyset$, where $U_1:=\{|w_1(z)|<r_1 : z\in V_{z_1}\}\subset\subset V_{z_1}$ is a neighborhood of $z_1$ in $\Omega$. For $l\in\{1,2\}$, denote that
		\begin{flalign*}
			\begin{split}
				G_l(t):=\inf\bigg\{\int_{(\{\psi_1<-t\}\cap D_l)\times Y}|\tilde{f}|^2e^{-\varphi}c(-\psi) : (\tilde{f}-f,(z_j,y))\in (\mathcal{O}(K_M)\otimes\mathcal{I}(\varphi+\psi))_{(z_j,y)}&\\
				\text{for \ any \ } z_j\in D_l \ \& \ \tilde{f}\in H^0((\{\psi_1<-t\}\cap D_l)\times Y,\mathcal{O}(K_M))\bigg\},&
			\end{split}
		\end{flalign*}
		where $t\geq s_1$, $D_1:=\{\psi_1<-s_1\}\cap U_1$, and $D_2:=\{\psi_1<-s_1\}\setminus \overline{U_1}$. Theorem \ref{Concave} shows that $G_1(h^{-1}(r))$ and $G_2(h^{-1}(r))$ are concave with respect to $r\in(0,\int_{s_1}^{+\infty}c(s)e^{-s}ds]$. As $G(t)=G_1(t)+G_2(t)$ for $t\geq s_1$ and $G(h^{-1}(r))$ is linear with respect to $r$, we have that $G_1(h^{-1}(r))$ and $G_2(h^{-1}(r))$ are linear with respect to $r\in(0,\int_{s_1}^{+\infty}c(s)e^{-s}ds]$, and
		\begin{equation}
			G_l(h^{-1}(r))=\int_{(\{\psi_1<-t\}\cap D_l)\times Y}|F|^2e^{-\varphi}c(-\psi)
		\end{equation}
		for $l\in\{1,2\}$. If $G_1(t)=0$ for some $t\geq s_1$, we have $F\equiv 0$, which contradicts to $G(0)\neq 0$. Thus $G_1(t)\neq 0$ for any $t\geq s_1$. Note that $\{j\in\mathbb{N}_+ : z_j\in D_1\}=\{z_1\}$, then it follows from Theorem \ref{one-p} that there exist a holomorphic $(1,0)$ form $f_0$ on $D_1$ and a holomorphic $(n-1,0)$ form $F_Y$ on $Y$ such that $\int_Y|F_Y|^2e^{-\varphi_2}\in(0,+\infty)$ and
		\begin{equation*}
			F=\pi_1^*(f_0)\wedge\pi_2^*(F_Y)
		\end{equation*}
		on $D_1\times Y$. By Lemma \ref{decom-product}, there exists a holomorphic $(1,0)$ form $F_0$ on $\Omega$ such that $F_0=f_0$ on $D_1$ and
		\begin{equation}\label{FF0FY-infinite}
			F=\pi_1^*(F_0)\wedge\pi_2^*(F_Y)
		\end{equation}
		on $M$.
		
		For any $t\geq 0$, denote that
		\begin{flalign*}
			\begin{split}
				G_{\Omega}(t):=\inf\bigg\{\int_{\{\psi_1<-t\}}|\tilde{f}|^2e^{-\varphi_1}c(-\psi_1) : (\tilde{f}-F_0,z_j)\in (\mathcal{O}(K_{\Omega}))_{z_j}\otimes\mathcal{I}(\varphi_1+\psi_1)_{z_j}&\\
				\text{for \ any\ } j\in\mathbb{N}_+\& \ \tilde{f}\in H^0(\{\psi_1<-t\},\mathcal{O}(K_{\Omega}))\bigg\}.&
			\end{split}
		\end{flalign*}
		Then it follows from Lemma \ref{F_t} that there exists a unique holomorphic $(1,0)$ form $\tilde{f}_t$ on $\{\psi_1<-t\}$ such that $(\tilde{f}_t-F_0,z_j)\in (\mathcal{O}(K_{\Omega}))_{z_j}\otimes\mathcal{I}(\varphi_1+\psi_1)_{z_j}$ for any $j\in\mathbb{N}_+$ and
		\begin{equation*}
			G_{\Omega}(t)=\int_{\{\psi_1<-t\}}|\tilde{f}_t|^2e^{-\varphi_1}c(-\psi_1).
		\end{equation*}
		According to the definition of $G_{\Omega}(t)$, we have
		\begin{equation}\label{tildefgeqF0-infinite}
			\int_{\{\psi_1<-t\}}|\tilde{f}_t|^2e^{-\varphi}c(-\psi)=G_{\Omega}(t)\leq\int_{\{\psi_1<-t\}}|F_0|^2e^{-\varphi}c(-\psi).
		\end{equation}
		Denote that $\tilde{F}_t:=\pi_1^*(\tilde{f}_t)\wedge\pi_2^*(F_Y)$. Lemma \ref{f1zf2w} shows that
		$(\tilde{F}_t-F,(z_j,y))\in (\mathcal{O}(K_M))_{(z_j,y)}\otimes\mathcal{I}(\varphi+\psi)_{(z_j,y)}$
		for any $(z_j,y)\in Z_0$, which implies that
		\begin{equation*}
			\int_{\{\psi<-t\}}|\tilde{F}_t|^2e^{-\varphi}c(-\psi)\geq\int_{\{\psi<-t\}}|F|^2e^{-\varphi}c(-\psi).
		\end{equation*}
		However, it follows from equality (\ref{FF0FY-infinite}), inequality (\ref{tildefgeqF0-infinite}) and the Fubini's Theorem that
		\begin{flalign*}
			\begin{split}
				&\int_{\{\psi<-t\}}|\tilde{F}_t|^2e^{-\varphi}c(-\psi)\\
				=&\int_{\{\psi_1<-t\}}|\tilde{f}_t|^2e^{-\varphi_1}c(-\psi_1)\cdot\int_Y|F_Y|^2e^{-\varphi_2}\\
				\leq&\int_{\{\psi_1<-t\}}|F_0|^2e^{-\varphi_1}c(-\psi_1)\cdot\int_Y|F_Y|^2e^{-\varphi_2}\\
				=&\int_{\{\psi<-t\}}|F|^2e^{-\varphi}c(-\psi).
			\end{split}
		\end{flalign*}
		Thus we have
		\begin{equation*}
			\int_{\{\psi<-t\}}|\tilde{F}_t|^2e^{-\varphi}c(-\psi)=\int_{\{\psi<-t\}}|F|^2e^{-\varphi}c(-\psi),
		\end{equation*}
		then $F\equiv \tilde{F}_t$ on $\{\psi<-t\}$, which implies that $F_0\equiv\tilde{f}_t$ on $\{\psi_1<-t\}$. Then we have
		\begin{flalign*}
			\begin{split}
				G(t)&=\int_{\{\psi<-t\}}|F|^2e^{-\varphi}c(-\psi)\\
				&=\int_{\{\psi_1<-t\}}|F_0|^2e^{-\varphi_1}c(-\psi_1)\cdot\int_Y|F_Y|^2e^{-\varphi_2}\\
				&=G_{\Omega}(t)\cdot\int_Y|F_Y|^2e^{-\varphi_2}
			\end{split}
		\end{flalign*}
		for any $t\geq 0$, which implies that $G_{\Omega}(h^{-1}(r))$ is linear with respect to $r\in (0,\int_0^{+\infty}c(s)e^{-s}ds]$. According to Proposition \ref{p:infinite}, we have
		
		(1). $\varphi_1+\psi_1=2\log |g|$, where $g$ is a holomorphic function on $\Omega$ such that $ord_{z_j}(g)=ord_{z_j}(f_1)+1$ for any $j\in\mathbb{N}_+$, here $F_0=f_1dw_j$ on $V_{z_j}$;
		
		(2). for any $j\in\mathbb{N}_+$,
		\begin{equation}
			\frac{p_j}{ord_{z_j}g}\lim_{z\rightarrow z_j}\frac{dg}{F_0}=c_0,
		\end{equation}
		where $c_0\in\mathbb{C}\setminus\{0\}$ is a constant independent of $j$;
		
		(3). $\sum\limits_{j\in\mathbb{N}_+}p_j<+\infty$.
		
		Denote that $a_j:=\lim\limits_{z\rightarrow z_j}\frac{F_0}{w^{k_j}dw_j}(z)$, where $k_j:=ord_{z_j}(f_1)$. Then we have that
		\begin{equation*}
			(F_0-a_jw_j^{k_j}dw_j)\in(\mathcal{O}(K_{\Omega})\otimes\mathcal{I}(\varphi_1+\psi_1))_{z_j}
		\end{equation*}
		and
		\begin{equation}
			\frac{p_j}{ord_{z_j}g}\lim_{z\rightarrow z_j}\frac{dg}{a_jw_j^{k_j}dw_j}=c_0,
		\end{equation}
		for any $j\in\mathbb{N}_+$. Then we obtain that
		\begin{equation*}
			(F-\pi_1^*(a_jw_j^{k_j}dw_j)\wedge\pi_2^*(F_Y))\in(\mathcal{O}(K_M)\otimes\mathcal{I}(\varphi+\psi))_{(z_j,y)}
		\end{equation*}
		for any $(z_j,y)\in Z_0$ according to $\int_Y|F_Y|^2e^{-\varphi_2}<+\infty$ and Lemma \ref{f1zf2w}. Thus we have $f=\pi_1^*(a_jw_j^{k_j}dw_j)\wedge \pi_2^*(f_{Y})+f_j$ on $V_{z_j}\times Y$ for any $j\in\mathbb{N}_+$, where $(f_j,(z_j,y))\in (\mathcal{O}(K_M))_{(z_j,y)}\otimes\mathcal{I}(\varphi+\psi)_{(z_j,y)}$ for any $j\in\mathbb{N}_+$ and $y\in Y$. Then the proof of Proposition \ref{infinite-p} is done.
	
	\end{proof}

	\section{Proof of Proposition \ref{tildeM}}
	In this section, we prove Proposition \ref{tildeM}.
	\begin{proof}[Proof of Proposition \ref{tildeM}]
	Assume that $\tilde{G}(h^{-1}(r))$ is linear with respect to $r\in (0,\int_0^{+\infty}c(s)e^{-s}ds]$. Then according to Lemma \ref{linear}, there exists a unique holomorphic $(n,0)$ form $F$ on $\tilde{M}$ satisfying that $(F-f)\in H^0(Z_0,(\mathcal{O}(K_{\tilde{M}})\otimes\mathcal{I}(\varphi+\psi))|_{Z_0})$, and $\tilde{G}(t;c)=\int_{\{\psi<-t\}}|F|^2e^{-\varphi}c(-\psi)$ for any $t\geq 0$. Then according to Lemma \ref{linear}, Remark \ref{ctildec} and Lemma \ref{tildecincreasing}, we can assume that $c$ is increasing near $+\infty$.
	
	Using the Weierstrass Theorem on open Riemann surfaces (see \cite{OF81}) and
	the Siu’s Decomposition Theorem, we have
	\[\varphi_1+\psi_1=2\log |g_0|+2u_0,\]
	where $g_0$ is a holomorphic function on $\Omega$ and $u_0$ is a subharmonic function on $\Omega$ such that $v(dd^cu_0,z)\in [0,1)$ for any $z\in\Omega$.
	
	Assume that $F=\sum_{l=\tilde{k}_j}^{\infty}\pi_1^*(w_j^ldw_j)\wedge \pi_2^*(F_{j,l})$ on $U$ according to Lemma \ref{decomp-tildeM}, where $U\subset \tilde{M}$ is a neighborhood of $Z_0$, $\tilde{k_j}$ is a nonnegative integer, $F_{j,l}$ is a holomorphic $(n-1,0)$ on $Y$ for any $j,l$, and $\tilde{F}_j:=F_{j,\tilde{k}_j}\not\equiv 0$. Assume that $g_0=d_jw_j^{k_j}h_j$ on $V_{z_j}$, where $d_j$ is a constant, $k_j$ is a nonnegative integer, and $h_j$ is a holomorphic function on $V_{z_j}$ such that $h_j(z_j)=1$ for any $j$ ($1\leq j<\gamma$).
	
	Let $W$ be an open subset of $Y$ such that $W\subset\subset Y$. Then for any $j$, $1\leq j<\gamma$, there exists $r_{j,W}>0$ such that $U_{j,W}\times W\subset\tilde{M}$, where $U_{j,W}:=\{z\in\Omega : |w_j(z)|<r_{j,W}\}$.
	
	According to Lemma \ref{fanxiangineq}, we have that $\tilde{k}_j+1-k_j=0$ for any $j\in I_F$ and
	\begin{flalign}
		\begin{split}
			&\frac{\tilde{G}(0)}{\int_0^{+\infty}c(s)e^{-s}ds}\\
			=&\lim_{t\rightarrow+\infty}\frac{\tilde{G}(t)}{\int_t^{+\infty}c(s)e^{-s}ds}\\
			\geq&\sum_{j\in I_F}\liminf_{t\rightarrow+\infty}\frac{\int_{(\{\psi_1<-t\}\cap U_{j,W})\times W}|F|^2e^{-\varphi}c(-\psi)}{\int_t^{+\infty}c(s)e^{-s}ds}\\
			\geq&\sum_{j\in I_F}\frac{2\pi e^{-2u_0(z_j)}}{p_j|d_j|^2}\int_W|\tilde{F}_j|^2e^{-\varphi_2},
		\end{split}
	\end{flalign}
	where $I_F:=\{j: 1\leq j<\gamma \ \&\ \tilde{k}_j+1-k_j\leq 0\}$. Besides, we can know that $u_0(z_j)>-\infty$. According to the arbitrariness of $W$, we have $\int_Y|\tilde{F}_j|^2e^{-\varphi_2}<+\infty$ for any $j\in I_F$, and
	\begin{equation}\label{tildeM-1}
		\frac{\tilde{G}(0)}{\int_0^{+\infty}c(s)e^{-s}ds}\geq \sum_{j\in I_F}\frac{2\pi e^{-2u_0(z_j)}}{p_j|d_j|^2}\int_Y|\tilde{F}_j|^2e^{-\varphi_2}.
	\end{equation}

	 Then according to Lemma \ref{e-varphic-psi}, Lemma \ref{f1zf2w}, Lemma \ref{k>k0} and $\tilde{k}_j+1-k_j=0$, $\int_Y|F_j|^2e^{-\varphi_2}<+\infty$ for $j\in I_F$, we know that $(F-\pi_1^*(w_j^{\tilde{k}_j}dw_j)\wedge \pi_2^*(\tilde{F}_j),(z_j,y))\in(\mathcal{O}(K_{\tilde{M}})\otimes\mathcal{I}(\varphi+\psi))_{(z_j,y)}$ for any $j\in I_F$ and $y\in Y$. According to Lemma \ref{k>k0}, we also have $(F,(z_j,y))\in(\mathcal{O}(K_{\tilde{M}})\otimes\mathcal{I}(\varphi+\psi))_{(z_j,y)}$ for any $j\notin I_F$ and $y\in Y$.
	
	By Lemma \ref{L2ext-finite-f}, there exists a holomorphic $(n,0)$ form $\tilde{F}$ on $M$, such that $(\tilde{F}-\pi_1^*(w_j^{\tilde{k}_j}dw_j)\wedge \pi_2^*(\tilde{F}_j),(z_j,y))\in(\mathcal{O}(K_M)\otimes\mathcal{I}(\varphi+\psi))_{(z_j,y)}$ for any $j\in I_F$, $y\in Y$, $(\tilde{F},(z_j,y))\in(\mathcal{O}(K_M)\otimes\mathcal{I}(\varphi+\psi))_{(z_j,y)}$ for any $j\notin I_F$, $y\in Y$ and
	\begin{equation*}
		\int_M|\tilde{F}|^2e^{-\varphi}c(-\psi)\leq \left(\int_0^{+\infty}c(s)e^{-s}ds\right)\sum_{j\in I_F}\frac{2\pi e^{-2u_0(z_j)}}{p_j|d_j|^2}\int_Y|\tilde{F}_j|^2e^{-\varphi_2}.	
	\end{equation*}
	Then $(\tilde{F}-F,(z_j,y))\in(\mathcal{O}(K_M)\otimes\mathcal{I}(\varphi+\psi))_{(z_j,y)}$ for any $j$ ($1\leq j<\gamma$) and $y\in Y$.

	Now According to the choice of $F$, we have that
	\begin{flalign}\label{tildeM-2}
		\begin{split}
			\tilde{G}(0)=&\int_{\tilde{M}}|F|^2e^{-\varphi}c(-\psi)\leq \int_{\tilde{M}}|\tilde{F}|^2e^{-\varphi}c(-\psi)\\
			\leq&\int_M|\tilde{F}|^2e^{-\varphi}c(-\psi)\\
			\leq&\left(\int_0^{+\infty}c(s)e^{-s}ds\right)\sum_{j\in I_F}\frac{2\pi e^{-2u_0(z_j)}}{p_j|d_j|^2}\int_Y|\tilde{F}_j|^2e^{-\varphi_2}.	
		\end{split}
	\end{flalign}	
	 Combining inequality (\ref{tildeM-1}) with inequality (\ref{tildeM-2}), we get that
	\begin{equation*}
		\int_{\tilde{M}}|\tilde{F}|^2e^{-\varphi}c(-\psi)=\int_M|\tilde{F}|^2e^{-\varphi}c(-\psi).
	\end{equation*}
	As $\tilde{F}\not\equiv 0$, the equality above implies that $\tilde{M}=M$.
	\end{proof}

	\section{Proofs of Theorem \ref{fib-L2ext} and Remark \ref{rem:fib-L2ext}}
	
	In this section, we give the proofs of Theorem \ref{fib-L2ext} and Remark \ref{rem:fib-L2ext}.
	
	\begin{proof}[Proof of Theorem \ref{fib-L2ext}]
		 Using the Weierstrass Theorem on open Riemann surfaces (see \cite{OF81}) and
		the Siu’s Decomposition Theorem, we have
		\[\varphi_1+\psi_1=2\log |g_0|+2u_0,\]
		where $g_0$ is a holomorphic function on $\Omega$ and $u_0$ is a subharmonic function on $\Omega$ such that $v(dd^cu_0,z)\in [0,1)$ for any $z\in\Omega$. Note that $ord_{z_j}g_0=k_j+1$ and
		\begin{equation*}
			e^{2u_0(z_j)}\lim_{z\rightarrow z_j}\left|\frac{g_0}{w_j^{k_j+1}}(z)\right|^2=e^{\alpha_j}c_{\beta}(z_j)^{2(k_j+1)}.
		\end{equation*}
		Using Lemma \ref{L2ext-finite-f}, we obtain that there exists a holomorphic $(n,0)$ form $F$ on $M$ such that $(F-f,(z_j,y))\in(\mathcal{O}(K_M)\otimes\mathcal{I}(\varphi+\psi))_{(z_j,y)}$ for any $(z_j,y)\in Z_0$ and
		\begin{equation*}
			\int_M|F|^2e^{-\varphi}c(-\psi)\leq\left(\int_0^{+\infty}c(s)e^{-s}ds\right)\sum_{j=1}^m\frac{2\pi|a_j|^2e^{-\alpha_j}}{p_jc_{\beta}(z_j)^{2(k_j+1)}}\int_Y|F_j|^2e^{-\varphi_2}.
		\end{equation*}
		
		Thus $(F-f,(z_j,y))\in(\mathcal{O}(K_{\tilde{M}})\otimes\mathcal{I}(\varphi+\psi))_{(z_j,y)}$ for any $(z_j,y)\in Z_0$ and
		\begin{equation*}
			\int_{\tilde{M}}|F|^2e^{-\varphi}c(-\psi)\leq\left(\int_0^{+\infty}c(s)e^{-s}ds\right)\sum_{j=1}^m\frac{2\pi|a_j|^2e^{-\alpha_j}}{p_jc_{\beta}(z_j)^{2(k_j+1)}}\int_Y|F_j|^2e^{-\varphi_2}.
		\end{equation*}
		
		In the following, we prove the characterization of the holding of the equality $\left(\int_0^{+\infty}c(s)e^{-s}ds\right)\sum\limits_{j=1}^m\frac{2\pi|a_j|^2e^{-\alpha_j}}{p_jc_{\beta}(z_j)^{2(k_j+1)}}\int_Y|F_j|^2e^{-\varphi_2}=\inf\big\{ \int_{\tilde{M}}|\tilde{F}|^2e^{-\varphi}c(-\psi):\tilde{F}$ is a holomorphic $(n,0)$ form on $\tilde{M}$ such that $(\tilde{F}-f,(z_j,y))\in(\mathcal{O}(K_{\tilde{M}})\otimes\mathcal{I}(\varphi+\psi))_{(z_j,y)}$ for any $(z_j,y)\in Z_0\big\}$.
		
		According to the above discussions (replacing $\psi_1$ by $\psi_1+t$, replacing $c(\cdot)$ by $c(\cdot+t)$ and replacing $\Omega$ by $\{\psi_1<-t\}$), for any $t\geq 0$, there exists a holomorphic $(n,0)$ form $F_t$ on $\{\psi<-t\}$ such that $(F_t-f,(z_j,y))\in(\mathcal{O}(K_M)\otimes\mathcal{I}(\varphi+\psi))_{(z_j,y)}$ for any $(z_j,y)\in Z_0$ and
		\begin{equation}\label{8.2}
			\int_{\{\psi<-t\}}|F_t|^2e^{-\varphi}c(-\psi)\leq\left(\int_t^{+\infty}c(s)e^{-s}ds\right)\sum_{j=1}^m\frac{2\pi|a_j|^2e^{-\alpha_j}}{p_jc_{\beta}(z_j)^{2(k_j+1)}}\int_Y|F_j|^2e^{-\varphi_2}.
		\end{equation}	
		
		Firstly, we prove the necessity. According to the above discussions, there exists a holomorphic $(n,0)$ form $\tilde{F}_1\neq 0$ on $M$ such that $(\tilde{F}_1-f,(z_j,y))\in(\mathcal{O}(K_M)\otimes\mathcal{I}(\varphi+\psi))_{(z_j,y)}$ for any $(z_j,y)\in Z_0$ and
		\begin{flalign}
			\begin{split}
				&\int_M|\tilde{F}_1|^2e^{-\pi_1^*(\varphi_1+\psi_1-2\sum_{j=1}^mp_jG_{\Omega}(\cdot,z_j))-\pi_2^*(\varphi_2)}c\left(\pi_1^*(-2\sum_{j=1}^mp_jG_{\Omega}(\cdot,z_j))\right)\\
				\leq&\left(\int_0^{+\infty}c(s)e^{-s}ds\right)\sum_{j=1}^m\frac{2\pi|a_j|^2e^{-\alpha_j}}{p_jc_{\beta}(z_j)^{2(k_j+1)}}\int_Y|F_j|^2e^{-\varphi_2}.
			\end{split}
		\end{flalign}	
		
		As $c(t)e^{-t}$ is decreasing on $(0,+\infty)$, $\psi_1\leq 2\sum_{j=1}^mp_jG_{\Omega}(\cdot,z_j)$ and $\left(\int_0^{+\infty}c(s)e^{-s}ds\right)$ $\cdot\sum\limits_{j=1}^m\frac{2\pi|a_j|^2e^{-\alpha_j}}{p_jc_{\beta}(z_j)^{2(k_j+1)}}\int_Y|F_j|^2e^{-\varphi_2}=\inf\big\{ \int_{\tilde{M}}|\tilde{F}|^2e^{-\varphi}c(-\psi):\tilde{F}$ is a holomorphic $(n,0)$ form on $\tilde{M}$ such that $(\tilde{F}-f,(z_j,y))\in(\mathcal{O}(K_{\tilde{M}})\otimes\mathcal{I}(\varphi+\psi))_{(z_j,y)}$ for any $(z_j,y)\in Z_0\big\}$, we have
		\begin{flalign*}
			\begin{split}
				&\int_{\tilde{M}}|\tilde{F}_1|^2e^{-\varphi}c(-\psi)\\
				=&\int_M|\tilde{F}_1|^2e^{-\pi_1^*(\varphi_1+\psi_1-2\sum_{j=1}^mp_jG_{\Omega}(\cdot,z_j))-\pi_2^*(\varphi_2)}c\left(\pi_1^*(-2\sum_{j=1}^mp_jG_{\Omega}(\cdot,z_j))\right).
			\end{split}
		\end{flalign*}
		As $\frac{1}{2}v(dd^c\psi_1,z_j)=p_j>0$, $c(t)e^{-t}$ is decreasing and $\int_0^{+\infty}c(s)e^{-s}ds<+\infty$, we have $\tilde{M}=M$, and $\psi_1=2\sum_{j=1}^mp_jG_{\Omega}(\cdot,z_j)$ according to Lemma \ref{l:psi=G}.
		
		As $e^{-\varphi_1}c(-\psi_1)=e^{-\varphi_1-\psi_1}e^{\psi_1}c(-\psi_1)$, $c(t)e^{-t}$ is decreasing on $(0,+\infty)$, and $\varphi_2$ is a plurisubharmonic function on $Y$, $e^{-\varphi}c(-\psi)$ has locally positive lower bound on $M\setminus Z_0$. Thus $G(h^{-1}(r))$ is concave with respect to $r$. According to the definition of $G(t)$ and inequality (\ref{8.2}), we have
		\begin{equation}
			\frac{G(t)}{\int_t^{+\infty}c(s)e^{-s}ds}\leq\sum_{j=1}^m\frac{2\pi|a_j|^2e^{-\alpha_j}}{p_jc_{\beta}(z_j)^{2(k_j+1)}}\int_Y|F_j|^2e^{-\varphi_2}=\frac{G(0)}{\int_0^{+\infty}c(s)e^{-s}ds}
		\end{equation}
		for any $t\geq 0$. Then $G(h^{-1}(r))$ is linear with respect to $r\in (0,\int_0^{+\infty}c(s)e^{-s}ds]$. From Theorem \ref{finite-p}, we get that $f=\pi_1^*(a'_jw_j^{k'_j}dw_j)\wedge \pi_2^*(F_{Y})+f_j$ on $V_{z_j}\times Y$, $\varphi_1+\psi_1=2\log |g_1|+2\sum\limits_{j=1}^mG_{\Omega}(\cdot,z_j)+2u_1$, $\prod\limits_{j=1}^m\chi_{z_j}=\chi_{-u_1}$, and
		\begin{equation*}
			\lim_{z\rightarrow z_j}\frac{a'_jw_j^{k'_j}dw_j}{g_1P_*\left(f_{u_1}\left(\prod\limits_{l=1}^mf_{z_l}\right)\left(\sum\limits_{l=1}^mp_l\dfrac{d{f_{z_{l}}}}{f_{z_{l}}}\right)\right)}=c_0
		\end{equation*}
		for any $j\in\{1,2,\ldots,m\}$, where
		$a'_j\in\mathbb{C}\setminus \{0\}$ is a constant, $k'_j$ is a nonnegative integer, $F_{Y}$ is a holomorphic $(n-1,0)$ form on $Y$ such that $\int_{Y}|F_{Y}|^2e^{-\varphi_2}\in (0,+\infty)$, $(f_j,(z_j,y))\in (\mathcal{O}(K_M))_{(z_j,y)}\otimes\mathcal{I}(\varphi+\psi)_{(z_j,y)}$ for any $(z_j,y)\in Z_0$, $g_1$ is a holomorphic function on $\Omega$ such that $ord_{z_j}(g_1)=k'_j$, $u_1$ is a harmonic function on $\Omega$
		and $c_0\in\mathbb{C}\setminus\{0\}$ is a constant independent of $j$. Since $f=\pi_1^*(a_jw_j^{k_j}dw_j)\wedge \pi_2^*(F_j)$ on $V_{z_j}\times Y$ for any $j$, we have $k_j=k'_j$, and $c_jF_j=c_0F_Y$ for any $j$, where \begin{equation}\label{c_j}
			c_j:=\lim_{z\rightarrow z_j}\frac{a_jw_j^{k_j}dw_j}{g_1P_*\left(f_{u_1}\left(\prod\limits_{l=1}^mf_{z_l}\right)\left(\sum\limits_{l=1}^mp_l\dfrac{d{f_{z_{l}}}}{f_{z_{l}}}\right)\right)}\in\mathbb{C}\setminus\{0\}.
		\end{equation}
		Thus $a_j\neq 0$ and $c_j\neq 0$ for any $j\in\{1,2,\ldots,m\}$. As $\Omega$ is an open Riemann surface, then there exists a holomorphic function $g_2$ on $\Omega$ such that $g_2(z)\neq 0$ for any $z\in\Omega\setminus\{z_1,z_2,\ldots,z_m\}$ and $ord_{z_j}(g_2)=k_j$ for any $j$. Denote that $g:=g_1/g_2$ and $u:=u_1+\log |g_1|-\sum_{j=1}^mk_jG_{\Omega}(\cdot,z_j)$. Then $g$ is a holomorphic function on $\Omega$ such that $g(z_j)\neq 0$ for any $j$, $u$ is a harmonic function on $\Omega$,
		\begin{equation*}
			\varphi_1+\psi_1=2\log|g|+2\sum_{j=1}^m(k_j+1)G_{\Omega}(\cdot,z_j)+2u
		\end{equation*}
		and
		\begin{equation*}
			\prod_{j=1}^m\chi_{z_j}^{k_j+1}=\chi_{-u}.
		\end{equation*}
		Note that there exists a holomorphic function
		\begin{equation*}
			f_u:=f_{u_1}\frac{P^*(g_2)}{\prod_{j=1}^mf_{z_j}^{k_j}}
		\end{equation*}
		on $\Delta$ such that $|f_u|=P^*e^u$. Then it follows from equality (\ref{c_j}) that
		\begin{equation*}
			c_j=\lim_{z\rightarrow z_j}\frac{a_jw_j^{k_j}dw_j}{gP_*\left(f_u\left(\prod\limits_{l=1}^mf_{z_l}^{k_l+1}\right)\left(\sum\limits_{l=1}^mp_l\dfrac{d{f_{z_{l}}}}{f_{z_{l}}}\right)\right)}.
		\end{equation*}
		Thus the five statements in Theorem \ref{fib-L2ext} hold.
		
		Secondly, we prove the sufficiency.  Assume that the five statements in Theorem \ref{fib-L2ext}	hold. As $\Omega$ is an open Riemann surface, then there exists a holomorphic function $g_2$ on $\Omega$ such that $g_2(z)\neq 0$ for any $z\in\Omega\setminus\{z_1,z_2,\ldots,z_m\}$ and $ord_{z_j}(g_2)=k_j$ for any $j$. Denote that $g_1:=gg_2$ is a holomorphic function on $\Omega$ such that $ord_{z_j}(g_1)=k_j$ for any $j$, and $u_1:=u+\sum_{j=1}^mk_jG_{\Omega}(\cdot,z_j)-\log|g_2|$ is a harmonic function $\Omega$. It follows from $\varphi_1+\psi_1=2\log|g|+2\sum_{j=1}^m(k_j+1)G_{\Omega}(\cdot,z_j)+2u$ and $\chi_{-u}=\prod_{j=1}^m\chi_{z_j}^{k_j+1}$ that
		\begin{equation*}
			\varphi_1+\psi_1=2\log|g_1|+2\sum_{j=1}^mG_{\Omega}(\cdot,z_j)+2u_1
		\end{equation*}
		and
		\begin{equation*}
			\chi_{-u_1}=\prod_{j=1}^m\chi_{z_j}.
		\end{equation*}
		Note that there exists a holomorphic function
		\begin{equation*}
			f_{u_1}:=f_u\frac{\prod_{j=1}^mf_{z_j}^{k_j}}{P^*(g_2)}
		\end{equation*}
		on $\Delta$ such that $|f_{u_1}|=P^*e^{u_1}$. As
		\begin{equation*}
			\lim_{z\rightarrow z_j}\frac{a_jw_j^{k_j}dw_j}{gP_*\left(f_u\left(\prod\limits_{l=1}^mf_{z_l}^{k_l+1}\right)\left(\sum\limits_{l=1}^mp_l\dfrac{d{f_{z_{l}}}}{f_{z_{l}}}\right)\right)}=c_j\in\mathbb{C}\setminus\{0\},
		\end{equation*}
		Then we have $a_j\neq 0$ and
		\begin{equation}
			\lim_{z\rightarrow z_j}\frac{a_jw_j^{k_j}dw_j}{g_1P_*\left(f_{u_1}\left(\prod\limits_{l=1}^mf_{z_l}\right)\left(\sum\limits_{l=1}^mp_l\dfrac{d{f_{z_{l}}}}{f_{z_{l}}}\right)\right)}=c_j.
		\end{equation}
		Besides, we have $c_jF_j=c_0F_Y$ and $a_jF_j=a_j'F_Y$ for some $a_j'\in\mathbb{C}\setminus\{0\}$. Then Theorem \ref{finite-p} shows that $G(h^{-1}(r))$ is linear with respect to $r$. It follows from Lemma \ref{linear} that there exists a holomorphic $(n,0)$ form $\tilde{F}$ such that $(\tilde{F}-f,(z_j,y))\in(\mathcal{O}(K_M)\otimes\mathcal{I}(\varphi+\psi))_{(z_j,y)}$ for any $(z_j,y)\in Z_0$ and
		\begin{equation*}
			\int_{\{\psi<-t\}}|\tilde{F}|^2e^{-\varphi}c(-\psi)=G(t)
		\end{equation*}
		for any $t\geq 0$. Assume that $\tilde{F}=\sum_{l=\tilde{k}_j}^{\infty}\pi_1^*(w_j^ldw_j)\wedge \tilde{F}_{j,l}$ on $V_{z_j}\times Y$ for any $j$ according to Lemma \ref{decomp}, where $F_{j,l}$ is a holomorphic $(n-1,0)$ form on $Y$ for any $j,l$ and $F_{j,\tilde{k}_j}\not\equiv 0$. Since $(\tilde{F}-f,(z_j,y))\in(\mathcal{O}(K_M)\otimes\mathcal{I}(\varphi+\psi))_{(z_j,y)}$ for any $(z_j,y)\in Z_0$, according to Lemma \ref{local-germ}, we have $\tilde{k}_j=k_j$ and $a_jF_j=\tilde{F}_{j,k_j}$ for any $j\in\{1,2,\ldots,m\}$. Then	Lemma \ref{fanxiangineq} shows that
		\begin{equation*}
			\frac{G(t)}{\int_t^{+\infty}c(s)e^{-s}ds}=\frac{\int_{\{\psi<-t\}}|\tilde{F}|^2e^{-\varphi}c(-\psi)}{\int_t^{+\infty}c(s)e^{-s}ds}\geq\sum_{j=1}^m\frac{2\pi|a_j|^2 e^{-2u_1(z_j)}}{p_j|d_j|^2}\int_Y|F_j|^2e^{-\varphi_2}
		\end{equation*}
		for sufficiently large $t$, where $d_j=\lim_{z\rightarrow z_j}(g_1/w_j^{k_j})(z)$ for any $j$. And we have
		\begin{equation*}
			\frac{e^{-2u_1(z_j)}}{|d_j|^2}=\frac{e^{-\alpha_j}}{c_{\beta}(z_j)^{2(k_j+1)}},
		\end{equation*}
		which implies that
		\begin{equation*}
			\frac{G(t)}{\int_t^{+\infty}c(s)e^{-s}ds}\geq\sum_{j=1}^m\frac{2\pi|a_j|^2 e^{-\alpha_j}}{p_jc_{\beta}(z_j)^{2(k_j+1)}}\int_Y|F_j|^2e^{-\varphi_2}
		\end{equation*}
		for sufficiently large $t$. As $G(h^{-1}(r))$ is linear with respect to $r$, according to inequality (\ref{L2result}), we have
		\begin{equation*}
			G(t)=\left(\int_t^{+\infty}c(s)e^{-s}ds\right)\sum_{j=1}^m\frac{2\pi|a_j|^2 e^{-\alpha_j}}{p_jc_{\beta}(z_j)^{2(k_j+1)}}\int_Y|F_j|^2e^{-\varphi_2}
		\end{equation*}
		for any $t\geq 0$. Especially, for $t=0$, we get that equality $\left(\int_0^{+\infty}c(s)e^{-s}ds\right)$ $\cdot\sum\limits_{j=1}^m\frac{2\pi|a_j|^2e^{-\alpha_j}}{p_jc_{\beta}(z_j)^{2(k_j+1)}}\int_Y|F_j|^2e^{-\varphi_2}=\inf\big\{ \int_M|\tilde{F}|^2e^{-\varphi}c(-\psi):\tilde{F}$ is a holomorphic $(n,0)$ form on $M$ such that $(\tilde{F}-f,(z_j,y))\in(\mathcal{O}(K_M)\otimes\mathcal{I}(\varphi+\psi))_{(z_j,y)}$ for any $(z_j,y)\in Z_0\big\}$ holds.
		
		The proof of Theorem \ref{fib-L2ext} is done.
	\end{proof}
	
	\subsection{Proof of Remark \ref{rem:fib-L2ext}}
	\
	Now we prove Remark \ref{rem:fib-L2ext}.
	
	\begin{proof}[Proof of Remark \ref{rem:fib-L2ext}]
		Follow the notation in the above subsection. As $G(h^{-1}(r))$ is linear with respect to $r$ and
		\begin{equation*}
			G(0)=\left(\int_0^{+\infty}c(s)e^{-s}ds\right)\sum_{j=1}^m\frac{2\pi|a_j|^2e^{-\alpha_j}}{p_jc_{\beta}(z_j)^{2(k_j+1)}}\int_Y|F_j|^2e^{-\varphi_2},
		\end{equation*}
		it follows from Remark \ref{rem-finite} that
		\begin{equation*}
			c_0\pi_1^*\left(g_1P_*\left(f_{u_1}\left(\prod\limits_{l=1}^mf_{z_l}\right)\left(\sum\limits_{l=1}^mp_l\dfrac{d{f_{z_{l}}}}{f_{z_{l}}}\right)\right)\right)\wedge\pi_2^*(F_{Y})
		\end{equation*}
		is the unique holomorphic $(n,0)$ form $F$ on $M$ such that $(F-f,(z_j,y))\in(\mathcal{O}(K_M)\otimes\mathcal{I}(\varphi+\psi))_{(z_j,y)}$ for any $(z_j,y)\in Z_0$ and
		\begin{equation*}
			\int_M|F|^2e^{-\varphi}c(-\psi)\leq\left(\int_0^{+\infty}c(s)e^{-s}ds\right)\sum_{j=1}^m\frac{2\pi|a_j|^2e^{-\alpha_j}}{p_jc_{\beta}(z_j)^{2(k_j+1)}}\int_Y|F_j|^2e^{-\varphi_2}.
		\end{equation*}
		Note that
		\begin{equation*}
			f_u:=f_{u_1}\frac{P^*(g_2)}{\prod_{j=1}^mf_{z_j}^{k_j}}
		\end{equation*}
		on $\Delta$ and $gg_2=g_1$. Then we obtain that
		\begin{equation*}
			c_0\pi_1^*\left(gP_*\left(f_u\left(\prod\limits_{l=1}^mf_{z_l}^{k_l+1}\right)\left(\sum\limits_{l=1}^mp_l\dfrac{d{f_{z_{l}}}}{f_{z_{l}}}\right)\right)\right)\wedge\pi_2^*(F_{Y})
		\end{equation*}
		is the unique holomorphic $(n,0)$ form $F$ on $M$ such that $(F-f,(z_j,y))\in(\mathcal{O}(K_M)\otimes\mathcal{I}(\varphi+\psi))_{(z_j,y)}$ for any $(z_j,y)\in Z_0$ and
		\begin{equation*}
			\int_M|F|^2e^{-\varphi}c(-\psi)\leq\left(\int_0^{+\infty}c(s)e^{-s}ds\right)\sum_{j=1}^m\frac{2\pi|a_j|^2e^{-\alpha_j}}{p_jc_{\beta}(z_j)^{2(k_j+1)}}\int_Y|F_j|^2e^{-\varphi_2}.
		\end{equation*}
		
	\end{proof}
	
	\section{Proofs of Theorem \ref{fib-L2ext-infinite}}
	
	In this section, we give the proofs of Theorem \ref{fib-L2ext-infinite}.
	\begin{proof}[Proof of Theorem \ref{fib-L2ext-infinite}]
		Using the Weierstrass Theorem on open Riemann surfaces (see \cite{OF81}) and
		the Siu’s Decomposition Theorem, we have
		\[\varphi_1+\psi_1=2\log |g_0|+2u_0,\]
		where $g_0$ is a holomorphic function on $\Omega$ and $u_0$ is a subharmonic function on $\Omega$ such that $v(dd^cu_0,z)\in [0,1)$ for any $z\in\Omega$. Note that $ord_{z_j}g_0=k_j+1$ and
		\begin{equation}
			e^{2u_0(z_j)}\lim_{z\rightarrow z_j}\left|\frac{g_0}{w_j^{k_j+1}}(z)\right|^2=e^{\alpha_j}c_{\beta}(z_j)^{2(k_j+1)}.
		\end{equation}
		Using Lemma \ref{L2ext-finite-f}, we obtain that there exists a holomorphic $(n,0)$ form $F$ on $M$ such that $(F-f,(z_j,y))\in(\mathcal{O}(K_M)\otimes\mathcal{I}(\varphi+\psi))_{(z_j,y)}$ for any $(z_j,y)\in Z_0$ and
		\begin{equation*}
			\int_M|F|^2e^{-\varphi}c(-\psi)\leq\left(\int_0^{+\infty}c(s)e^{-s}ds\right)\sum_{j=1}^{\infty}\frac{2\pi|a_j|^2e^{-\alpha_j}}{(k_j+1)c_{\beta}(z_j)^{2(k_j+1)}}\int_Y|F_j|^2e^{-\varphi_2}.
		\end{equation*}
		
		In the following, we prove that the equality $\left(\int_0^{+\infty}c(s)e^{-s}ds\right)\sum\limits_{j=1}^{\infty}\frac{2\pi|a_j|^2e^{-\alpha_j}}{(k_j+1)c_{\beta}(z_j)^{2(k_j+1)}}$ $\cdot\int_Y|F_j|^2e^{-\varphi_2}=\inf\big\{ \int_M|\tilde{F}|^2e^{-\varphi}c(-\psi):\tilde{F}$ is a holomorphic $(n,0)$ form on $M$ such that $(\tilde{F}-f,(z_j,y))\in(\mathcal{O}(K_M)\otimes\mathcal{I}(\varphi+\psi))_{(z_j,y)}$ for any $(z_j,y)\in Z_0\big\}$ can not hold. In the following, we assume that the equality holds and get a contradiction.
		
		According to the above discussions (replacing $\psi_1$ by $\psi_1+t$, replacing $c(\cdot)$ by $c(\cdot+t)$ and replacing $\Omega$ by $\{\psi_1<-t\}$), for any $t\geq 0$, there exists a holomorphic $(n,0)$ form $F_t$ on $\{\psi<-t\}$ such that $(F_t-f,(z_j,y))\in(\mathcal{O}(K_M)\otimes\mathcal{I}(\varphi+\psi))_{(z_j,y)}$ for any $(z_j,y)\in Z_0$ and
		\begin{equation}\label{9.2}
			\int_{\{\psi<-t\}}|F_t|^2e^{-\varphi}c(-\psi)\leq\left(\int_t^{+\infty}c(s)e^{-s}ds\right)\sum_{j=1}^{\infty}\frac{2\pi|a_j|^2e^{-\alpha_j}}{(k_j+1)c_{\beta}(z_j)^{2(k_j+1)}}\int_Y|F_j|^2e^{-\varphi_2}.
		\end{equation}	
		
		According to the above discussions, there exists a holomorphic $(n,0)$ form $\tilde{F}_1\neq 0$ on $M$ such that $(\tilde{F}_1-f,(z_j,y))\in(\mathcal{O}(K_M)\otimes\mathcal{I}(\varphi+\psi))_{(z_j,y)}$ for any $(z_j,y)\in Z_0$ and
		\begin{flalign}
			\begin{split}
				&\int_M|\tilde{F}_1|^2e^{-\pi_1^*(\varphi_1+\psi_1-2\sum_{j=1}^{\infty}{(k_j+1)G_{\Omega}(\cdot,z_j))-\pi_2^*(\varphi_2)}}c\left(\pi_1^*(-2\sum_{j=1}^{\infty}(k_j+1)G_{\Omega}(\cdot,z_j))\right)\\
				\leq&\left(\int_0^{+\infty}c(s)e^{-s}ds\right)\sum_{j=1}^{\infty}\frac{2\pi|a_j|^2e^{-\alpha_j}}{(k_j+1)c_{\beta}(z_j)^{2(k_j+1)}}\int_Y|F_j|^2e^{-\varphi_2}.
			\end{split}
		\end{flalign}	
		
		As $c(t)e^{-t}$ is decreasing on $(0,+\infty)$, $\psi_1\leq 2\sum_{j=1}^{\infty}(k_j+1)G_{\Omega}(\cdot,z_j)$ and $\left(\int_0^{+\infty}c(s)e^{-s}ds\right)$ $\cdot\sum\limits_{j=1}^{\infty}\frac{2\pi|a_j|^2e^{-\alpha_j}}{(k_j+1)c_{\beta}(z_j)^{2(k_j+1)}}\int_Y|F_j|^2e^{-\varphi_2}=\inf\big\{ \int_M|\tilde{F}|^2e^{-\varphi}c(-\psi):\tilde{F}$ is a holomorphic $(n,0)$ form on $M$ such that $(\tilde{F}-f,(z_j,y))\in(\mathcal{O}(K_M)\otimes\mathcal{I}(\varphi+\psi))_{(z_j,y)}$ for any $(z_j,y)\in Z_0\big\}$, we have
		\begin{flalign*}
			\begin{split}
				&\int_M|\tilde{F}_1|^2e^{-\varphi}c(-\psi)\\
				=&\int_M|\tilde{F}_1|^2e^{-\pi_1^*(\varphi_1+\psi_1-2\sum_{j=1}^{\infty}(k_j+1)G_{\Omega}(\cdot,z_j))-\pi_2^*(\varphi_2)}c\left(\pi_1^*(-2\sum_{j=1}^{\infty}(k_j+1)G_{\Omega}(\cdot,z_j))\right).
			\end{split}
		\end{flalign*}
		As $\frac{1}{2}v(dd^c\psi_1,z_j)=k_j+1>0$, $c(t)e^{-t}$ is decreasing and $\int_0^{+\infty}c(s)e^{-s}ds<+\infty$, we have $\psi_1=2\sum_{j=1}^{\infty}(k_j+1)G_{\Omega}(\cdot,z_j)$ according to Lemma \ref{l:psi=G}.
		
		As $e^{-\varphi_1}c(-\psi_1)=e^{-\varphi_1-\psi_1}e^{\psi_1}c(-\psi_1)$, $c(t)e^{-t}$ is decreasing on $(0,+\infty)$, and $\varphi_2$ is a plurisubharmonic function on $Y$, $e^{-\varphi}c(-\psi)$ has locally positive lower bound on $M\setminus Z_0$. Thus $G(h^{-1}(r))$ is concave with respect to $r$. According to the definition of $G(t)$ and inequality (\ref{9.2}), we have
		\begin{equation}
			\frac{G(t)}{\int_t^{+\infty}c(s)e^{-s}ds}\leq\sum_{j=1}^{\infty}\frac{2\pi|a_j|^2e^{-\alpha_j}}{(k_j+1)c_{\beta}(z_j)^{2(k_j+1)}}\int_Y|F_j|^2e^{-\varphi_2}=\frac{G(0)}{\int_0^{+\infty}c(s)e^{-s}ds}
		\end{equation}
		for any $t\geq 0$. Then $G(h^{-1}(r))$ is linear with respect to $r\in (0,\int_0^{+\infty}c(s)e^{-s}ds]$.
		
		Note that $\sum\limits_{j\rightarrow+\infty}(k_j+1)=+\infty$, which contradicts to Proposition \ref{infinite-p}. Thus we have that $\left(\int_0^{+\infty}c(s)e^{-s}ds\right)\sum\limits_{j=1}^{\infty}\frac{2\pi|a_j|^2e^{-\alpha_j}}{(k_j+1)c_{\beta}(z_j)^{2(k_j+1)}}$ $\cdot\int_Y|F_j|^2e^{-\varphi_2}>\inf\big\{ \int_M|\tilde{F}|^2e^{-\varphi}c(-\psi):\tilde{F}$ is a holomorphic $(n,0)$ form on $M$ such that $(\tilde{F}-f,(z_j,y))\in(\mathcal{O}(K_M)\otimes\mathcal{I}(\varphi+\psi))_{(z_j,y)}$ for any $(z_j,y)\in Z_0\big\}$, which implies that there exists a holomorphic $(n,0)$ form $F$ on $M$ such that $(F-f,(z_j,y))\in(\mathcal{O}(K_M)\otimes\mathcal{I}(\varphi+\psi))_{(z_j,y)}$ for any $(z_j,y)\in Z_0$ and
		\begin{equation*}
			\int_M|F|^2e^{-\varphi}c(-\psi)<\left(\int_0^{+\infty}c(s)e^{-s}ds\right)\sum_{j=1}^{\infty}\frac{2\pi|a_j|^2e^{-\alpha_j}}{(k_j+1)c_{\beta}(z_j)^{2(k_j+1)}}\int_Y|F_j|^2e^{-\varphi_2}.
		\end{equation*}	
	\end{proof}

\section{Proofs of Theorem \ref{thm:suita}, Remark \ref{r:suita},  Theorem \ref{thm:extend} and Remark \ref{r:extend}}
In this section, we prove Theorem \ref{thm:suita}, Remark \ref{r:suita},  Theorem \ref{thm:extend} and Remark \ref{r:extend}.

\subsection{Proofs of Theorem \ref{thm:suita} and Remark \ref{r:suita}}
\begin{proof}[Proofs of Theorem \ref{thm:suita} and Remark \ref{r:suita}]

Let $f_1=dw\wedge d\tilde w_1\wedge\ldots\wedge d\tilde{w}_{n-1}$ on $V_{z_0}\times U_{y_0}$, and let $f_2=d\tilde w_1\wedge\ldots\wedge d\tilde{w}_{n-1}$ on $U_{y_0}$. Let $\psi=\pi_1^*(2G_{\Omega}(\cdot,z_0))$. It follows from Lemma \ref{local-germ}, that $(F_1-F_2,(z_0,y))\in\mathcal{I}(\psi)_{(z_0,y)}$ if and only if $F_1(z_0,y)=F_2(z_0,y)$ for any $y\in Y$, where $F_1$ and $F_2$ are holomorphic functions on $V_{z_0}\times Y$. Let $f$ be a holomorphic $(n-1,0)$ form on $Y$ satisfying $\int_Y|f|^2<+\infty$. It follows from Theorem \ref{fib-L2ext} that there exists a holomorphic $(n,0)$ form $F$ on $M$ such that
$F|_{\{z_0\}\times Y}=\pi_1^*(dw)\wedge\pi_2^*(f)$
and
\[\int_M|F|^2\leq\frac{2\pi}{c_{\beta}(z_0)^2}\int_Y|f|^2.\]
Note that
\[B_Y(y_0)=\frac{2^{n-1}}{\inf\left\{\int_Y|f|^2:f\in H^0(Y,\mathcal{O}(K_Y))\,\&\,f(y_0)=f_2(y_0)\right\}}\]
and
\[B_M((z_0,y_0))=\frac{2^{n}}{\inf\left\{\int_M|F|^2:F\in H^0(M,\mathcal{O}(K_M))\,\&\,F((z_0,y_0))=f_1((z_0,y_0))\right\}}.\]
Thus, we have
\[c_{\beta}(z_0)^2B_Y(y_0)\leq\pi B_M((z_0,y_0)).\]

In the following, we prove the characterization of the holding of the equality $c_{\beta}(z_0)^2B_Y(y_0)\leq\pi B_M((z_0,y_0))$.

By the definition of $B_Y$, there exists a holomorphic $(n-1,0)$ form $f_0$ on $Y$ such that $f_0(y_0)=f_2(y_0)$ and
\[B_Y(y_0)=\frac{2^{n-1}}{\int_Y|f_0|^2}>0.\]
It follows from Theorem \ref{fib-L2ext} that there exists a holomorphic $(n,0)$ form $F_0$ on $M$ such that $F_0=\pi_1^*(dw)\wedge\pi_2^*(f_0)$ and
\begin{equation}\label{eq:1226d}
	\int_M|F_0|^2\leq\frac{2\pi}{c_{\beta}(z_0)^2}\int_Y|f_0|^2.
\end{equation}

Firstly, we prove the necessity. Note that \[B_M((z_0,y_0))\ge\frac{2^n}{\int_M|\tilde F|^2},\]
where $\tilde F$ is a holomorphic $(n,0)$ form on $M$ satisfying  $\tilde F=\pi_1^*(dw)\wedge\pi_2^*(f_0)$ on $\{z_0\}\times Y$. According to inequality \eqref{eq:1226d},
\[c_{\beta}(z_0)^2 B_Y(y_0)= \pi B_M((z_0,y_0)),\]
and
\[B_Y(y_0)=\frac{2^{n-1}}{\int_Y|f_0|^2},\]
we obtain that
\[\frac{2\pi}{c_{\beta}(z_0)^2}\int_Y|f_0|^2=\inf\left\{\int_M|\tilde F|^2:\tilde F\in H^0(M,\mathcal{O}(K_M))\ \&\  \tilde F|_{\{z_0\}\times Y}=\pi_1^*(dw)\wedge\pi_2^*(f_0)\right\}.\]
It follows from Theorem \ref{fib-L2ext} that $\chi_{z_0}=1$. $\chi_{z_0}=1$ implies that there exists a holomorphic function $\tilde{f}$ on $\Omega$ such that $|\tilde{f}|=e^{G_{\Omega}(\cdot,z_0)}$, thus $\Omega$ is conformally equivalent to the unit disc less a (possible) closed set of inner capacity zero (see \cite{suita72}, see also \cite{Yamada} and \cite{GZ15}).

Secondly, we prove the sufficiency. As $\Omega$ is conformally equivalent to the unit disc less a (possible) closed set of inner capacity zero, we have $\chi_{z_0}=1$. We prove
\[c_{\beta}(z_0)^2 B_Y(y_0)= \pi B_M((z_0,y_0))\]
by contradiction: if not, there exists a holomorphic $(n,0)$ form $\tilde F_0$ on $M$ such that $\tilde F_0((z_0,y_0))=f_1((z_0,y_0))$ and
\begin{equation}
	\label{eq:1226e}
	\int_M|\tilde F_0|^2<\frac{2\pi}{c_{\beta}(z_0)^2}\int_Y|f_0|^2.
\end{equation}
There exists a holomorphic $(n-1,0)$ form $\tilde f_0$ on $Y$ such that $\tilde F_0=\pi_1^*(dw)\wedge\pi_2^*(\tilde f_0)$ on $\{z_0\}\times Y$. Hence $\tilde f_0(y_0)=f_2(y_0)=f_0(y_0)$, which implies that $\int_Y|\tilde f|^2\geq \int_Y|f_0|^2$. Combining with inequality \ref{eq:1226e}, we have
\[\inf\left\{\int_M|\tilde F|^2:F\in H^0(M,\mathcal{O}(K_M))\ \&\ \tilde F|_{\{z_0\}\times Y}=\pi_1^*(dw)\wedge\pi_2^*(\tilde f_0)\right\}<\frac{2\pi}{c_{\beta}(z_0)^2}\int_Y|\tilde f_0|^2,\]
which contradicts to Theorem \ref{fib-L2ext}. Hence
\[c_{\beta}(z_0)^2 B_Y(y_0)= \pi B_M((z_0,y_0)).\]

Thus, Theorem \ref{thm:suita} holds.

Note that $B_{\tilde{M}}((z_0,y_0))\ge B_M((z_0,y_0))>0$ and $B_{\tilde{M}}((z_0,y_0))= B_M((z_0,y_0))$ if and only if $M=\tilde{M}$, thus Theorem \ref{thm:suita} shows that Remark \ref{r:suita} holds.

\end{proof}

\subsection{Proof of Theorem \ref{thm:extend} and Remark \ref{r:extend}}
\
\begin{proof}[Proofs of Theorem \ref{thm:extend} and Remark \ref{r:extend}]
	
	Let $f_1=dw\wedge d\tilde w_1\wedge\ldots\wedge dw_{n-1}$ on $V_{z_0}\times U_{y_0}$, and let $f_2=d\tilde w_1\wedge\ldots\wedge dw_{n-1}$ on $U_{y_0}$. Let $\psi=\pi_1^*(2G_{\Omega}(\cdot,z_0))$. It follows from Lemma \ref{local-germ}, that $(F_1-F_2,(z_0,y))\in\mathcal{I}(\psi)_{(z_0,y)}$ if and only if $F_1(z_0,y)=F_2(z_0,y)$ for any $y\in Y$, where $F_1$ and $F_2$ are holomorphic functions on $V_{z_0}\times Y$. Let $f$ be a holomorphic $(n-1,0)$ form on $Y$ satisfying $\int_Y|f|^2<+\infty$. It follows from Theorem \ref{fib-L2ext} that there exists a holomorphic $(n,0)$ form $F$ on $M$ such that
	$F|_{\{z_0\}\times Y}=\pi_1^*(dw)\wedge\pi_2^*(f)$
	and
	\[\int_M|F|^2\rho\leq\frac{2\pi\rho(z_0)}{c_{\beta}(z_0)^2}\int_Y|f|^2.\]
	Note that
	\[B_Y(y_0)=\frac{2^{n-1}}{\inf\left\{\int_Y|f|^2:f\in H^0(Y,\mathcal{O}(K_Y))\,\&\,f(y_0)=f_2(y_0)\right\}}\]
	and
	\[B_{M,\rho}((z_0,y_0))=\frac{2^{n}}{\inf\left\{\int_M|F|^2\rho:F\in H^0(M,\mathcal{O}(K_M))\,\&\,F((z_0,y_0))=f_1((z_0,y_0))\right\}}.\]
	Thus, we have
	\[c_{\beta}(z_0)^2B_Y(y_0)\leq\pi\rho(z_0) B_{M,\rho}((z_0,y_0)).\]
	
	In the following, we prove the characterization of the holding of the equality $c_{\beta}(z_0)^2B_Y(y_0)\leq\pi\rho(z_0) B_{M,\rho}((z_0,y_0))$.
	
	By the definition of $B_Y$, there exists a holomorphic $(n-1,0)$ form $f_0$ on $Y$ such that $f_0(y_0)=f_2(y_0)$ and
	\[B_Y(y_0)=\frac{2^{n-1}}{\int_Y|f_0|^2}>0.\]
	It follows from Theorem \ref{fib-L2ext} that there exists a holomorphic $(n,0)$ form $F_0$ on $M$ such that $F_0=\pi_1^*(dw)\wedge\pi_2^*(f_0)$ and
	\begin{equation}\label{eq:1226d-e}
		\int_M|F_0|^2\rho\leq\frac{2\pi\rho(z_0)}{c_{\beta}(z_0)^2}\int_Y|f_0|^2.
	\end{equation}
	
	Firstly, we prove the necessity. Note that \[B_{M,\rho}((z_0,y_0))\ge\frac{2^n}{\int_M|\tilde F|^2\rho},\]
	where $\tilde F$ is a holomorphic $(n,0)$ form on $M$ satisfying  $\tilde F=\pi_1^*(dw)\wedge\pi_2^*(f_0)$ on $\{z_0\}\times Y$. According to inequality \eqref{eq:1226d-e},
	\[c_{\beta}(z_0)^2 B_Y(y_0)= \pi\rho(z_0) B_{M,\rho}((z_0,y_0)),\]
	and
	\[B_Y(y_0)=\frac{2^{n-1}}{\int_Y|f_0|^2},\]
	we obtain that
	\[\frac{2\pi\rho(z_0)}{c_{\beta}(z_0)^2}\int_Y|f_0|^2=\inf\left\{\int_M|\tilde F|^2\rho:\tilde F\in H^0(M,\mathcal{O}(K_M))\ \&\  \tilde F|_{\{z_0\}\times Y}=\pi_1^*(dw)\wedge\pi_2^*(f_0)\right\}.\]
	It follows from Theorem \ref{fib-L2ext} that $\chi_{z_0}=\chi_{-u}$.
	
	Secondly, we prove the sufficiency. We prove
	\[c_{\beta}(z_0)^2B_Y(y_0)\leq\pi\rho(z_0) B_{M,\rho}((z_0,y_0))\]
	by contradiction: if not, there exists a holomorphic $(n,0)$ form $\tilde F_0$ on $M$ such that $\tilde F_0((z_0,y_0))=f_1((z_0,y_0))$ and
	\begin{equation}
		\label{eq:1226e-e}
		\int_M|\tilde F_0|^2\rho<\frac{2\pi\rho(z_0)}{c_{\beta}(z_0)^2}\int_Y|f_0|^2.
	\end{equation}
	There exists a holomorphic $(n-1,0)$ form $\tilde f_0$ on $Y$ such that $\tilde F_0=\pi_1^*(dw)\wedge\pi_2^*(\tilde f_0)$ on $\{z_0\}\times Y$. Hence $\tilde f_0(y_0)=f_2(y_0)=f_0(y_0)$, which implies that $\int_Y|\tilde f|^2\geq \int_Y|f_0|^2$. Combining with inequality \ref{eq:1226e-e}, we have
	\[\inf\left\{\int_M|\tilde F|^2\rho:F\in H^0(M,\mathcal{O}(K_M))\ \&\ \tilde F|_{\{z_0\}\times Y}=\pi_1^*(dw)\wedge\pi_2^*(\tilde f_0)\right\}<\frac{2\pi\rho(z_0)}{c_{\beta}(z_0)^2}\int_Y|\tilde f_0|^2,\]
	which contradicts to Theorem \ref{fib-L2ext}. Hence
	\[c_{\beta}(z_0)^2B_Y(y_0)\leq\pi\rho(z_0) B_{M,\rho}((z_0,y_0)).\]
	
	Thus, Theorem \ref{thm:extend} holds.
	
	Note that $B_{\tilde{M},\rho}((z_0,y_0))\geq B_{M,\rho}((z_0,y_0))>0$ and $B_{\tilde{M},\rho}((z_0,y_0))= B_{M,\rho}((z_0,y_0))$ if and only if $M=\tilde{M}$, thus Theorem \ref{thm:extend} shows that Remark \ref{r:extend} holds.
	
\end{proof}

	
	\vspace{.1in} {\em Acknowledgements}. The authors would like to thank Dr. Zhitong Mi for checking this paper and pointing
	out some mistakes. The second named author was supported by National Key R\&D Program of China 2021YFA1003100, NSFC-11825101, NSFC-11522101 and NSFC-11431013.

\end{document}